\documentclass[12pt,letter]{article}

\usepackage{color}
\definecolor{green}{rgb}{0.834580, 0.893016, 0.718697}

\usepackage{titletoc}
\usepackage{hyperref}
\usepackage[page,header]{appendix}
\usepackage{lipsum}
\usepackage[textfont=it]{caption}
\usepackage{upgreek}

\usepackage{bbm}
\usepackage{natbib}
\usepackage{amstext}
\usepackage{amsthm}
\usepackage{epsfig}
\usepackage{pdfpages}
\usepackage{amsmath,amssymb,amsthm}
\usepackage{graphicx}
\usepackage{fancyhdr}
\usepackage{calc}
\usepackage{dsfont}
\usepackage{subfigure}
\usepackage{rotating}
\usepackage{multirow}
\usepackage{amsmath}
\usepackage{lscape}
\usepackage{tabularx}
\usepackage{bm}
\usepackage{caption}
\usepackage{color}
\usepackage{enumerate,array}
\usepackage{verbatim}
\usepackage{fancyhdr}
\usepackage{stackengine}
\usepackage{enumitem}
\usepackage{scalerel}

\makeatletter
\DeclareRobustCommand\check[1]{{\mathpalette\@widecheck{#1}}}
\def\@widecheck#1#2{%
    \setbox\z@\hbox{\m@th$#1#2$}%
    \setbox\tw@\hbox{\m@th$#1%
       \widehat{%
          \vrule\@width\z@\@height\ht\z@
          \vrule\@height\z@\@width\wd\z@}$}%
    \dp\tw@-\ht\z@
    \@tempdima\ht\z@ \advance\@tempdima2\ht\tw@ \divide\@tempdima\thr@@
    \setbox\tw@\hbox{%
       \raise\@tempdima\hbox{\scalebox{1}[-1]{\lower\@tempdima\box
\tw@}}}%
    {\ooalign{\box\tw@ \cr \box\z@}}}
\makeatother

\usepackage{hyperref}
\definecolor{cornellred}{rgb}{0.7, 0.11, 0.11}
\hypersetup{    
	colorlinks=true,
	linkcolor=blue,
	citecolor=blue,
	urlcolor=black
}

\renewcommand{\baselinestretch}{1.5}

\numberwithin{equation}{section}

\allowdisplaybreaks

\newtheorem{theorem}{Theorem}[section]
\newtheorem{lemma}{Lemma}[section]

\theoremstyle{definition}

\newtheorem{remark}{Remark}[section]

\newtheorem{assumption}{Assumption}

\newcommand{\cov}{{\rm cov}}

\newcommand{\E}{{\rm E}}

\newcommand{\I}{\mathcal{I}}

\renewcommand{\H}{\mathcal{H}}

\newcommand{\as}{{\rm a.s.}}

\newcommand{\one}{\mathbbm{1}}
\def\trans{^{\mkern-1mu\mathsf{T}\mkern-1mu}}

\newcommand\G{\mathbb{G}}

\newcommand\TT{{\mathbb T}}
\newcommand\VV{{\mathbb V}}
\newcommand\WW{{\mathbb W}}
\newcommand\BB{{\mathbb B}}

\newcommand{\var}{{\rm var}}

\newcommand{\e}{\epsilon}

\newcommand{\EE}{\mathcal{E}}

\newcommand{\converged}{\overset{d}{\longrightarrow}}
\newcommand{\weakconverge}{\rightsquigarrow}

\newcommand{\iidsim}{\overset{\small{\text{iid}}}{\sim}}

\renewcommand{\l}{\langle}
\renewcommand{\r}{\rangle}
\renewcommand{\phi}{\varphi}

\renewcommand{\tilde}{\widetilde}
\renewcommand{\hat}{\widehat}
\renewcommand{\epsilon}{\varepsilon}
\renewcommand{\P}{{\rm P}}

\def\boxit#1{\vbox{\hrule\hbox{\vrule\kern6pt  \vbox{\kern6pt#1\kern6pt}\kern6pt\vrule}\hrule}}

\def\boxit#1{\vbox{\hrule\hbox{\vrule\kern6pt
          \vbox{\kern6pt#1\kern6pt}\kern6pt\vrule}\hrule}}

\newtheorem{innercustomgeneric}{\customgenericname}
\providecommand{\customgenericname}{}
\newcommand{\newcustomtheorem}[2]{%
  \newenvironment{#1}[1]
  {%
   \renewcommand\customgenericname{#2}%
   \renewcommand\theinnercustomgeneric{##1}%
   \innercustomgeneric
  }
  {\endinnercustomgeneric}
}
\newcustomtheorem{customthm}{Assumption}
\newcustomtheorem{customcor}{Corollary}

\renewcommand{\baselinestretch}{1.4}

\makeatletter
\def\singlespace{\deltaf\baselinestretch{1}\@normalsize}

  \renewenvironment{thebibliography}[1]{%
    \begin{oldthebibliography}{#1}%
      \setlength{\parskip}{0.3ex}%
      \setlength{\itemsep}{0ex}%
  }%
  {%
    \end{oldthebibliography}%
  }

\begin{document}

\begin{center}
{\bf \Large 

An  RKHS approach for pivotal  inference in functional linear regression 
} \\[1cm]
\end{center}

\baselineskip=13pt
\begin{center}

Holger Dette, Jiajun Tang\\
Fakult\"at f\"ur Mathematik, Ruhr-Universit\"at Bochum, Bochum, Germany
\end{center}

\baselineskip=20pt
\vspace{.5cm}

\noindent {\bf Abstract:} 
We develop methodology for testing hypotheses regarding  the slope function in functional linear regression for time series via a reproducing kernel Hilbert space approach. In contrast to most of the literature, which
considers tests for the exact nullity of the slope function, we are interested in the null hypothesis
that the slope function vanishes only approximately, where deviations are measured with respect to 
the $L^2$-norm.  An asymptotically pivotal  test is proposed, 
which does not require the estimation  of nuisance parameters and long-run covariances. 
 The key technical tools to prove the validity of our approach include a uniform   Bahadur representation and a weak invariance principle  for a sequential process of estimates of the slope function. Both scalar-on-function and function-on-function linear regression are considered and finite-sample methods for implementing our methodology are provided. We also illustrate the potential of our methods by means of  a small simulation study and a data example.\\

\smallskip


\section{Introduction}

Statistical methods for analysing functional data have been extensively developed in the past decades, as reviewed in the  monographs \cite{ramsay2005}, \cite{FerratyVieu2010}, \cite{horvath2012}, \cite{hsingeubank2015} and the survey article by  \cite{wang2016}. Because of its good interpretability, the functional linear regression model
\begin{align}\label{model0}
Y_i=\int_0^1X_i(s)\,\beta_0(s)\,ds+\epsilon_i\,,\qquad i\in\mathbb Z\,,
\end{align}
has become a useful toolbox for functional data analysis and has
gained considerable attention \citep[see, for exemple,][among many others]{cardot1999,
muller2005,yao2005,hall2007,yuancai}. In this paper  $\{(X_i,\epsilon_i)\}_{i\in\mathbb Z}$ denotes a strictly stationary time series, where the $X_i$'s are mean zero square-integrable random functions on the interval $[0,1]$, and the $\e_i$'s are centred random noise. 

As the slope function $\beta_0$ characterizes the dependence between the predictor and the response, many authors have worked on its estimation  and corresponding statistical inference.  
A popular method for analysing the 
slope function in  model \eqref{model0}
is through the functional principle components (FPC) \citep[see, for example,][among many others]{yao2005,hall2007,horvath2012,hilgert2013}.  
Other authors considered  a reproducing kernel Hilbert space (RKHS) approach to develop inference tools 
for $\beta_0$ and  corresponding theoretical results regarding consistency and optimality. 
\cite{yuancai}  and \cite{caiyuan2012} studied  an RKHS estimator  and   its  prediction risk in the  scalar-on-function linear regression model.
\cite{shang2015} proposed an RKHS framework of inference for the generalized functional linear regression and \cite{hao2021} considered the functional Cox model.  These authors additionally suggested tests for the nullity of the slope function 
\citep[see also][for an alternative approach
in
the scalar-on-function linear regression]{qu2017}.
Recently,  \cite{dettetang} used  an RKHS approach to develop  statistical inference methodology  in the function-on-function linear model. A common feature of all these references consists
in the fact that the proposed methodology
depends on the knowledge of  nuisance parameters appearing in the asymptotic variance of the estimators of the slope function. 
These parameters describe the behaviour of the sequence of solutions of  a system of 
estimated integro-differential equations 
induced by the covariance operator of the predictor,
and therefore their 
 estimation  is not an easy problem.
In  the case of independent data 
(as considered in all references using the RKHS approach),  several estimators 
have been proposed  and studied. On the other hand, for time series  data these nuisance parameters  would be  of an even  more complicated structure 
because of the  dependencies in the data, which 
would make its estimation an extremely 
difficult problem.

The purpose of the present paper is to develop  pivotal 
statistical  inference tools for  the slope function $\beta_0$  in the
functional linear regression model \eqref{model0}  using an RKHS approach, which avoids the estimation
of nuisance parameters.
Most of the literature with a focus on testing considers 
 hypotheses of the  form 
\begin{align}\label{ch}
H_0:\,\int_0^1|\beta_0(s)|^2\,ds=0\quad&\text{versus}\quad H_1:\, \int_0^1|\beta_0(s)|^2\,ds\neq0\,,
\end{align}
which is the classical hypothesis 
of the null effect ($\beta_0\equiv0$) of the functional covariate
\citep[see, for example][among many others]{cardot2003,garcia2014,lei2014,kong2016,su2017,Tekbudak}.
Following \cite{berger1987} we  argue that it  is rare, and perhaps impossible, to have a null
hypothesis that can be exactly modeled as $\beta_0\equiv0$. More precisely, in most applications
the covariate $X$ has {\it some} effect on the response $Y$,   and the ``real'' question is, if this effect 
is small and negligible. As an alternative  we will therefore 
consider  the hypotheses
\begin{align}\label{rh}
H_0:\,\int_0^1|\beta_0(s)|^2\,ds\leq\Delta\quad&\text{versus}\quad H_1:\, \int_0^1|\beta_0(s)|^2\,ds>\Delta\,,
\end{align}
for some (small) pre-specified threshold $\Delta > 0$
 that represents the maximal acceptable deviation (measured with respect to the $L^2$ distance) of  $\beta_0$ from the null-function.
Note that in contrast to \eqref{ch} the formulation of the hypotheses in \eqref{rh} is  symmetric, in the sense
that the null and the alternative can be interchanged. 
This allows us
to test at a controlled type I error that the effect of the covariate on the response is negligible, that is
$\int_0^1|\beta_0(s)|^2\,ds\leq\Delta$.
Throughout this paper we will call hypotheses of the form \eqref{ch}  and \eqref{rh} ``classical'' and ``relevant'' hypotheses, respectively.
We refer to \cite{hodges1954,berger1987} for a theoretical discussion  of  relevant hypotheses,
and \cite{chow1992,wellek2010} for applications  in biostatistics, where these references concentrate on 
real valued (or finite dimensional) parameters. In the context of functional data, relevant hypotheses have only found recent   attention in the literature 
\citep[see][among others]{fogarty2014,detteaos2020, dettejrssb2020}.

The aim of this article is the development of pivotal methodology for testing relevant hypotheses \eqref{rh} with no need to estimate nuisance parameters. Our approach is based on a novel self-normalization technique, which has recently been introduced by \cite{dettejrssb2020} in the context of testing relevant hypotheses regarding the mean and
 covariance operator of stationary time series and differs substantially from the  the common  self-normalization
approach proposed 
for testing classical hypotheses regarding finite dimensional parameters 
\citep[see][among many others]{lobato2001,shao2010,shaozhang2010,zhang2011,zhang2015,zhang2018}.
 In Section~\ref{sec:sf} we consider scalar-on-function linear regression
and introduce the reproducing kernel Hilbert space estimator (see Section~\ref{sec:rkhs}).
 Section~\ref{sec:one} is devoted to the development of 
 our self-normalization methodology for the relevant hypotheses \eqref{rh}. As a by-product we also 
 construct (asymptotically)  pivotal
 confidence intervals for the $L^2$-norm of the slope function. Here the crucial
 result is a weak  invariance principle for the process of estimators   $\{ \hat \beta ( \nu ) \}_{\nu \in [\nu_0,1]}$,
 where $\nu_0\in(0,1]$ is a constant and $\hat \beta ( \nu ) $ denotes the estimator of $\beta_0$ calculated from the 
 data $\{(X_i,Y_i)\}_{i= 1, \ldots ,\lfloor n \nu \rfloor }$  (see Theorem \ref{thm:wip} and the discussion in the following paragraph).
 Moreover, we also  consider the problem  of comparing the slope functions from two samples  in Section~\ref{sec:two}.
 In Section~\ref{sec:ff} we extend our methodology to relevant hypotheses for function-on-function linear regression.
Finite sample properties are studied in Section~\ref{sec:finitesample}, where in Section~\ref{sec:implementation} we provide 
details for the numerical implementation of our approach, and simulated data experiments and a real data example is included in Sections~\ref{sec:simulation} and \ref{sec:realdata}, respectively.
In addition, technical details containing the proofs of our theoretical results  and several auxiliary lemmas are included in the online supplementary material. 

To our best knowledge, testing relevant  hypotheses
regarding the slope function has only been considered by  \cite{kutta2021}. Roughly speaking, these authors investigated  a normal equation corresponding to the linear model \eqref{model0}, which is then solved  by an application of a regularized inverse based on
 a spectral-cut-off series estimator.  Although such an approach has some theoretical advantages,
its practical usefulness is limited by the fact that it  requires the estimation   of the spectral decomposition of the regularized inverse. 
In contrast, the estimator considered in this paper is defined as the minimizer of a regularized loss function in an appropriate reproducing kernel Hilbert space.
As a consequence, our approach also provides 
an easy solution of the estimation problem in 
the function-on-function version of the linear model \eqref{model0}.

\section{Scalar-on-function linear regression}\label{sec:sf}

We begin by introducing   some notations which are  used throughout this article. Let $L^2([0,1])$ and $L^2([0,1]^2)$ denote the Hilbert space of square-integrable functions on $[0,1]$ and $[0,1]^2$, respectively, equipped with the usual $L^2$ inner product $\l\cdot,\cdot\r_{L^2}$ and the corresponding $L^2$ norm $\|\cdot\|_{L^2}$. Let $\ell^\infty([0,1])$ denote the set of all bounded real functions on $[0,1]$, and define $\| f \|_\infty := \sup_{t \in [0,1]}
|f (t)| $
as the sup-norm of the function $f$. Let
``$\weakconverge$" denote  weak convergence in $\ell^\infty([0,1])$, and ``$\converged$" denotes the usual  convergence in distribution in $\mathbb{R}^k$ (for some positive integer $k$). 
For $a\in\mathbb R$, let $\lfloor a\rfloor$ denote the largest integer smaller than or equal to $a$.


\subsection{The reproducing kernel Hilbert space approach}\label{sec:rkhs}

Suppose a sample generated by the scalar-on-function linear regression model \eqref{model0} is available and consists of $n$ observations $(X_1,Y_1),\ldots,(X_n,Y_n)$. Let $\nu_0\in(0,1]$ be an arbitrary but fixed parameter. For any $\nu\in[\nu_0,1]$, we first define the RKHS estimator for the slope function $\beta_0$ based on the first $\lfloor n\nu\rfloor$ observations $(X_1,Y_1),\ldots,(X_{\lfloor n\nu\rfloor},Y_{\lfloor n\nu\rfloor})$.  For this purpose, let
\begin{align}\label{H}
\H&=\Big\{\beta:[0,1]\to\mathbb{R}\,|\,\partial^{(\theta)}\beta\text{ is absolutely continuous, for }0\leq\theta\leq m-1\,;\partial^{(m)}\beta\in L^2([0,1])\Big\}
\end{align}
denote the Sobolev space on $[0,1]$ of order $m>1/2$ (see, for example, \citealp{whaba1990}), and 
define for $\nu\in[\nu_0,1]$ 
\begin{align}\label{hatbeta}
\hat\beta_{n,\lambda}(\cdot,\nu)
%
&=\underset{\beta\in \H}{\arg\min}\ \Bigg[ \frac{1}{2\lfloor n\nu\rfloor}\sum_{i=1}^{\lfloor n\nu\rfloor}\left\{Y_i-\int_0^1X_i(s)\,\beta(s)\,ds\right\}^2+\frac{\lambda}{2} J(\beta,\beta)\Bigg]\,.
\end{align}
Here, $\lambda>0$ is a regularization parameter and for $\beta_1,\beta_2\in\H$
\begin{align}\label{jm}
J(\beta_1,\beta_2)=\int_0^1\beta_1^{(m)}(s)\, \beta_2^{(m)}(s)\,ds
\end{align}
defines the penalty functional. In \eqref{hatbeta}, we use the notation  $\hat\beta_{n,\lambda}(\cdot,\nu)$ 
for the estimator of $\beta_0$
to reflect its dependence on the parameters $\lambda$ and $\nu$. We emphasize again that $\hat\beta_{n,\lambda}(\cdot,\nu)$ is the RKHS estimator based on the first $\lfloor n\nu\rfloor$ observations $(X_1,Y_1),\ldots , (X_{\lfloor n\nu\rfloor},Y_{\lfloor n\nu\rfloor})$, and that the parameter $\nu\in[\nu_0,1]$ stands for the (approximate) proportion of the sample $\{(X_i,Y_i)\}_{i=1}^n$ used to obtain $\hat\beta_{n,\lambda}(\cdot,\nu)$. The case where $\nu=1$ corresponds to the scenario where we use the whole sample $\{(X_i,Y_i)\}_{i=1}^n$.

For $\nu\in[\nu_0,1]$, let $L_{n,\lambda,\nu}(\beta)$ denote the objective functional in \eqref{hatbeta}, that is 
\begin{align*}
L_{n,\lambda,\nu}(\beta)=\frac{1}{2\lfloor n\nu\rfloor}\sum_{i=1}^{\lfloor n\nu\rfloor}\left\{Y_i-\int_0^1X_i(s)\,\beta(s)\,ds\right\}^2+\frac{\lambda}{2} J(\beta,\beta)\,,
\end{align*}
and note that the Fr\'echet derivatives of $L_{n,\lambda,\nu}(\beta)$ are given by
\begin{align}\label{dl}
&\mathcal D L_{n,\lambda,\nu}(\beta)\beta_1=-\frac{1}{\lfloor n\nu\rfloor}\sum_{i=1}^{\lfloor n\nu\rfloor}\left\{Y_i-\int_0^1X_i(s_1)\,\beta(s_1)\,ds_1\right\}\int_0^1X_i(s_2)\,\beta_1(s_2)\,ds_2+\lambda J(\beta,\beta_1)\,;\notag\\
&\mathcal D^2 L_{n,\lambda,\nu}(\beta)\beta_1\beta_2=\frac{1}{\lfloor n\nu\rfloor}\sum_{i=1}^{\lfloor n\nu\rfloor}\int_0^1X_i(s_1)\,\beta_1(s)\,ds_1\times\int_0^1X_i(s_2)\,\beta_2(s_2)\,ds_2+\lambda J(\beta_1,\beta_2)\,,
\end{align}
and $\mathcal D^3 L_{n,\lambda,\nu}(\beta)\equiv0$. Let 
\begin{equation}
 \label{det2}   
C_X(s,t)=\cov\{X_1(s),X_1(t)\}
\end{equation}
denote the covariance kernel of the predictor, for $s,t\in[0,1]$. Then, we have
\begin{align*}
\E\{\mathcal D^2 L_{n,\lambda,\nu}(\beta)\beta_1\beta_2\}=\int_0^1\int_0^1C_X(s,t)\,\beta_1(s)\,\beta_2(t)\,ds\,dt+\lambda J(\beta_1,\beta_2)\,.
\end{align*}
This motivates the consideration of the following map $\l\cdot,\cdot\r_K:\H\times\H\to\mathbb{R}$ defined by
\begin{align}\label{inner}
\langle\beta_1,\beta_2\rangle_K=V(\beta_1,\beta_2)+\lambda J(\beta_1,\beta_2)\,,\qquad \beta_1,\beta_2\in\H\,,
\end{align}
where the function $J$ is defined in \eqref{jm} and
\begin{align}\label{v}
&V(\beta_1,\beta_2)=\int_0^1\int_0^1C_X(s,t)\,\beta_1(s)\,\beta_2(t)\,ds\,dt\,.
\end{align}
In order to facilitate our theoretical analysis, we first make the following mild assumption on the covariance function $C_X$.
\begin{assumption}\label{a1}
The covariance kernel $C_X$  in \eqref{det2} is continuous on $[0,1]^2$. For any $\gamma\in L^2([0,1])$, $\int_0^1 C_X(s,t)\gamma(s)ds=0$ for any $t\in[0,1]$ implies that $\gamma\equiv0$.  
\end{assumption}

Under Assumption~\ref{a201}, it is known (see, for example, \citealp{yuancai,shang2015}) that the mapping $\l\cdot,\cdot\r_K$ in \eqref{inner} defines an inner product on $\H$, and we  denote by  $\|\cdot\|_K$   its corresponding norm. In addition, $\H$ is a reproducing kernel Hilbert space (RKHS) equipped with the inner product $\l\cdot,\cdot\r_K$. We follow \cite{shang2015} and assume   that there exists a sequence of functions in $\H$ that diagonalizes the  operators $V$ in  \eqref{v} and $J$ in  \eqref{jm}  simultaneously.

\begin{assumption}[Simultaneous diagonalization]\label{a201}
There exists a sequence of functions $\{\phi_{k}\}_{k\geq1}$ in $\H$, such that $\Vert \phi_{k}\Vert_{\infty}\leq c\, k^{a}$ for any $k\geq 1$, and
\begin{align*} 
V(\phi_{k},\phi_{k'})=\delta_{kk'}\,,\qquad J(\phi_{k},\phi_{k'})=\rho_{k}\,\delta_{kk'}\,,\qquad \text{for any }k,k'\geq1\,,
\end{align*}
where $a\geq 0$, $c>0$ are constants, $\delta_{kk'}$ is the Kronecker delta and  the sequence $\{ \rho_{k} \}_{k\geq1}$ satisfies $\rho_{k}\asymp k^{2D}$ for some constant $D>a+1/2$. Furthermore, any $\beta\in\H$ admits the expansion $\beta=\sum_{k=1}^\infty V(\beta,\phi_{k})\phi_{k}$ with convergence in $\H$ with respect to the norm $\Vert\cdot\Vert_K$.
\end{assumption}

It is shown in \cite{shang2015} that, under suitable conditions, Assumption~\ref{a201} is satisfied if we take $\{(\rho_{k},\phi_{k})\}_{k\geq1}$ to be
the eigenvalue-eigenfunction pairs  of the following integro-differential equations with boundary conditions.
\begin{align}\label{id}
\left\{
\begin{aligned}
&\displaystyle\rho\int_0^1 C_X(s,t)\, x(t)\,dt=(-1)^{m}x^{(2m)}(s)\,,\\
&x^{(\theta)}(0)=x^{(\theta)}(1)=0\,,\qquad \quad\text{ for }m\leq\theta\leq 2m-1\,.
\end{aligned}\right.
\end{align}
For the inner product $\l\cdot,\cdot\r_K$ in \eqref{inner}, under Assumption~\ref{a201}, we have
\begin{align*}
\l\phi_{k},\phi_{k'}\r_K=V(\phi_{k},\phi_{k'})+\lambda J(\phi_{k},\phi_{k'})=(1+\lambda\rho_{k})\,\delta_{kk'}\qquad(k,k'\geq 1)\,.
\end{align*}
Therefore, it follows  that  $\l\beta,\phi_{k}\r_K=\sum_{k'=1}^\infty V(\beta,\phi_{k'})\l\phi_{k},\phi_{k'}\r_K=(1+\lambda\rho_{k})V(\beta,\phi_{k})$  for any $\beta\in\H$,
which implies the representation 
\begin{align}\label{expansion}
\beta
=\sum_{k=1}^\infty\frac{\l\beta,\phi_{k}\r_K}{1+\lambda\rho_{k}}\,\phi_{k}\,.
\end{align}
For any $\beta_1,\beta_2\in\mathcal{H}$ and $J$ defined in \eqref{jm}, let $W_\lambda:\H\to\H$ denote the operator such that $\l W_\lambda(\beta_1),\beta_2\r_K=\lambda J(\beta_1,\beta_2)$. By definition, for the eigenfunctions $\{\phi_{k}\}_{k\geq 1}$ in Assumption~\ref{a201}, we have $\l W_\lambda(\phi_{k}),\phi_{k'}\r_K=\lambda J(\phi_{k},\phi_{k'})=\lambda\rho_{k}\,\delta_{kk'}$, for any $k,k'\geq1$, so that in view of \eqref{expansion},
\begin{align}\label{wphi}
W_\lambda(\phi_{k})=\sum_{k'=1}^\infty\frac{\l W_\lambda(\phi_{k}),\phi_{k'}\r_K}{1+\lambda\rho_{k'}}\phi_{k'}=\frac{\lambda\,\rho_{k}\,\phi_{k}}{1+\lambda\rho_{k}}\,.
\end{align}
In addition, note that $\mathfrak G_{z}(\beta)=\int_0^1\beta(s)z(s)ds$ is a bounded linear functional on $\H$, for any $z\in L^2([0,1])$ and $\beta\in\H$. By the Riesz representation theorem, there exists a unique element $\tau_{\lambda}(z)\in\mathcal{H}$ such that
\begin{align*}
\l\tau_{\lambda}(z),\beta\r_K=\mathfrak G_{z}(\beta)=\int_0^1\beta(s)\,z(s)\,ds\,.
\end{align*}
In particular, $\l\tau_{\lambda}(z),\phi_{k}\r_K=\l z,\phi_{k}\r_{L^2}$, so that in view of \eqref{expansion},
\begin{align}\label{tau}
\tau_{\lambda}(z)=\sum_{k=1}^\infty \frac{\l z,\phi_{k}\r_{L^2}}{1+\lambda\rho_{k}}\,\phi_{k}\,.
\end{align}
Now, for any $\beta,\beta_1,\beta_2\in\H$, in view of \eqref{dl}, define
\begin{equation}\label{sn}
\begin{split}
&S_{n,\lambda,\nu}(\beta)=-\frac{1}{\lfloor n\nu\rfloor}\sum_{i=1}^{\lfloor n\nu\rfloor}\tau_{\lambda}(X_i)\left\{Y_i-\int_0^1X_i(s)\,\beta(s)\,ds\right\}+ W_\lambda(\beta)\,,\\
&\mathcal D S_{n,\lambda,\nu}(\beta)\beta_1=\frac{1}{\lfloor n\nu\rfloor}\sum_{i=1}^{\lfloor n\nu\rfloor}\tau_{\lambda}(X_i)\int_0^1X_i(s)\,\beta_1(s)\,ds+W_\lambda(\beta_1)\,,
\end{split}
\end{equation}
so that $\mathcal D L_{n,\lambda,\nu}(\beta)\beta_1=\l S_{n,\lambda,\nu}(\beta),\beta_1\r_K$ and $\mathcal D^2 L_{n,\lambda,\nu}(\beta)\beta_1\beta_2=\l\mathcal D S_{n,\lambda,\nu}(\beta)\beta_1,\beta_2\r_K$.
Note that, by definition,
\begin{align}\label{snlambdanubeta0}
S_{n,\lambda,\nu}(\beta_0)=-\frac{1}{\lfloor n\nu\rfloor}\sum_{i=1}^{\lfloor n\nu \rfloor}\e_i\,\tau_{\lambda}(X_i)+ W_\lambda(\beta_0)\,.
\end{align}

Recall  the definition of the estimator $\hat\beta_{n,\lambda}(\cdot,\nu)$ 
  defined in \eqref{hatbeta} and consider the statistic
 \begin{align}\label{tn}
\hat{\mathbb{T}}_n=\int_0^1|\hat\beta_{n,\lambda}(s,1)|^2\,ds \,.
\end{align}
It can be shown that, under suitable conditions, the statistic
$\hat{\mathbb{T}}_n$  defines a consistent estimator of
\begin{align}\label{d0}
d_0=\int_0^1|\beta_0(s)|^2\,ds\,,
\end{align}
so that the null hypothesis in  \eqref{rh} should be rejected for large values of $\hat\TT_n$. In fact, it is  a direct consequence of Theorem~\ref{thm:2.1} below, that
\begin{align*}
\sqrt{n}\lambda^{(2a+1)/(2D)}(\hat\TT_n-d_0)\converged N(0,4\sigma_d^2)\,,
\end{align*}
where 
\begin{align}\label{sigmad2}
\sigma_d^2=\lim_{\lambda\downarrow0}\int_0^1\int_0^1C_{U,\lambda}(s,t)\,\beta_0(s)\,\beta_0(t)\,ds\,dt\, , 
\end{align}
the quantity   $C_{U,\lambda}$
is defined by 
\begin{align}\label{cu}
C_{U,\lambda}(s,t)=\lambda^{(2a+1)/D}\sum_{\ell=-\infty}^{+\infty}\cov\big\{\e_0\, \tau_{\lambda}(X_0)(s)\,,\e_\ell \,\tau_{\lambda}(X_\ell)(t)\big\}\,,
\end{align}
and the operator $\tau_\lambda$ is given  in \eqref{tau}.
Unfortunately, in practice, the long-run covariance  $C_{U,\lambda}$ in \eqref{cu} and the asymptotic variance $\sigma_d^2$ in \eqref{sigmad2} is often either intractable or difficult to estimate. This is due to the fact that $\sigma_d^2$ is defined as the limit of a series, which in turns relies on the operator $\tau_\lambda$ in \eqref{tau} and therefore depends on the eigen-system  $\{(\rho_k,\phi_k)\}_{k\geq1}$
of the integro-differential equations in \eqref{id}.
Moreover, $C_{U,\lambda}$ defined in \eqref{cu} depends on the unknown nuisance parameters $a$ and $D$ in Assumption~\ref{a201}, which makes its estimation even more challenging.  These difficulties motivate us to propose a self-normalization approach so that pivotal  tests can be constructed for the relevant hypotheses \eqref{rh} even without the knowledge of $\sigma_d^2$ in \eqref{sigmad2} and the nuisance parameters $a$ and $D$.

\subsection{Self-normalization}\label{sec:one}

In order to establish our self-normalization methodology, we first list below several technical assumptions.

\begin{assumption}[Regularity conditions]\label{a:subg}~

\begin{enumerate}[label={\rm(\ref*{a:subg}.\arabic*)},ref={\rm\ref*{a:subg}.\arabic*},series=a3,nolistsep,leftmargin=1.4cm]
\item\label{a3.1} There exists a constant $\varpi>0$ such that $\E\{\exp(\varpi\|X_0\|_{L^2}^2)\}<\infty$.

\item\label{a3.2} 
For any $\beta\in\H$, $\E\big(\l X_0,\beta\r_{L^2}^4\big)\leq c_0\big\{\E\big(\l X_0,\beta\r_{L^2}^2\big)\big\}^2$, for some constant $c_0>0$.

\item\label{a3.3}The true slope function $\beta_0$ is such that $\sum_{k=1}^\infty \rho_k^2\,V^2(\beta_0,\phi_k)<\infty$.

\item\label{a34} For $(s,t)\in[0,1]^2$ and $C_{U,\lambda}$ in \eqref{cu}, the limit $C_U(s,t)=\lim_{\lambda\downarrow0}C_{U,\lambda}(s,t)$ exists.
\end{enumerate}

\end{assumption}

\begin{assumption}\label{a:rate}
The constants $a$ and $D$ in Assumption~\ref{a201}
and the regularization parameter 
$\lambda$ in \eqref{hatbeta} satisfy  $\lambda=o(1)$, $n^{-1}\lambda^{-(2a+1)/D}=o(1)$, $ n\lambda^{2+(2a+1)/(2D)}=o(1)$ as $n\to\infty$. In addition, $n^{-1}\lambda^{-2\varsigma}\log n=o(1)$ and $\lambda^{-2\varsigma+(2D+2a+1)/(2D)}\log n=o(1)$  as $n\to\infty$, where $\varsigma=(2D-2a-1)/(4Dm)+(a+1)/(2D)>0$.
\end{assumption}

%

\begin{remark}
Assumption~\ref{a3.1} requires an  exponential tail of $\|X_0\|_{L^2}$, which can be satisfied for a variety of processes; see, for example, \cite{shang2015}. Assumption~\ref{a3.2} is a common condition in linear regression models for functional data; see, for example \cite{caiyuan2012} and \cite{shang2015}. Assumption~\ref{a3.3} corresponds to the so-called undersmoothing scenario in \cite{shang2015}; see their Remark 3.2. Assumption~\ref{a:rate} specifies the conditions for the convergence rates for the regularization parameter $\lambda$ in \eqref{hatbeta}.

\end{remark}

In order to characterize the dependence structures of the functional time series, we use  the concept of $m$-approximability (see, for example, \citealp{potscher,hormann2010,berkes2013}). 
\begin{assumption}\label{a:m}~
For $i\in\mathbb Z$, $(X_i,Y_i)$ is generated by the model \eqref{model0} and satisfies the following assumptions.
\begin{enumerate}[label={\rm(\ref*{a:m}.\arabic*)},ref={\rm\ref*{a:m}.\arabic*},series=Qbeta,nolistsep,leftmargin=1.4cm]

\item\label{a5.1} 
$X_i=g(\ldots,\xi_{i-1},\xi_i)$ and $\e_i=h(\ldots,\eta_{i-1},\eta_i)$, for $i\in\mathbb Z$ and some deterministic measurable functions $g:\mathcal S^{\infty}\to L^2([0,1])$ and $h:\mathbb{R}^\infty\to\mathbb{R}$, where $\mathcal S$ is some measurable space and  $\xi_i=\xi_i(t,\upomega)$ is jointly measurable in $(t,\upomega)$. The $\xi_i$'s and the $\eta_i$'s are independent and identically distributed (i.i.d).

\item\label{a5.2} For any $s\in[0,1]$, $\E\{X_0(s)\}=\E(\e_0)=0$. For some $\delta\in(0,1)$, $\E|\e_0|^{2+\delta}<\infty$.


\item The sequences $ \{X_i\}_{i\in\mathbb Z}$ and $\{\e_i\}_{i\in\mathbb Z}$ can be approximated by $\ell$-dependent sequences $\{X_{i,\ell}\}_{i,\ell\in\mathbb Z}$ and $\{\e_{i,\ell}\}_{i,\ell\in\mathbb Z}$, respectively, in the sense that, for some $\kappa>2+\delta$,
\begin{align*} 
\sum_{\ell=1}^\infty\big(\E\|X_i-X_{i,\ell}\|_{L^2}^{2+\delta}\big)^{1/\kappa}<\infty\,,\qquad\sum_{\ell=1}^\infty\big(\E|\e_i-\e_{i,\ell}|^{2+\delta}\big)^{1/\kappa}<\infty\,.
\end{align*}
Here, $X_{i,\ell}=g(\xi_i,\xi_{i-1},\ldots,\xi_{i-\ell+1},\boldsymbol{\xi}_{i,\ell}^*)$ and $\e_{i,\ell}=h(\eta_i,\eta_{i-1},\ldots,\eta_{i-\ell+1},\boldsymbol{\eta}_{i,\ell}^*)$, where $\boldsymbol{\xi}_{i,\ell}^*=(\xi^*_{i,\ell,i-\ell},\xi^*_{i,\ell,i-\ell-1},\ldots)$ and $\boldsymbol{\eta}_{i,\ell}^*=(\eta^*_{i,\ell,i-\ell},\eta^*_{i,\ell,i-\ell-1},\ldots)$, and where the $\xi^*_{i,\ell,k}$'s and the $\eta^*_{i,\ell,k}$'s are independent copies of $\xi_0$ and $\eta_0$, and are independent of $\{\xi_i\}_{i\in\mathbb Z}$ and $\{\eta_i\}_{i\in\mathbb Z}$, respectively.

\end{enumerate}

\end{assumption}

\begin{remark}
Assumption~\ref{a5.1} implies that the error process $\{\e_i\}_{i\in\mathbb Z}$
is independent of the predictor $\{X_i\}_{i\in\mathbb Z}$, which is
a common assumption in the literature (see, for example, Section~6 in \citealp{hormann2010}).
\end{remark}

%
We first establish a uniform Bahadur representation of the slope function. 
Observing the definition of 
$S_{n,\lambda,\nu}$  in \eqref{sn} and for the operator $W_\lambda(\cdot)$   below equation \eqref{expansion}, we expect that
\begin{align*}
\hat\beta_{n,\lambda}(\cdot,\nu)-\beta_0\approx-S_{n,\lambda,\nu}(\beta_0)=\frac{1}{\lfloor n\nu\rfloor}\sum_{i=1}^{\lfloor n\nu \rfloor}\e_i\,\tau_{\lambda}(X_i)- W_\lambda(\beta_0)\,.
\end{align*}
This is justified by the following theorem proved in Section~\ref{app:thm:bahadur} of the online supplementary material.

\begin{theorem}[Uniform Bahadur representation]\label{thm:bahadur}
Suppose Assumptions~\ref{a1}--\ref{a:m} are satisfied. Then, for any fixed (but arbitrary) $\nu_0\in(0,1]$,
\begin{align}\label{bahadur}
\sup_{\nu\in[\nu_0,1]}\,\bigg\Vert\nu\big\{\hat\beta_{n,\lambda}(\cdot,\nu)-\beta_0+W_\lambda(\beta_0)\big\}-\frac{1}{n}\sum_{i=1}^{\lfloor n\nu \rfloor}\e_i\,\tau_{\lambda}(X_i) \bigg\Vert_K=O_p(v_n)\,, 
\end{align}
where for the constant $\varsigma>0$ in Assumption~\ref{a:rate},
\begin{align*} 
v_n = n^{-1/2}\lambda^{-\varsigma}\big(\lambda^{1/2}+n^{-1/2}\lambda^{-(2a+1)/(4D)}\big)(\log n)^{1/2} \,.
\end{align*}
\end{theorem}
We define for  $i\in\mathbb Z$ and $\tau_{\lambda}(\cdot)$ in \eqref{tau} the random variables
\begin{equation}\label{u}
\begin{split}
&U_i=\lambda^{(2a+1)/(2D)}\,\e_i\,\tau_{\lambda}(X_i)=\lambda^{(2a+1)/(2D)}\e_i\sum_{k=1}^\infty\frac{\l X_i,\phi_k\r_{L^2}}{1+\lambda\rho_k}\phi_k\,.
%
\end{split}
\end{equation}
Theorem~\ref{thm:bahadur} shows that, under suitable conditions, the following approximation
\begin{align*}
\nu\big\{\hat\beta_{n,\lambda}(\cdot,\nu)-\beta_0+W_\lambda(\beta_0)\big\}\approx\frac{1}{n}\sum_{i=1}^{\lfloor n\nu \rfloor}\e_i\tau_{\lambda}(X_i)=\frac{1}{n\lambda^{(2a+1)/(2D)}}\sum_{i=1}^{\lfloor n\nu\rfloor}U_i\,,
\end{align*}
holds  uniformly in $\nu\in[\nu_0,1]$ with respect to  the $\|\cdot\|_K$-norm, where $\nu_0\in(0,1]$ is an arbitrary but fixed value. 
Next, we verify the weak invariance principle of the process $\{n^{-1/2}\sum_{i=1}^{\lfloor n\nu\rfloor}U_i\}_{n\in\mathbb N}$.
For this purpose we define the function class
\begin{align}\label{f}
\mathcal F=\bigg\{g:[0,1]\times[0,1]\to\mathbb{R}\,\Big|\sup_{\nu\in[0,1]}\int_0^1|g(s,\nu)|^2\,ds<\infty\bigg\}\,.
\end{align}
The following theorem is proved in Section~\ref{proof:thm:wip} of the online supplementary material.
\begin{theorem}[Weak invariance principle]\label{thm:wip}
Under Assumptions~\ref{a1}--\ref{a:subg} and \ref{a:m}, there exists a mean zero Gaussian process $\{\Gamma(s,\nu)\}_{s,\nu\in[0,1]}$  in $\mathcal F$ defined in \eqref{f},
with covariance function
\begin{align*}
\cov\big\{\Gamma(s_1,\nu_1),\Gamma(s_2,\nu_2)\big\}=(\nu_1\wedge\nu_2)\,C_{U}(s_1,s_2)\,,
\end{align*}
such that, as $n\to\infty$,
\begin{align*}
\sup_{\nu\in[0,1]}\int_0^1\Bigg\{\frac{1}{\sqrt n}\sum_{i=1}^{\lfloor n\nu\rfloor}U_i(s)-\Gamma(s,\nu)\Bigg\}^2\, ds=o_p(1)\,.
\end{align*}
\end{theorem}

Theorem~\ref{thm:wip} shows that the partial sum $n^{-1/2}\sum_{i=1}^{\lfloor n\nu\rfloor}U_i$ can be approximated by a Gaussian process $\Gamma$ in the $L^2$ sense, uniformly in $\nu\in[0,1]$.  As a consequence, we obtain from Theorem~\ref{thm:bahadur} the approximation
\begin{align*} 
\sup_{\nu\in[\nu_0,1]}\,\int_0^1\Big[\sqrt n\lambda^{(2a+1)/(2D)}\nu\big\{\hat\beta_{n,\lambda}(s,\nu)-\beta_0(s)+W_\lambda(\beta_0)\big\}
-\Gamma(s,\nu)\Big]^2\,ds=o_p(1)\,.
\end{align*}
Next, in order to propose our self-normalization methodology, we define a useful quantity regarding the difference between the $L^2$-norms of the estimator  $\hat\beta_{n,\lambda}(\cdot,\nu)$ defined in \eqref{hatbeta} and the true slope function $\beta_0$. For $\nu\in[\nu_0,1]$, let
\begin{align}\label{gnnu}
\hat\G_n(\nu)=\sqrt n\lambda^{(2a+1)/(2D)}\,\nu^2\int_0^1\big\{\hat\beta^{\,2}_{n,\lambda}(s,\nu)-\beta_0^2(s)\big\}\,ds\,.
\end{align}
Theorems~\ref{thm:bahadur} and \ref{thm:wip} allow us to establish in the following theorem the weak convergence of the process $\{\hat\G_n(\nu)\}_{\nu\in[\nu_0,1]}$, which is proved in Section~\ref{proof:thm:2.1} of the online supplementary material.

\begin{theorem}\label{thm:2.1}
Suppose Assumptions~\ref{a1}--\ref{a:m} are satisfied. Then, for the $\hat\G_n$ defined in \eqref{gnnu}, we have
\begin{align*}
\big\{\hat\G_n(\nu)\big\}_{\nu\in[\nu_0,1]}\weakconverge\big\{2\sigma_d\,\nu\BB(\nu)\big\}_{\nu\in[\nu_0,1]}
\qquad \text{in }\ell^\infty([\nu_0,1])\,,
\end{align*}
where $\BB$ denotes the standard Brownian motion and  $\sigma_d$ is defined in \eqref{sigmad2}.

\end{theorem}

Let $\omega$ denote a probability measure on the interval $[\nu_0,1]$, and define
\begin{align}\label{vn}
&\hat{\mathbb{V}}_{n}=\Bigg[\int_{\nu_0}^1\bigg|\nu^2\int_0^1\big\{\hat\beta^{\,2}_{n,\lambda}(s,\nu)-\hat\beta^{\,2}_{n,\lambda}(s,1)\big\}\,ds\bigg|^2\,\omega(d\nu)\Bigg]^{1/2}\,.
%
%
\end{align}
Then, for the statistic $\hat\TT_n$ and  the $L^2$-norm $d_0$ in \eqref{tn} and \eqref{d0}, respectively, by the continuous mapping theorem and Theorem~\ref{thm:2.1}, we find
\begin{align}
\nonumber
\sqrt{n}\lambda^{(2a+1)/(2D)}\Big(\big(\hat\TT_n-d_0\big)\,,\hat\VV_n\Big)&=\Bigg(\hat\G_n(1)\,, \bigg\{\int_{\nu_0}^1\big|\hat\G_n(\nu)-\nu^2\hat\G_n(1)\big|^2\,\omega(d\nu)\bigg\}^{1/2}\Bigg)\\
\converged & \Bigg(2\sigma_d\,\BB(1)\,,2\sigma_d\bigg\{\int_{\nu_0}^1|\nu\,\BB(\nu)-\nu^2\BB(1)|^2\,\omega(d\nu)\bigg\}^{1/2}\Bigg)\,. \label{det9}
\end{align}
In particular, the ratio $(\hat\TT_n-d_0)/\hat\VV_{n}$ will be  asymptotically 
free of the nuisance parameters $a,D$ and $\sigma_d$, provided that $\sigma_d^2 > 0$. The following theorem formalizes this idea of self-normalization, and is proved in Section~\ref{proof:thm:pivot} of the online supplementary material.

\begin{theorem}\label{thm:pivot}
Suppose Assumptions~\ref{a1}--\ref{a:m} are satisfied and assume that $\sigma_d^2 > 0$. For the $\hat\TT_n$, $d_0$ and $\hat\VV_n$ defined in \eqref{tn}, \eqref{d0} and \eqref{vn}, respectively, 
we have
\begin{align}\label{core}
\frac{\hat\TT_n-d_0}{\hat\VV_{n}}\converged\WW=\frac{\BB(1)}{\scaleobj{1.2}{\big\{}\int_{\nu_0}^1|\nu\,\BB(\nu)-\nu^2\,\BB(1)|^2\,\omega(d\nu)\scaleobj{1.2}{\big\}^{\scaleobj{0.83}{1/2}}}}\,.
\end{align}
\end{theorem}

Theorem~\ref{thm:pivot} reveals a self-normalized 
statistic $(\hat\TT_n-d_0)/\hat\VV_{n}$ that converges weakly to a pivotal random variable $\WW$, since its distribution
does not depend on the nuisance parameters (namely $a$ and $D$ in Assumption~\ref{a201}, and the $\sigma_d^2$ in \eqref{sigmad2}) or the eigen-system $\{(\rho_k,\phi_k)\}_{k\geq1}$. Moreover, the distribution of $\WW$ in \eqref{core} can easily be simulated from computer-generated sample paths of standard Brownian motions. For illustration, Table~\ref{tab:q} contains the simulated 90\%, 95\% and 99\%-quantiles of 
the distribution of $\WW$, where $\nu_0=1/4$ and $1/2$, and the probability measure $\omega$ in \eqref{core} is the discrete uniform distribution supported on the set $\{\nu_0+(1-\nu_0)/Q\}_{q=1}^Q$, where $Q=5, 25, 100$. This allows us to define the following test for the relevant hypotheses \eqref{rh}. Letting $\mathcal Q_{1-\alpha}(\WW)$ denote the $(1-\alpha)$-quantile of the distribution of $\WW$ in \eqref{core}, we propose to reject the null hypothesis in \eqref{rh} at nominal level $\alpha$, if
\begin{align}\label{test}
\hat\TT_n> \mathcal Q_{1-\alpha}(\WW)\hat\VV_n+\Delta\,.
\end{align}
The following theorem proved in Section~\ref{proof:thm:one} of the online supplementary material provides a theoretical justification of the consistency of the test defined in \eqref{test} at nominal level $\alpha$.
\begin{table}[htbp]
\caption{Simulated $90\%$, $95\%$ and $99\%$-quantiles of the distribution of $\WW$ defined in \eqref{core}, based on $10^4$ replications, with $\nu_0=1/4$ and $1/2$, and with discrete uniform distribution $\omega$ supported on the set $\{\nu_0+q(1-\nu_0)/Q\}_{q=1}^Q$, for $Q=5,\,25$ and $100$.}
\centering
\begin{tabular}{l|cc|cc|cc} 
\cline{1-7}
&\multicolumn{2}{c|}{90\%}&
\multicolumn{2}{c|}{95\%}&
\multicolumn{2}{c}{99\%}\\
\cline{1-7}
\multicolumn{1}{c|}{$\nu_0$}&$1/4$&$1/2$&$1/4$&$1/2$&$1/4$&$1/2$\\    
\cline{1-7}
$Q=5$ &  8.210 & 9.277 & 11.94 & 13.79 & 21.72 & 25.76 \\
$Q=25$& 7.349  & 8.476  &    10.21&      11.55  &16.43 &   20.37\\
$Q=100$&    7.690 &     8.622 & 10.48 &     12.09  & 16.83 &     20.03 \\
\cline{1-7}
\end{tabular}
\label{tab:q}
\end{table}
\begin{theorem}\label{thm:one}
Assume $\Delta>0$.
Under Assumptions~\ref{a1}--\ref{a:m}
we have
\begin{align}\label{eq:rele}
\lim_{n\to\infty}\P\big\{\hat\TT_n> \mathcal Q_{1-\alpha}(\WW)\hat\VV_n+\Delta\big\}=\left\{\begin{array}{ll}
0 & \quad \text{if }\,d_0<\Delta\,\\
\alpha &\quad \text{if }\,d_0=\Delta\text{ and }\sigma_d^2>0\,\\
1 &\quad \text{if }\,d_0>\Delta\,
\end{array}\right.\,.
\end{align}
\end{theorem}


\begin{remark} \label{delta}
 The choice of the threshold $\Delta$ in the relevant hypotheses  in \eqref{rh} 
 has to be carefully discussed with experts from the field of application. We note that this is not an easy problem, but we argue that instead of testing a null hypothesis, which is believed to be not true, one should carefully think about  the effect, which is of real scientific interest. 

 If this is not possible, we recommend to construct a confidence
 interval. 
To be precise, for the  statistics $\hat\TT_n$ and $\hat\VV_n$ defined in \eqref{tn} and \eqref{vn}, respectively,
the set
\begin{align}
\label{det23}
\hat{\I}_n:=\Big[0\,,\hat\TT_n+\mathcal Q_{1-\alpha}(\WW){\hat{\mathbb{V}}_n}\Big]
\end{align}
defines an asymptotic $(1-\alpha)$-confidence interval for the squared $L^2$-norm 
$d_0=\int^1_0 | \beta_0(s)|^2 d  s$ of the unknown slope function. To see this, note that it follows in the case $d_0 > 0$ from Theorem \ref{thm:pivot} that 
\begin{equation}
\label{det21}
\P_{d_0>0}\big(d _0\in\hat{\I}_n\big)=
\P_{d_0>0} \bigg \{\frac{\hat{\TT}_n-d_0}{\hat{\mathbb{V}}_n}\geq-\mathcal Q_{1-\alpha}(\WW) \bigg \}\to1-\alpha
\end{equation}
as $n\to\infty$, where we have used the fact that the distribution of the random variable
$\mathbb{W}$ in \eqref{core} is symmetric, that is $-\mathcal Q_{1-\alpha}(\WW)=\mathcal Q_{\alpha}(\WW)$. In the case $d_0=0$,
since $\hat\TT_n,\hat\VV_n\geq0$ almost surely, it follows that,
$$
\P_{d_0=0}\big(d_0\in\hat{\I}_n\big)=\P_{d_0=0}\Big\{\hat{\TT}_n+\mathcal Q_{1-\alpha}(\WW)
\hat{\mathbb{V}}_n \geq0 \Big\}=1\,.
$$
Moreover, if it is reasonable to assume that  the parameter $d_0 = \int_0^1 | \beta_0(s)|^2ds$ is positive,
an asymptotic two-sided  confidence interval
for $d_0> 0$
is given by
\begin{align}
\label{det24}
\Big (
\max \big \{ 0,\, \hat\TT_n-\mathcal Q_{1-\alpha/2}(\WW){\hat{\mathbb{V}}_n}\} \,,\hat\TT_n+\mathcal Q_{1-\alpha/2}(\WW){\hat{\mathbb{V}}_n}\Big]~,
\end{align}
which follows by \eqref{core},
observing that,  by \eqref{det9}, $\hat{\mathbb{T}}_n=  d_0 + o_p(1)$ and
$\hat{\mathbb{V}}_n
=o_{p} (1)$ as $n\to\infty$, and $\hat\VV_n\geq0$ almost surely.

Alternatively, it is also possible to test 
 the relevant hypotheses for a finite number of thresholds $\Delta^{(1)} < \ldots < \Delta^{(L)}$ simultaneously, for some $L\in\mathbb N_+$.   In particular, rejection for a $ \Delta^{(L_0)} $  means rejection for all smaller thresholds. In this sense, evaluating the test for several thresholds is  logically consistent for the user, and it is possible to determine for a fixed nominal level $\alpha$   the largest threshold such that the null hypothesis is rejected.
\end{remark}

\begin{remark} \label{omega}
Note that the statistic $\hat{\mathbb{V}}_n$  in \eqref{vn}
depends on the constant $\nu_0$ and the measure $\omega$.
However, we argue that the resulting test \eqref{test}
is not very sensitive with respect to the choice
of these quantities.  Note that 
 these quantities also appear in the definition 
of the pivotal random variable
$\mathbb {W}$ in 
\eqref{core}. Thus, intuitively, there is a cancellation effect in the decision rule \eqref{test}, and we 
demonstrate the  resulting robustness by a small simulation study at the end of Section \ref{sec:simulation}. 
\end{remark}
\begin{remark}\label{beta*}
The methodology for the relevant hypotheses \eqref{rh} can be extended to constructing tests for the relevant hypotheses regarding the location of the slope function at a pre-specified function $\beta_*$, that is
\begin{align}\label{rh*}
H_0:\,\int_0^1|\beta_0(s)-\beta_* (s)|^2\,ds\leq\Delta\quad&\text{versus}\quad H_1:\, \int_0^1|\beta_0(s) - \beta_*(s) |^2\,ds>\Delta\,.
\end{align}
In this case, for the RKHS estimator $\hat\beta_{n,\lambda}(\cdot,\nu)$ in \eqref{hatbeta}, define
\begin{align*}
&\hat{\mathbb{T}}^*_n=\int_0^1|\hat\beta_{n,\lambda}(s,1)-\beta_*(s)|^2\,ds\,,
\\
&\hat{\mathbb{V}}^*_{n}=\Bigg[\int_{\nu_0}^1\bigg|\nu^2\int_0^1\Big\{\big|\hat\beta_{n,\lambda}(s,\nu)-\beta_*(s)\big|^{2}-\big|\hat\beta_{n,\lambda}(s,1)-\beta_*(s)\big|^{2}\Big\}\,ds\bigg|^2\,\omega(d\nu)\Bigg]^{1/2}\,.
\end{align*}
Then, the corresponding decision rule is to reject $H_0$ in \eqref{rh*} at nominal level $\alpha$ if
\begin{align*}
\hat\TT_n^*> \mathcal Q_{1-\alpha}(\WW)\hat\VV_n^*+\Delta\,,
\end{align*}
where $\WW$ is defined in \eqref{core}.
The proof of consistency of the above test can be achieved by using arguments similar to the ones used to prove Theorem~\ref{thm:one}, and is therefore omitted for the sake of brevity.

\end{remark}

\begin{remark} \label{comp}
We shall briefly compare our results with those  in \cite{kutta2021}.
 Roughly speaking, these authors considered  an empirical version of the nornal equation
$$
\E(Y_i X_i)  = {\cal C}_X \beta\,,$$
where ${\cal C}_X$ denotes the operator induced by the covariance kernel $C_X$, which is then solved  by  an application of 
a regularized inverse based on
 a spectral-cut-off series estimator.  
As an alternative, our method is based on a reproducing kernel Hilbert space approach using the minimizer of a regularized
optimization problem.  This makes its extension
for inference 
regarding the slope  in a function-on-function model  very easy (see the discussion in Section~\ref{sec:ff}).
Moreover, the dependence structure of the time series in \cite{kutta2021} is characterized by the so-called $\phi$-mixing (see, for example, \citealp{dehling2002}), whereas in this article we adopt the concept of $m$-approximability. As pointed out by \cite{hormann2010},
 verifying $m$-approximability is much easier than
 the verification of  $\phi$-mixing.
\end{remark}

\subsection{A test for a relevant difference between two slopes}\label{sec:two}

Suppose $\{(X_{1,i},Y_{1,i})\}_{i\in\mathbb Z}$ and $\{(X_{2,i},Y_{2,i})\}_{i\in\mathbb Z}$ denote two independent strictly stationary time series, where the $X_{1,i}$'s and $X_{2,i}$'s are mean zero random functions in $L^2([0,1])$ and the $Y_{j,i}$'s are defined by
\begin{align}\label{model01}
Y_{j,i}=\int_0^1X_{j,i}(s)\,\beta_j(s)\,ds+\epsilon_{j,i}\,,\qquad i\in\mathbb Z\,,\ j=1,2\,.
\end{align}
Suppose the $j$-th sample consists of $n_j$ observations $(X_{j,1},Y_{j,1}),\ldots,(X_{j,n_j},Y_{j,n_j})$, for $j=1,2$.
For a threshold $\Delta>0$, we consider the following relevant relevant hypotheses  for the 
difference between the two slope functions 
w.r.t.~the $L^2$ norm:
\begin{align}\label{htwo}
H_0^T:\,\int_0^1|\beta_1(s)-\beta_2(s)|^2\,ds\leq\Delta\quad&\text{versus}\quad H_1^T:\, \int_0^1|\beta_1(s)-\beta_2(s)|^2\,ds>\Delta\,.
\end{align}
Applying the self-normalization methodology developed in Section~\ref{sec:one} for the one sample case, we first define for each sample the RKHS estimator based on the partial sample. For the Sobolev space $\H$ defined in \eqref{H}, $\nu\in[\nu_0,1]$, and $j=1,2$, define
\begin{align}\label{betat}
\hat\beta_{n_j,\lambda_j}(\cdot,\nu)&=\underset{\beta\in \H}{\arg\min}\ \Bigg[ \frac{1}{2\lfloor n_j\nu\rfloor}\sum_{i=1}^{\lfloor n_j\nu\rfloor}\left\{Y_{j,i}-\int_0^1X_{j,i}(s)\,\beta(s)\,ds\right\}^2+\frac{\lambda_j}{2} J(\beta,\beta)\Bigg]\,,
\end{align}
and take the difference
\begin{align*}
\hat\beta_{n_1,n_2}(\cdot,\nu)=\hat\beta_{n_1,\lambda_1}(\cdot,\nu)-\hat\beta_{n_2,\lambda_2}(\cdot,\nu)\,
\end{align*}
%
(the dependence of the estimator $\hat\beta_{n_1,n_2}(\cdot,\nu) $ on the regularization parameters $\lambda_1, \lambda_2$ will not be refelcted in our notation).
Following ideas similar to the ones for the one sample problem in Section~\ref{sec:one}, we define
\begin{equation}\label{tn1n2}
\begin{split}
&\hat\TT_{n_1,n_2}=\int_0^1\big | \hat\beta_{n_1,n_2}(s,1)\big |^2ds\,,\\
&\hat \VV_{n_1,n_2}=\bigg[\int_{\nu_0}^1\bigg|\nu^2\int_0^1\big\{\hat\beta^{\,2}_{n_1,n_2}(s,\nu)-\hat\beta^{\,2}_{n_1,n_2}(s,1)\big\}\,ds\bigg|^2\,\omega(d\nu)\bigg]^{1/2}\,.
\end{split}
\end{equation}
In order to define our test for the two sample relevant hypotheses in \eqref{htwo}, we apply the methodology in Sections~\ref{sec:rkhs} and \ref{sec:one} and study the asymptotic properties of the statistics $\hat\TT_{n_1,n_2}$ and $\hat \VV_{n_1,n_2}$. For 
$j=1,2$, let
\begin{align}\label{lrj}
\l\beta_1,\beta_2\r_j=\int_0^1\int_0^1\cov\{X_{j,1}(s),X_{j,1}(t)\}\,\beta_1(s)\,\beta_2(t)\,ds\,dt+\lambda_j J(\beta_1,\beta_2)
\end{align}
define an inner product on $\H$. For any $z\in L^2([0,1])$ and $\beta\in\H$, and for $j=1,2$, by the Riesz representation theorem, let $\tau_j(z)\in\mathcal{H}$ denote the unique element such that $\l\tau_j(z),\beta\r_j=\l z,\beta\r_{L^2}$. 
In addition, suppose Assumption~\ref{a201} is satisfied for the $j$-th sample ($j=1,2$) with parameters $a_j,D_j>0$, respectively (see Assumption~\ref{a201t} in Section~\ref{app:a:two} of the online supplementary material for details). 
For $\nu\in[\nu_0,1]$, define
\begin{align}\label{gn1n2}
\hat\G_{n_1,n_2}(\nu)=\sqrt {n_1}\lambda_1^{(2a_1+1)/(2D_1)}\nu^2\int_0^1\Big[\hat\beta^{\,2}_{n_1,n_2}(s,\nu)-\{\beta_1(s)-\beta_2(s)\}^2\Big]\,ds\,,
\end{align}
so that, in view of \eqref{tn1n2},
\begin{align}\label{tv}
&\Bigg(\sqrt{n_1}\lambda_1^{(2a_1+1)/(2D_1)}\bigg\{\hat\TT_{n_1,n_2}-\int_0^1|\beta_1(s)-\beta_2(s)|^2\,ds\bigg\}\,,\sqrt{n_1}\lambda_1^{(2a_1+1)/(2D_1)}\hat\VV_{n_1,n_2}\Bigg)\notag\\
&=\Bigg(\hat\G_{n_1,n_2}(1)\,, \bigg\{\int_{\nu_0}^1\big|\hat\G_{n_1,n_2}(\nu)-\nu^2\hat\G_{n_1,n_2}(1)\big|^2\,\omega(d\nu)\bigg\}^{1/2}\Bigg)\,.
\end{align}
Then, in order to show the asymptotic distributions of $\hat\TT_{n_1,n_2}$ and $\hat\VV_{n_1,n_2}$ in \eqref{tn1n2}, it suffices to show the weak convergence of the process $\{\hat\G_{n_1,n_2}(\nu)\}_{\nu\in[\nu_0,1]}$ defined in \eqref{gn1n2}. To achieve this, we make the following assumption regarding the convergence rates of the sample sizes $n_1,n_2$ and the regularization parameters $\lambda_1,\lambda_2$.
\begin{assumption}\label{a6}
Assume $n_1,n_2\to\infty$ and $n_2\lambda_2^{(2a_2+1)/D_2}/\big(n_1\lambda_1^{(2a_1+1)/D_1}\big)\to\gamma>0$.
\end{assumption}
The following long-run covariance for each sample defined by
\begin{align*}
C_{U,\lambda,j}(s,t)=\lambda_j^{(2a_j+1)/D_j}\sum_{\ell=-\infty}^{+\infty}\cov\big\{\e_{j,0}\, \tau_j(X_{j,0})(s)\,,\e_{j,\ell} \,\tau_j(X_{j,\ell})(t)\big\}\,,\qquad(j=1,2)
\end{align*}
plays a crucial role in the asymptotic distribution of $\hat\G_{n_1,n_2}$.
Let
\begin{align}\label{sigma12}
\sigma_{1,2}^2=\lim_{\lambda\downarrow0}\int_0^1\int_0^1\{C_{U,\lambda,1}(s,t)+\gamma\, C_{U,\lambda,2}(s,t)\}\{\beta_1(s)-\beta_2(s)\}\{\beta_1(t)-\beta_2(t)\}\,ds\,dt\,.
\end{align}
Theorem~\ref{thm:gnt} below establishes the weak convergence of the process $\{\hat\G_{n_1,n_2}(\nu)\}_{\nu\in[\nu_0,1]}$. The proof follows arguments similar to the ones used to prove Theorem~\ref{thm:one}, and is given in Section~\ref{app:two} of the online supplementary material, where we also state the necessary assumptions for this statement.

\begin{theorem}\label{thm:gnt}
Under Assumption~\ref{a6} and Assumptions~\ref{a1t}--\ref{a:mt} in Section~\ref{app:a:two} of the online supplementary material, we have
\begin{align*}
\big\{\hat\G_{n_1,n_2}(\nu)\big\}_{\nu\in[\nu_0,1]}\weakconverge\big\{2\sigma_{1,2}\,\nu\BB(\nu)\big\}_{\nu\in[\nu_0,1]}
\qquad \text{in }\ell^\infty([\nu_0,1])\,,
\end{align*}
where $\BB$ denotes the standard Brownian motion and the $\sigma_{1,2}$ is defined in \eqref{sigma12}.
\end{theorem}

For the $\hat\TT_{n_1,n_2}$ and $\hat\VV_{n_1,n_2}$ in \eqref{tn1n2}, by \eqref{tv} and Theorem~\ref{thm:gnt}, we deduce that,
\begin{align*}
&\Bigg(\sqrt{n_1}\lambda_1^{(2a_1+1)/(2D_1)}\bigg\{\hat\TT_{n_1,n_2}-\int_0^1|\beta_1(s)-\beta_2(s)|^2\,ds\bigg\}\,,\sqrt{n_1}\lambda_1^{(2a_1+1)/(2D_1)}\hat\VV_{n_1,n_2}\Bigg)\\
%
%
&\converged\Bigg(2\sigma_{1,2}\,\BB(1)\,,2\sigma_{1,2}\bigg\{\int_{\nu_0}^1\,|\nu\,\BB(\nu)-\nu^2\,\BB(1)|^2\,\omega(d\nu)\bigg\}^{1/2}\Bigg)\,.
\end{align*}
Therefore, when $\sigma_{1,2}^2 > 0$, by the continuous mapping theorem, we find
\begin{align*}
\frac{\hat\TT_{n_1,n_2}-\int_0^1|\beta_1(s)-\beta_2(s)|^2ds}{\hat\VV_{n_1,n_2}}\converged\frac{2\sigma_{1,2}\BB(1)}{2\sigma_{1,2}\scaleobj{1.2}{\big\{}\int_{\nu_0}^1|\nu\,\BB(\nu)-\nu^2\,\BB(1)|^2\,\omega(d\nu)\scaleobj{1.2}{\big\}^{\scaleobj{0.83}{1/2}}}} 
\stackrel{d}{=}
\WW\,,
\end{align*}
where the $\WW$ is defined in \eqref{core}. Now, we 
propose to reject the null hypothesis in \eqref{htwo} at nominal level $\alpha$ if
\begin{align}\label{two}
\hat\TT_{n_1,n_2}>\mathcal Q_{1-\alpha}(\WW)\hat\VV_{n_1,n_2}+\Delta\,.
\end{align}
The following theorem shows that the test \eqref{two} is a asymptotically consistent test  for the relevant hypotheses \eqref{htwo} at nominal level $\alpha$. The proof is omitted for the sake of brevity,
because it is based on Theorem~\ref{thm:gnt} and follows arguments similar to the ones used to prove Theorem~\ref{thm:one}.

\begin{theorem}\label{thm:two}
Assume $\Delta>0$.
Under the assumptions of Theorem~\ref{thm:gnt},  we have
\begin{align*}
\lim_{n\to\infty}\P\big\{\hat\TT_{n_1,n_2}>\mathcal Q_{1-\alpha}(\WW)\hat\VV_{n_1,n_2}+\Delta\big\}=\left\{\begin{array}{ll}
0 & \quad \text{if}\,\ \|\beta_1-\beta_2\|_{L^2}^2<\Delta\,\\
\alpha &\quad \text{if}\,\ \|\beta_1-\beta_2\|_{L^2}^2=\Delta\text{ and }
\sigma_{1,2}^2>0
\,\\
1 &\quad \text{if}\,\ \|\beta_1-\beta_2\|_{L^2}^2>\Delta\,
\end{array}\right.\,.
\end{align*}
\end{theorem}

\section{Function-on-function linear regression}\label{sec:ff}

In this section, we extend the new  self-normalization methodology 
to the problem of testing relevant hypotheses regarding the slope function in functional linear regression where both the response and the predictor are functions. Suppose that  $\{(X_i,Y_i)\}_{i\in\mathbb Z}$ is a stationary time series in $L^2([0,1]) \times L^2([0,1]) $, defined  by  the  function-on-function linear regression model
\begin{align}\label{model0f}
Y_i(t)=\int_0^1\beta_0(s,t)\,X_i(s)\,ds+\epsilon_i(t)\,,\quad i\in\mathbb Z\,,\ t\in[0,1]\,,
\end{align}
where $\{ \epsilon_i\}_{i \in \mathbb{Z}}$ is a centred random noise process in $L^2([0,1])$, and the slope function $\beta_0$ is defined on $[0,1]^2$. Suppose a sample generated by model \eqref{model0f} is available and consists of $n$ observations $(X_1,Y_1),\ldots,(X_n,Y_n)$.
We consider the following relevant hypotheses 
\begin{align}\label{hf}
H_0^f:\,\int_0^1\int_0^1|\beta_0(s,t)|^2\,ds\,dt\leq\Delta\quad&\text{versus}\quad H_1^f:\, \int_0^1\int_0^1|\beta_0(s,t)|^2\,ds\,dt>\Delta\,,
\end{align}
where $\Delta>0$ is a pre-specified threshold.
Let
\begin{align*} 
\H_f&=\Big\{\beta:[0,1]^2\to\mathbb{R}\,\big|\,\partial^{(\theta_1,\theta_2)}\beta\text{ is absolutely continuous, for }0\leq\theta_1+\theta_2\leq m_f-1\,;\notag\\
&\hspace{3cm}\partial^{(\theta_1,\theta_2)}\beta\in L^2([0,1]^2),\text{ for }\theta_1+\theta_2=m_f\Big\}
\end{align*}
denote the Sobolev space on $[0,1]^2$ of order $m_f>1$.  Let $\nu_0\in(0,1]$ denote an arbitrary fixed constant. For $\nu\in[\nu_0,1]$, following the methodology developed in Section~\ref{sec:sf}, we start by defining the RKHS estimator $\hat\beta_{n,\lambda}(\cdot;\nu)$ based on the observations $(X_1,Y_1),\ldots,(X_{\lfloor n\nu\rfloor},Y_{\lfloor n\nu\rfloor})$ via the following minimization problem:
\begin{align}\label{betaf}
\hat\beta_{n,\lambda}(\,\cdot\,;\nu)
&=\underset{\beta\in \H_f}{\arg\min}\ \Bigg[ \frac{1}{2\lfloor n\nu\rfloor}\sum_{i=1}^{\lfloor n\nu\rfloor}\int_0^1\left\{Y_i(t)-\int_0^1X_i(s)\,\beta(s,t)\,ds\right\}^2dt+\frac{\lambda}{2} J_f(\beta,\beta)\Bigg]\,,
\end{align}
where $J_f$ is the thin-plate spline smoothness penalty functional (see, for example, \citealp{wood2003}), defined by, for $m_f>1$,
\begin{align}\label{jf}
&J_f(\beta_1,\beta_2)=\sum_{\theta=0}^m{m_f\choose \theta}\int_0^1\int_0^1\frac{\partial^{m_f}\beta_1}{\partial s^{\theta}\,\partial t^{m_f-\theta}}\times\frac{\partial^{m_f}\beta_2}{\partial s^{\theta}\,\partial t^{m_f-\theta}}\,ds\,dt\,.
\end{align}

For the RKHS estimator $\hat\beta_{n,\lambda}(\cdot,\nu)$ defined in \eqref{betaf}, we apply the methodology developed in Sections~\ref{sec:rkhs} and \ref{sec:one}, and define the following statistics
\begin{equation}\label{tnf}
\begin{split}
&\hat\TT_n^f=\int_0^1\int_0^1|\hat\beta_{n,\lambda}(s,t;1)|^2\,ds\,dt\,;\\
&\hat{\mathbb{V}}^f_{n}=\Bigg[\int_{\nu_0}^1\,\bigg|\nu^2\int_0^1\int_0^1\big\{\hat\beta^{\,2}_{n,\lambda}(s,t;\nu)-\hat\beta^{\,2}_{n,\lambda}(s,t;1)\big\}\,ds\,dt\bigg|^2\,\omega(d\nu)\Bigg]^{1/2}\,,
\end{split}
\end{equation}
so that under suitable conditions $\hat\TT_n^f$ is a consistent estimator of
\begin{align}\label{d0f}
d_0^f=\int_0^1\int_0^1|\beta_0(s,t)|^2\,ds\,dt\,.
\end{align}
In order to study the asymptotic properties of $\hat\TT_n^f$ and $\hat{\mathbb{V}}^f_{n}$ in \eqref{tnf}, we define an inner product on $\H_f$ by
\begin{align*} 
\langle\beta_1,\beta_2\rangle_f=V_f(\beta_1,\beta_2)+\lambda J_f(\beta_1,\beta_2)\,,
\end{align*}
where 
\begin{align}\label{vf}
&V_f(\beta_1,\beta_2)=\int_{[0,1]^3}C_X(s_1,s_2)\,\beta_1(s_1,t)\,\beta_2(s_2,t)\,ds_1\,ds_2\,dt\,
\end{align}
and  $C_X(s_1,s_2)=\cov\{X_1(s_1),X_1(s_2)\}$
denotes the covariance function of the predictor. It was shown in \cite{dettetang} that, under Assumption~\ref{a1},  the mapping $\l\cdot,\cdot\r_f$ is a well-defined inner product in $\H$, and $\H$ is a reproducing kernel Hilbert space (RKHS) equipped with the inner product $\l\cdot,\cdot\r_f$, and we use $\|\cdot\|_f$ to denote its corresponding norm. 
For functions $x,y$ on $[0,1]$, let $x\otimes y$ denote the function defined by $x\otimes y(s,t)=x(s)y(t)$. 
Following \cite{dettetang}, we assume that there exists a sequence of functions in $\H_f$ that diagonalizes operators $V_f$ in  \eqref{vf} and $J_f$ in  \eqref{jf} simultaneously.

\begin{assumption}[Simultaneous diagonalization for functional response]\label{a201f}
There exists a sequence of functions $\phi_{k\ell} =x_{k\ell} \otimes \eta_\ell \in\H$, such that $\Vert \phi_{k\ell}\Vert_{\infty}\leq c (k\ell)^{a_f}$ for any $k,\ell\geq 1$, and
\begin{align*} 
V_f(\phi_{k\ell},\phi_{k'\ell'})=\delta_{kk'}\,\delta_{\ell\ell'}\,,\qquad J_f(\phi_{k\ell},\phi_{k'\ell'})=\rho_{k\ell}\,\delta_{kk'}\,\delta_{\ell\ell'}\,,\qquad \text{for any }k,k',\ell,\ell'\geq1\,,
\end{align*}
where $a_f\geq 0$, $c>0$ are constants, $\delta_{kk'}$ is the Kronecker delta and $\rho_{k\ell}$ is such that $\rho_{k\ell}\asymp (k\ell)^{2D_f}$ for some constant $D_f>a_f+1/2$. Furthermore, any $\beta\in\H_f$ admits the expansion $\beta=\sum_{k,\ell=1}^\infty V_f(\beta,\phi_{k\ell})\phi_{k\ell}$ with convergence in $\H_f$ with respect to the norm $\Vert\cdot\Vert_f$.
\end{assumption}

Under Assumption~\ref{a201f}, we have
for the inner product $\l\cdot,\cdot\r_f$ in \eqref{inner},
$$
\l\phi_{k\ell},\phi_{k'\ell'}\r_f=V_f(\phi_{k\ell},\phi_{k'\ell'})+\lambda J_f(\phi_{k\ell},\phi_{k'\ell'})=(1+\lambda\rho_{k\ell})\,\delta_{kk'}\,\delta_{\ell\ell'}~~~~(\text{for }k,k',\ell,\ell'\geq 1)\,.
$$
For any $\beta_1,\beta_2\in\mathcal{H}_f$ and $J_f$ defined in \eqref{jf}, let $W_\lambda^f:\H_f\to\H_f$ denote the operator such that $\l W_\lambda^f(\beta_1),\beta_2\r_f=\lambda J_f(\beta_1,\beta_2)$. 
For any $z\in L^2([0,1])$ and $\beta\in\H$, by the Riesz representation theorem, let $\tau_\lambda^f (z)$ denote the unique element in $\H_f$ such that
\begin{align*} 
\l\tau_\lambda^f (z),\beta\r_f=\int_0^1\int_0^1\beta(s,t)\,z(s,t)\,ds\,dt\,.
\end{align*}
In particular, $\l\tau_\lambda^f (z),\phi_{k\ell}\r_f=\l z,\phi_{k\ell}\r_{L^2}$, so that
\begin{align*} 
\tau_\lambda^f (z)=\sum_{k,\ell=1}^\infty\frac{\l z,\phi_{k\ell}\r_{L^2}}{1+\lambda\rho_{k\ell}}\,\phi_{k\ell}\,.
\end{align*}

In order to study the asymptotic properties of the statistics $\hat\TT_n^f$ and $\hat\VV_n^f$ defined in \eqref{tnf}, first, it can be shown that, under suitable conditions,
\begin{align*}
\nu\big\{\hat\beta_{n,\lambda}(\cdot;\nu)-\beta_0+W^f_\lambda(\beta_0)\big\}\approx\frac{1}{n}\sum_{i=1}^{\lfloor n\nu \rfloor}\tau_\lambda^f (X_i\otimes\e_i)\,,
\end{align*}
w.r.t.~the $\|\cdot\|_f$-norm, uniformly in $\nu\in[\nu_0,1]$; for details
see  Lemma~\ref{lem:vn} in the online supplementary materials.
Similar to the discussion  in Section~\ref{sec:one}, we consider, for $\nu\in[\nu_0,1]$,
\begin{align}\label{gnf}
\hat\G_n^f(\nu)=\sqrt {n}\lambda^{(2a_f+1)/(2D_f)}\,\nu^2\int_0^1\int_0^1\big\{\hat\beta^{\,2}_{n,\lambda}(s,t;\nu)-\beta_0^2(s,t)\big\}\,ds\,dt\,.
\end{align}
In addition, define the long-run covariance
\begin{align*} 
C_{f,\lambda}\{(s_1,t_1),(s_2,t_2)\}=\lambda^{(2a+1)/D}\sum_{\ell=-\infty}^{+\infty}\cov\big\{\tau_\lambda^f(X_0\otimes\e_0)(s_1,t_1)\,,\tau_\lambda^f(X_\ell\otimes\e_\ell)(s_2,t_2)\big\}\,,
\end{align*}
and let
\begin{align}\label{sigmaf}
\sigma_f^2=\lim_{\lambda\downarrow0}\int_{[0,1]^4}C_{f,\lambda}\{(s_1,t_1),(s_2,t_2)\}\,\beta_0(s_1,t_1)\,\beta_0(s_2,t_2)\,ds_1\,ds_2\,dt_1\,dt_2\,.
\end{align}
We  now state several assumptions required for the asymptotic theory developed in this section, namely the regularity conditions in Assumption~\ref{a:subgf}, the conditions on the convergence rates for the regularization parameter in Assumption~\ref{a:ratef}, and the $m$-approximability condition for functional time series in Assumption~\ref{a:mf}.

\begin{assumption}\label{a:subgf}~

\begin{enumerate}[label={\rm(\ref*{a:subgf}.\arabic*)},ref={\rm\ref*{a:subgf}.\arabic*},series=a3,nolistsep,leftmargin=1.4cm]
\item\label{a3.1f} There exists a constant $\varpi>0$ such that $\E\{\exp(\varpi\|X_0\|_{L^2}^2)\}<\infty$.

\item\label{a3.2f} 
For any $w\in L^2([0,1])$, $\E\big(\l X_0,w\r_{L^2}^4\big)\leq c_0\big\{\E\big(\l X_0,w\r_{L^2}^2\big)\big\}^2$, for some constant $c_0>0$.

\item\label{wn} $\cov\{\epsilon(t_1),\epsilon(t_2)\}=\sigma_\e^2\,\delta(t_1,t_2)$, for some $\sigma_\e^2>0$, where $\delta$ is the delta function.


\item\label{a3.3f}The true slope function $\beta_0$ is such that $\sum_{k,\ell=1}^\infty  \rho_{k\ell}^2\,V_f^2(\beta_0,\phi_{k\ell})<\infty$.

\item\label{a34f} For $(s_1,t_1),(s_2,t_2)\in[0,1]^2$, the limit $C_{f}\{(s_1,t_1),(s_2,t_2)\}=\lim_{\lambda\downarrow0}C_{f,\lambda}\{(s_1,t_1),(s_2,t_2)\}$ exists.

\end{enumerate}
\end{assumption}

The reason for postulating Assumption~\ref{wn} is that the $L^2$-loss function in \eqref{betaf} corresponds to the likelihood function of Gaussian white noise processes; see, for example, \cite{wellner2003}.



\begin{assumption}\label{a:ratef}

For the constants $a_f$ and $D_f$ in Assumption~\ref{a201f},
and the regularization parameter 
$\lambda$ in \eqref{betaf},  $\lambda=o(1)$, $n^{-1}\lambda^{-(2a_f+1)/D_f}=o(1)$, $ n\lambda^{2+(2a_f+1)/(2D_f)}=o(1)$ as $n\to\infty$. In addition, $n^{-1}\lambda^{-2\varsigma_f}\log n=o(1)$ and $\lambda^{-2\varsigma_f+(2D_f+2a_f+1)/(2D_f)}\log n=o(1)$  as $n\to\infty$, where $\varsigma_f=(2D_f-2a_f-1)/(4D_fm_f)+(a_f+1)/(2D_f)>0$.
\end{assumption}

\begin{assumption}\label{a:mf}~
For $i\in\mathbb Z$, $(X_i,Y_i)$ is generated by the model \eqref{model0f} that follows the following assumptions.
\begin{enumerate}[label={\rm(\ref*{a:mf}.\arabic*)},ref={\rm\ref*{a:mf}.\arabic*},series=Qbeta,nolistsep,leftmargin=1.6cm]

\item 
$X_i=g(\ldots,\xi_{i-1},\xi_i)$ and $\e_i=h(\ldots,\eta_{i-1},\eta_i)$, for some deterministic measurable functions $g,h:\mathcal S^{\infty}\mapsto L^2([0,1])$, where $\mathcal S$ is some measurable space,  $\xi_i=\xi_i(t,\upomega)$ and $\eta_i=\eta_i(t,\upomega)$ are jointly measurable in $(t,\upomega)$. The $\xi_i$'s and the $\eta_i$'s are i.i.d.

\item For any $s,t\in[0,1]$, $\E\{X_0(s)\}=\E\{\e_0(t)\}=0$. For some $\delta\in(0,1)$, $\E\|\e_0\|_{L^2}^{2+\delta}<\infty$.


\item The sequences $ \{X_i\}_{i\in\mathbb Z}$ and $\{\e_i\}_{i\in\mathbb Z}$ can be approximated by $\ell$-dependent sequences $\{X_{i,\ell}\}_{i,\ell\in\mathbb Z}$ and $\{\e_{i,\ell}\}_{i,\ell\in\mathbb Z}$, respectively, in the sense that, for some $\kappa>2+\delta$,
\begin{align*} 
\sum_{\ell=1}^\infty\big(\E\|X_i-X_{i,\ell}\|_{L^2}^{2+\delta}\big)^{1/\kappa}<\infty\,,\qquad\sum_{\ell=1}^\infty\big(\E\|\e_i-\e_{i,\ell}\|_{L^2}^{2+\delta}\big)^{1/\kappa}<\infty\,.
\end{align*}
Here, $X_{i,\ell}=g(\xi_i,\xi_{i-1},\ldots,\xi_{i-\ell+1},\boldsymbol{\xi}_{i,\ell}^*)$ and $\e_{i,\ell}=h(\eta_i,\eta_{i-1},\ldots,\eta_{i-\ell+1},\boldsymbol{\eta}_{i,\ell}^*)$, where $\boldsymbol{\xi}_{i,\ell}^*=(\xi^*_{i,\ell,i-\ell},\xi^*_{i,\ell,i-\ell-1},\ldots)$ and $\boldsymbol{\eta}_{i,\ell}^*=(\eta^*_{i,\ell,i-\ell},\eta^*_{i,\ell,i-\ell-1},\ldots)$, and where the $\xi^*_{i,\ell,k}$'s and the $\eta^*_{i,\ell,k}$'s are independent copies of $\xi_0$ and $\eta_0$, and are independent of $\{\xi_i\}_{i\in\mathbb Z}$ and $\{\eta_i\}_{i\in\mathbb Z}$, respectively.
\end{enumerate}
\end{assumption}

The following theorem establishes the weak convergence of the process $\{\hat\G_n^f(\nu)\}_{\nu\in[\nu_0,1]}$ defined in \eqref{gnf}, and is proved in Section~\ref{app:thm:gnf} of the online supplementary material.

\begin{theorem}\label{thm:gnf}
Suppose that Assumptions~\ref{a1} and \ref{a201f}--\ref{a:mf} are satisfied. Then, for the constant $\sigma_f^2$ defined in \eqref{sigmaf}, we have
\begin{align*}
\big\{\hat\G_n^f(\nu)\big\}_{\nu\in[\nu_0,1]}\weakconverge\big\{2\sigma_f\,\nu\BB(\nu)\big\}_{\nu\in[\nu_0,1]}
\qquad \text{in }\ell^\infty([\nu_0,1])\,,
\end{align*}
where $\BB$ denotes the standard Brownian motion.
\end{theorem}

Recalling  the definition of the statistics $\hat\TT_n^f$ and $\hat\VV_n^f$ in \eqref{tnf} and of  $d_0^f$ in \eqref{d0f}, it follows from Theorem~\ref{thm:gnf} and the continous mapping theorem that
\begin{align*}
%
\sqrt{n}\lambda^{(2a_f+1)/(2D_f)}\Big(\big(\hat\TT_n^f-d_0^f\big),\hat\TT_n^f\Big)&=\Bigg(\hat\G_{n}^f(1)\,, \bigg\{\int_{\nu_0}^1\big|\hat\G_{n}^f(\nu)-\nu^2\,\hat\G_{n}^f(1)\big|^2\,\omega(d\nu)\bigg\}^{1/2}\Bigg)\\
&\converged\Bigg(2\sigma_f\,\BB(1)\,,2\sigma_f\bigg\{\int_{\nu_0}^1\,|\nu\,\BB(\nu)-\nu^2\,\BB(1)|^2\,\omega(d\nu)\bigg\}^{1/2}\Bigg)\,.
\end{align*}
Therefore, by the continuous mapping theorem, when $\sigma_f^2 > 0$,
\begin{align*}
\frac{\hat\TT_{n}^f-d_0^f}{\hat\VV_{n}^f}\converged\frac{2\sigma_f\BB(1)}{2\sigma_f\scaleobj{1.2}{\big\{}\int_{\nu_0}^1|\nu\,\BB(\nu)-\nu^2\,\BB(1)|^2\,\omega(d\nu)\scaleobj{1.2}{\big\}^{\scaleobj{0.83}{1/2}}}}\stackrel{d}{=}\WW\,,
\end{align*}
where the random variable $\WW$ is defined in \eqref{core}.
%
Finally, we propose to reject the null hypothesis in \eqref{hf} at nominal level $\alpha$ if
\begin{align}\label{testf}
\hat\TT_{n}^f>\mathcal Q_{1-\alpha}(\WW)\hat\VV_{n}^f+\Delta\,.
\end{align}
The following theorem shows the consistency of the test \eqref{testf} for the relevant hypotheses \eqref{hf} at nominal level $\alpha$. The proof uses arguments similar to the ones used to prove Theorem~\ref{thm:one}, and is therefore omitted for the sake of brevity.


\begin{theorem}\label{thm:f}
Assume $\Delta>0$.
Under Assumptions~\ref{a1} and \ref{a201f}--\ref{a:mf}, we have
\begin{align*}
\lim_{n\to\infty}\P\big\{\hat\TT_{n}^f>\mathcal Q_{1-\alpha}(\WW)\hat\VV_{n}^f+\Delta\big\}=
\left\{\begin{array}{ll}
0 & \quad \text{if}\,\ d_0^f<\Delta\\
\alpha &\quad \text{if}\,\ d_0^f=\Delta\text{ and }
\sigma_f^2>0\\
1 &\quad \text{if}\,\ d_0^f>\Delta
\end{array}\right.\,.
\end{align*}
\end{theorem}

\section{Finite sample properties}\label{sec:finitesample}

\subsection{Implementation}\label{sec:implementation}

In this section we discuss some details regarding the implementation of the proposed tests 
 for the relevant hypotheses. In Section~\ref{sec:scalar} we deal with scalar-on-function linear regression and Section~\ref{sec:function} is dedicated to function-on-function linear regression. 

To begin with, for the pivotal random variable $\WW$ in \eqref{core}, in practice, we may choose the probability measure $\omega$ as the discrete uniform distribution on the interval $[\nu_0,1]$. To be precise, for some positive integer $Q$, let
\begin{align}\label{nuq}
\nu_q=\nu_0+q(1-\nu_0)/Q\,,\qquad\quad\text{for }1\leq q\leq Q\,.
\end{align}
Then, we define $\omega$ as the discrete uniform distribution supported on the set $\{\nu_q\}_{q=1}^Q$ with equal probability mass $1/Q$.
With this choice,   the pivotal random variable $\WW$ in \eqref{core} is given by
\begin{align}\label{wq}
\WW_Q=\frac{\BB(1)}{\scaleobj{1.2}{\big\{}Q^{-1}(1-\nu_0)\sum_{q=1}^Q\big|\nu_q\,\BB(\nu_q)-\nu^2_q\,\BB(1)\big|^2\scaleobj{1.2}{\big\}^{\scaleobj{0.83}{1/2}}}}\,.
\end{align}

\subsubsection{Scalar-on-function linear regression}\label{sec:scalar}

In this section we propose finite-sample methods for implementing the test \eqref{test} in Section~\ref{sec:one} for the relevant hypotheses \eqref{rh}.
Recall that for the test defined in \eqref{test}, we need to compute the statistics $\hat\TT_n$ and $\hat\VV_n$ defined in \eqref{tn} and \eqref{vn}, respectively. We start by rewriting $\hat\TT_n$ and $\hat\VV_n$ by $\hat\TT_{n,Q}=\big\|\hat\beta_{n,\lambda}(\cdot,\nu_Q)\big\|_{L^2}^2$ and
\begin{align}\label{temp}
&\hat\VV_{n,Q}=\Bigg\{\frac{1-\nu_0}{Q}\,	\sum_{ q =1}^{Q}\nu_ q ^4\,\Big(\big\|\hat\beta_{n,\lambda}(\cdot,\nu_q)\big\|_{L^2}^2-\big\|\hat\beta_{n,\lambda}(\cdot,\nu_Q)\big\|_{L^2}^2\Big)^2\Bigg\}^{1/2}\,.
\end{align}
Here, for $1\leq q\leq Q$, the RKHS estimator $\hat\beta_{n,\lambda}(\cdot,\nu_q)$ of the slope function is  defined in \eqref{hatbeta}, based on the observations $(X_1,Y_1),\ldots,(X_{n_q},Y_{n_q})$,
where $n_q = \lfloor \nu_qn \rfloor $ ($q=1, \ldots, Q$).
Since $\hat\beta_{n,\lambda}(\cdot,\nu_q)$ is defined as the solution of a penalized minimization problem on  an infinite dimensional function space $\H$ defined in \eqref{H}, exact solutions are inaccessible. We circumvent this difficulty by introducing the following finite-sample method, and propose a method to choose the regularization parameter $\lambda$ in \eqref{H}. To begin with, we deduce from Assumption~\ref{a201} that $J(x_{k},x_{k'\ell'})=\rho_{k}\,\delta_{kk'}$, so that for $\beta=\sum_{k=1}^\infty b_{k}\,\phi_{k}\in\H$ and for $b_{k}\in\mathbb{R}$, we have $J(\beta,\beta)=\sum_{k=1}^\infty b_{k}^2\,\rho_{k}$. Consider the Sobolev space on $[0,1]$ of order $m=2$. In this case, the penalty functional in \eqref{hatbeta} is $J(\beta,\beta)=\int_0^1\{\beta''(s)\}^2ds$. In view of \eqref{id}, we choose the $\rho_k$'s and $\phi_k$'s as 
the  eigenvalues and eigenfunctions of  integro-differential equation
\begin{align}\label{eq}
\left\{
\begin{aligned}
&\displaystyle(-1)^{m}\,\phi^{(4)}(s)=\rho\int_0^1 C_X(s,t)\, \phi(t)\,dt\,,\\
&\phi^{(\theta)}(0)=\phi^{(\theta)}(1)=0\,,\qquad\text{ for }\theta=3\text{ and }4\,.
\end{aligned}\right.
\end{align}
In order to find the eigenvalues and the eigenfunctions of \eqref{eq}, we 
use \texttt{Chebfun}, an efficient open-source Matlab add-on package available at \nolinkurl{https://www.chebfun.org/}. We substitute the covariance function $C_X$ in \eqref{eq} by its empirical version $\hat{C}_X$, and compute the eigenvalues $\hat\rho_{k}$ and the normalized eigenfunctions $\hat\phi_{k}$ of equation \eqref{eq}.
This allows us to approximate the space $\H$ by a finite-dimensional linear space spanned by $\{\hat\phi_{k}\}_{1\leq k\leq r}$, defined by $\tilde\H=\big\{\sum_{1\leq k\leq r}b_{k}\,\hat \phi_{k}\big\}$, where $r$ is a parameter that depends on the sample size $n$.

For $1\leq q\leq Q$, $1\leq i\leq n_q$ and $1\leq k\leq r$, let $\omega_{ik}=\int_0^1X_i(s)\hat \phi_{k}(s)ds$ and let $\Omega_q=(\omega_{ik})_{1\leq i\leq{n_q},1\leq k\leq r}$ denote a $n_q\times r$ matrix; let $\Lambda={\rm diag}\big\{\hat\rho_{1},\ldots,\hat\rho_{r}\big\}$ denote an $r\times r$ diagonal matrix;
let $\tilde Y_q=(Y_{1},Y_2,\ldots, Y_{n_q})\trans\in\mathbb R^{n_q}$ denote a $n_q$-dimensional  vector. 
If we write $\tilde\beta(\cdot,\nu_q)=\sum_{k=1}^r b_{k,q}\,\hat\phi_{k}\in\tilde\H$, then, in order to approximate $\hat\beta_{n,\lambda}(\cdot,\nu_q)$ in \eqref{temp}, we find the $b_{k,q}$'s by solving the following optimization problem
\begin{align}\label{lambda}
\hat b_q=(\hat b_{1,q},\ldots,\hat b_{r,q})\trans&=\underset{b_1,\ldots,b_r\in\mathbb R}{\arg\min}\Bigg\{\frac{1}{2n_q}\sum_{i=1}^{n_q}\bigg| Y_i-\sum_{k=1}^{r}b_{k}\int_0^1 X_i(s)\,\hat \phi_{k}(s)\,ds\bigg|^2+\frac{\lambda}{2}  \sum_{k=1}^{r}b_{k}^{2}\,\hat \rho_{k}\Bigg\}\notag\\
&=\underset{b\,\in\mathbb{R}^r}{\arg\min}\Bigg\{\frac{1}{2n_q}(\tilde Y_{q}-\Omega_q\,b)\trans(\tilde Y_{q}-\Omega_q\,b)+\frac{\lambda}{2} \sum_{k=1}^r b_q\trans\,\Lambda\, b_q\,\Bigg\}\,.
\end{align}
By direct calculations, for $1\leq q \leq Q$, we find the solution to \eqref{lambda} defined by
\begin{align}\label{hatb}
\hat b_{q}=(\Omega_q\trans\,\Omega_q+n_q\lambda\Lambda)^{-1}\,\Omega_q\trans\,\tilde Y_q\,,
\end{align}
so that if we let $\hat\phi=(\hat \phi_{1},\ldots,\hat\phi_{r})\trans$ denote a function-valued $r$-dimensional vector, we can approximate the estimator 
$\hat\beta_{n,\lambda}(\cdot,\nu_q)$ in \eqref{hatbeta} by 
\begin{align*}
\tilde\beta(\cdot,\nu_q)=\hat b_{q}^{\,\mathsf{T}}\,\hat\phi\,,\qquad 1\leq q\leq Q\,.
\end{align*}
Let $\hat\Phi=\big(\l\hat\phi_i,\hat\phi_j\r_{L^2}\big)_{r\times r}$ denote an $r\times r$ matrix. Then, recalling the definition of  $\hat b_q$ in \eqref{hatb}, we may approximate $\hat\TT_{n,Q}=\int_0^1|\hat\beta_{n,\lambda}(s,\nu_Q)|^2\,ds$ and $\hat\VV_{n,Q}$ in \eqref{temp} by $\tilde\TT_{n,Q}={\hat b_{Q}}^{\,\mathsf{T}}\,\hat\Phi\,\hat b_{Q}$ and
\begin{align*}
&\tilde\VV_{n,Q}=\Bigg\{\frac{1-\nu_0}{Q}\,\sum_{ q =1}^{Q}\nu_ q ^4\,\Big(\hat b_{q}^{\,\mathsf{T}}\,\hat\Phi\,\hat b_{q}-{\hat b_{Q}}^{\,\mathsf{T}}\,\hat\Phi\,\hat b_{Q}\Big)^2\Bigg\}^{1/2}\,,
\end{align*}
respectively. Then,   the decision rule in 
the test \eqref{test} is defined by  rejecting $H_0$ in \eqref{rh} at nominal level $\alpha$, if
\begin{align}
\label{det5}
\tilde\TT_{n,Q} > \mathcal Q_{1-\alpha}(\WW_Q) \tilde\VV_{n,Q}+\Delta\,,
\end{align}
where $Q_{1-\alpha}(\WW_Q)$ denotes the $(1-\alpha)$-quantile of the pivotal distribution of 
the $\WW_Q$  in \eqref{wq}.




In order to choose the regularization parameter $\lambda$ in \eqref{lambda}, we propose to use a modified version of the generalized cross-validation (GCV, see, for example, \citealp{golub} and \citealp{whaba1990}). To be specific, for $1\leq q\leq Q$, let $\hat Y_{q}(\lambda)=\Omega_q(\Omega_q\trans\,\Omega_ q +n_q\lambda\Lambda)^{-1}\Omega_q\trans\,\tilde Y_q$ and let $H_{q}(\lambda)$ denote the so-called hat matrix with ${\rm tr}\{H_{q}(\lambda)\}={\rm tr}\{\Omega_q(\Omega_q\trans\,\Omega_q+n_q\lambda\,\Lambda)^{-1}\,\Omega_q\trans\}$. Then, we propose to choose $\lambda$ as the minimizer of  the modified GCV score
\begin{align*}
\text{GCV}(\lambda)=\sum_{q=1}^Q\frac{n_q^{-1}\Vert \hat Y_{q}(\lambda)-\tilde Y_q\Vert_{2}^2}{|1-n_q^{-1}\,{\rm{tr}}\{H_q(\lambda)\}|^2}\,.
\end{align*}

\subsubsection{Function-on-function linear regression}\label{sec:function}

We now consider the implementation of the  test \eqref{testf} for the relevant hypotheses \eqref{hf}
in  the function-on-function linear regression. Recall from the beginning of Section~\ref{sec:finitesample} that we take the probability measure $\omega$ as the discrete uniform distribution supported on the set $\{\nu_q\}_{q=1}^Q$ defined in \eqref{nuq}. Then,  recalling from \eqref{betaf} that the RKHS estimator $\hat\beta_{n,\lambda}(\cdot;\nu_q)$ is based on  the observations $(X_1,Y_1),\ldots,(X_{\lfloor n\nu_q\rfloor},Y_{\lfloor n\nu_q\rfloor})$, we rewrite 
$\hat\TT_n^f$ and $\hat\VV_n^f$ in \eqref{tnf} as
\begin{equation}\label{tnq}
\begin{split}
&\hat\TT_{n,Q}^f=\int_0^1\int_0^1|\hat\beta_{n,\lambda}(s,t;\nu_Q)|^2\,ds\,dt\,;\\
&\hat{\mathbb{V}}^f_{n,Q}=\Bigg[\frac{1-\nu_0}{Q}\,	\sum_{ q =1}^{Q}\nu_ q ^4\,\bigg|\int_0^1\int_0^1\big\{\hat\beta^{\,2}_{n,\lambda}(s,t;\nu_q)-\hat\beta^{\,2}_{n,\lambda}(s,t;\nu_Q)\big\}\,ds\,dt\bigg|^2\Bigg]^{1/2}\,.
\end{split}
\end{equation}
We then proceed to introduce our finite-sample methods for computing $\hat\beta_{n,\lambda}(\cdot;\nu_q)$. Recall from Assumption~\ref{a201f} that $J_f(\phi_{k\ell},\phi_{k'\ell'})=\rho_{k\ell}\,\delta_{kk'}\,\delta_{\ell\ell'}$, so that for $\beta=\sum_{k,\ell=1}^\infty b_{k\ell}\,\phi_{k\ell}\in\H_f$ and for $b_{k\ell}\in\mathbb{R}$, we have $J_f(\beta,\beta)=\sum_{k,\ell=1}^\infty b_{k\ell}^2\,\rho_{k\ell}$. If we consider the Sobolev space on $[0,1]^2$ of order $m=2$, then, in this case, the penalty functional in \eqref{betaf} is $J_f(\beta,\beta)=\int_0^1\int_0^1(\beta_{ss}^2+2\beta_{st}^2+\beta_{tt}^2)\,ds\,dt$, where $\beta_{st} = \frac{\partial^2 \beta  }{\partial s \partial t}$. Applying similar ideas as in the scalar response case 
proposed in Section~\ref{sec:scalar}, we propose to approximate the infinite-dimensional space $\mathcal H_f$ by a finite-dimensional space. To achieve this, we first find the empirical eigenfunctions in Assumptions~\ref{a201f} by applying Proposition~B.1 in \cite{dettetang} and considering the following integro-differential equations with boundary conditions:
\begin{align}\label{eqf}
\left\{ 
\begin{array}{ll}
\displaystyle\rho_\ell\int_{0}^1\hat C_X(s,t)\, x(t)\,dt= x^{(4)}(s)-2(\ell-1)^2\pi^2 x^{(2)}(s)+(\ell-1)^4\pi^4\,,\\
 x^{(\theta)}(0)= x^{(\theta)}(1)=0\,,\qquad\text{ for }\theta=3\text{ and }4\,,
\end{array}\right.
\end{align}
for each $\ell\geq1$, where $\hat C_X$ denotes the empirical covariance function of the predictor $X$. For each $\ell\geq1$, let $\{\hat x_{k\ell}\}_{k\geq1}$ denote the normalized eigenfunctions of \eqref{eqf} with the corresponding eigenvalues $\{\hat\rho_{k\ell}\}_{k\geq1}$, which can be obtained by using the Matlab package \texttt{Chebfun}. Let $\{\eta_\ell\}_{\ell\geq1}$ denote the cosine basis of $L^2([0,1])$, that is, 
$
\eta_1\equiv 1$,
$\eta_\ell(t) = \sqrt{2}\cos\{(\ell-1)\pi t\}$,   ($ \ell=  2,3, \ldots $). Then, we take the empirical eigenfunction $\hat\phi_{k\ell}=\hat x_{k\ell}\otimes\eta_\ell$, for $k,\ell\geq1$. Now, we approximate the Sobolev space $\H_f$ defined in \eqref{hf} by
$\tilde\H_f=\big\{\sum_{k=1}^r\sum_{\ell=1}^rb_{k\ell}\,\hat \phi_{k\ell}:b_{k\ell}\in\mathbb R\big\}$, where $r$ is a truncation parameter that depends on the sample size $n$.

Recall from \eqref{betaf} that $\hat\beta_{n,\lambda}(\cdot;\nu_q)$ is the RKHS estimator based on the observations $(X_1,Y_1),\ldots,(X_{n_q},Y_{n_q})$, where $n_q=\lfloor n\nu_q\rfloor$ and   $\nu_q$ 
is defined in \eqref{nuq}. Now, for $1\leq q\leq Q$, $1\leq i\leq n_q$ and $1\leq k,\ell\leq r$, let $\omega_{ik\ell}=\int_0^1X_i(s)\hat x_{k\ell}(s)ds$; for each $1\leq \ell\leq r$ and $1\leq q\leq Q$, let $\Omega_{\ell,q}=(\omega_{ik\ell})_{1\leq i\leq n_q,1\leq k\leq r}$ denote a $n_q\times r$ matrix; let $\hat\Lambda_{\ell}={\rm diag}\big\{\hat\rho_{1\ell},\ldots,\hat\rho_{r\ell}\big\}$ denote an $r\times r$ diagonal matrix; let $\breve Y_{i\ell}=\l Y_i,\eta_\ell\r_{L^2}$ and let $\tilde Y_{\ell,q}=(\breve Y_{1\ell},\ldots,\breve Y_{n_q\ell})\trans\in\mathbb{R}^{n_q}$ denote a $n_q$-dimensional vector. If we write $\tilde\beta(\cdot,\nu_q)=\sum_{k=1}^r\sum_{\ell=1}^r\tilde b_{k\ell}^{(q)}\,\hat\phi_{k\ell}\in\tilde\H_f$, for $\tilde b_{k\ell}^{(q)}\in\mathbb R$, then, in order to approximate $\hat\beta_{n,\lambda}(\cdot;\nu_q)$ in \eqref{betaf}, for each $1\leq q\leq Q$, we can find the  coefficients $\tilde b_{k\ell}^{(q)}$'s by solving the following optimization problem
\begin{align}\label{lambdaf}
\{\tilde b_{k\ell}^{(q)}\}&=\underset{\{b_{k\ell}^{(q)}\}}{\arg\min}\left\{\frac{1}{2n_q}\sum_{i=1}^{n_q}\int_0^1\bigg| Y_i(t)-\sum_{k,\ell=1}^{r}b_{k\ell}^{(q)}\,\eta_\ell(t)\int_0^1 X_i(s)\,\hat x_{k\ell}(s)\,ds\bigg|^2dt+\frac{\lambda}{2}  \sum_{k,\ell=1}^{r}b_{k\ell}^{(q)^2}\,\hat \rho_{k\ell}\right\}\notag\\
&=\underset{\{b_{k\ell}^{(q)}\}}{\arg\min}\left\{\frac{1}{2n_q}\sum_{i=1}^{n_q}\sum_{\ell=1}^{r}\bigg|\breve Y_{i\ell}-\sum_{k=1}^{r}b_{k\ell}^{(q)}\int_0^1 X_i(s)\,\hat x_{k\ell}(s)\,ds\bigg|^2+\frac{\lambda }{2} \sum_{k,\ell=1}^{r}b_{k\ell}^{(q)^2}\,\hat \rho_{k\ell}\right\}\notag\\
&=\underset{b_1^{(q)},\ldots,b_r^{(q)}\in\mathbb R^r}{\arg\min}\Bigg\{\frac{1}{2n_q}\sum_{\ell=1}^r\big(\tilde Y_{\ell,q}-\Omega_{\ell,q}\,b_{\ell}^{(q)}\big)\trans\big(\tilde Y_{\ell,q}-\Omega_{\ell,q}\,b_{\ell}^{(q)}\big)+\frac{\lambda}{2} \sum_{\ell=1}^r b_{\ell}^{(q)\trans}\,\hat\Lambda_\ell\, b_{\ell}^{(q)}\Bigg\}\,,
\end{align}
where we write $b_{\ell}^{(q)}=(b_{1\ell}^{(q)},\ldots,b_{r\ell}^{(q)})\trans\in\mathbb{R}^r$ for $1\leq\ell\leq r$. A direct calculation shows that  for $1\leq\ell\leq r$ the solution of \eqref{lambdaf} is given by 
\begin{align*} 
\hat b_{\ell}^{(q)}=\big(\Omega_{\ell,q}\trans\,\Omega_{\ell,q}+n_q\lambda\hat\Lambda_\ell\big)^{-1}\Omega_{\ell,q}\trans\,\tilde Y_{\ell,q}\,.
\end{align*}
Therefore, we can approximate  the estimator $\hat\beta_{n,\lambda}(\cdot;\nu_q)$ in \eqref{betaf} by 
\begin{align*}
\tilde\beta(\cdot;\nu_q)=\sum_{\ell=1}^r(\hat b_{\ell}^{(q)\trans}\,\hat x_{\ell})\otimes\eta_\ell\,, 
\end{align*}
where  $\hat x_\ell=(\hat x_{1\ell},\ldots,\hat x_{r\ell})\trans$ denotes an  $r$-dimensional vector of functions. Then, the statistics $\hat\TT_{n,Q}^f$ and $\hat\VV_{n,Q}^f$ in \eqref{tnq} can be accordingly approximated by $\tilde\TT_{n,Q}^f=\sum_{\ell=1}^r\hat b_{\ell}^{(Q)\trans}\,\hat\Phi\,\hat b_{\ell}^{(Q)}$ and
\begin{align*}
%
&\tilde\VV_{n,Q}^f=\Bigg[\frac{1-\nu_0}{Q}\,\sum_{ q =1}^{Q}\nu_ q ^4\,\bigg\{\sum_{\ell=1}^r\Big(\hat b_{\ell}^{(q)\trans}\,\hat\Phi\,\hat b_{\ell}^{(q)}-\hat b_{\ell}^{(Q)\trans}\,\hat\Phi\,\hat b_{\ell}^{(Q)}\Big)\bigg\}^2\Bigg]^{1/2}\,,
\end{align*}
respectively. The decision rule in the test \eqref{testf} is finally defined by, rejecting the null hypothesis in \eqref{hf} at nominal level $\alpha$ if
\begin{align}
\label{det6}
\tilde\TT_{n,Q}^f >\mathcal Q_{1-\alpha}(\WW_Q)
\tilde\VV_{n,Q}^f
+\Delta\,,
\end{align}
where $\mathcal Q_{1-\alpha}(\WW_Q)$ denotes the $(1-\alpha)$-quantile quantile of the pivotal distribution of $\WW_Q$ defined in \eqref{wq}.
A modified version of the generalized cross-validation (GCV) is applied to choose the regularization parameter $\lambda$ in \eqref{lambdaf}. More precisely, we take the value that minimizes the modified GCV score
\begin{align*}
\text{GCV}(\lambda)=\sum_{q=1}^Q\frac{n_q^{-1}\sum_{\ell=1}^r\Vert \hat Y_{\ell,q}(\lambda)-\tilde Y_{\ell,q}\Vert_{2}^2}{|1-{\rm{tr}}\{H_q(\lambda)\}/n_q|^2}\,,
\end{align*}
where $\hat Y_{\ell,q}(\lambda)=\Omega_{\ell,q}(\Omega_{\ell,q}\trans\,\Omega_{\ell,q}+n_q\lambda\hat\Lambda_\ell)^{-1}\Omega_{\ell,q}\trans\,\tilde Y_{\ell,q}$ and $H_{q}(\lambda)$ is the so-called hat matrix with ${\rm tr}\{H_{q}(\lambda)\}=\sum_{\ell=1}^r{\rm tr}\{\Omega_{\ell,q}(\Omega_{\ell,q}\trans\,\Omega_{\ell,q}+n_q\lambda\,\hat\Lambda_\ell)^{-1}\,\Omega_{\ell,q}\trans\}$.

\subsection{Simulated data}
\label{sec:simulation}

We applied our method to various settings of simulated data, where we consider both scalar-on-function and function-on-function linear regression. In order to evaluate the function $X$ (and $Y$ in the functional response case) on its domain $[0,1]$, we take 
$100$ equally spaced time points. For all the settings, we took the nominal level $\alpha=0.05$, and all reported results are based on $500$ simulation runs.

We first consider scalar-on-function linear regression \eqref{model0} and the relevant hypotheses \eqref{rh}. For the true slope function $\beta_0$, we consider the following two settings.
\begin{enumerate}[]
\item[(S1)] Let $f_1\equiv1$, $f_{j+1}(s)=\sqrt{2}\cos(j\pi s)$, for $j\geq 1$, and define $\beta_0=\tilde\beta_0/\|\tilde\beta_0\|_{L^2}$, where
$$
\tilde\beta_0(s)=f_1(s)+4\sum_{j=2}^{50}(-1)^{j+1}j^{-2}f_j(s)\,,\qquad s\in[0,1]\,.$$

\item[(S2)] $\beta_0(s)=\sqrt2\exp(-s/4)$.

\end{enumerate}
The first setting (S1) is similar to the ones used in \cite{yuancai}, except that we standardize the slope function such that $d_0=\|\beta_0\|_{L^2}^2=1$. For the second setting (S2), we have $d_0=\|\beta_0\|_{L^2}^2=4-4/\sqrt e\approx1.57$.
For the predictor process $\{X_i\}_{i\in\mathbb Z}$, we consider a similar setting as in \cite{dettejrssb2020}.
We first generate i.i.d.~random variables $\eta_i$ defined by $\eta_i=\sum_{j=1}^{50}j^{-1}Z_{ij} \,f_j$, where $Z_{ij}\iidsim {\rm Normal}(0,1)$, for $i\in\mathbb Z$ and $1\leq j\leq 50$. We consider the following two settings.

\begin{enumerate}[label={\rm(\roman*)},series=i]
\item\label{i} The functional moving average process FMA(1) defined by, for $1\leq i\leq n$,
\begin{align*}
X_i=\eta_i+\theta_i\,\eta_{i-1}\,,
\end{align*}
where $\theta_{i}\iidsim \text{unif}(-1/\sqrt{2},1/\sqrt{2})$.

\item\label{ii}  The i.i.d.~case $X_i=\sqrt{7/6}\,\eta_i$, so that the predictors in Settings~\ref{i} and \ref{ii} have the same point-wise variance.
\end{enumerate}
For the error process, we generate i.i.d.~standard normal random variables $\{\xi_i\}$ and  take $\e_i=c_1\cdot(\xi_i+\upsilon_{i,1}\xi_{i-1}+\upsilon_{i,2}\xi_{i-2})$, where $\upsilon_{i,j}\iidsim \text{unif}(-1/\sqrt{2},1/\sqrt{2})$, for $i\in\mathbb Z$ and $j=1,2$, and $c_1>0$ is chosen such that $\var(\e_i)/\var(\|X_i\|_{L^2})=0.3$.

%

For the function-on-function linear regression \eqref{model0f}, we consider the relevant hypotheses \eqref{hf}. For the true slope function $\beta_0$, we consider the following two settings:
\begin{enumerate}[]
\item[(F1)] Let $f_1\equiv1$, $f_{j+1}(s)=\sqrt{2}\cos(j\pi s)$, for $j\geq 1$, and take $\beta_0=\tilde\beta_0/\|\tilde\beta_0\|_{L^2}$, where
$$
\tilde\beta_0(s,t)=f_1(s)f_1(t)+4\sum_{j=2}^{50}(-1)^{j+1}j^{-2}f_j(s)f_j(t)\,.$$

\item[(F2)] Let $\beta_0(s,t)=\sqrt 2\exp\{-(s+t)/4\}$.

\end{enumerate}
The slope functions in the 
first setting (F1) is similar to the one used in \cite{sun2018}, except that we standardize the slope function so that $d_0^f=\|\beta_0\|_{L^2}^2=1$. For the second setting (F2) we have $d_0^f=\|\beta_0\|_{L^2}^2=8(1-1/\sqrt e)^2\approx 1.24$. The settings for the predictor process $\{X_i\}_{i\in\mathbb Z}$ are the same as settings \ref{i} and \ref{ii} for the scalar-on-function linear regression.
For the error process, we first generate i.i.d.~Gaussian processes $\{\zeta_i\}$ with covariance function $\cov\{\zeta_i(s),\zeta_i(t)\}=\delta(s,t)$, where $\delta$ denotes the delta function, and take $\e_i=c_2\cdot(\zeta_i+\upsilon_{i,1}\zeta_{i-1}+\upsilon_{i,2}\zeta_{i-2})$, for $1\leq i\leq n$, where $c_2>0$ is chosen such that $\var(\|\e_i\|_{L^2})/\var(\|X_i\|_{L^2})=0.3$, and $\upsilon_{i,j}\iidsim \text{unif}(-1/\sqrt{2},1/\sqrt{2})$, for $i\in\mathbb Z$ and $j=1,2$.

In Figures~\ref{fig:sf} and \ref{fig:ff} we 
display the empirical rejection probabilities 
of the tests \eqref{det5} and \eqref{det6} for the scalar response case and the functional response case, respectively, where we vary the value of the threshold $\Delta$; we took $\nu_0=1/2$ and chose $\omega$ as the discrete uniform distribution on the $\{\nu_q\}_{q=1}^Q$, with $Q=25$, where the $\nu_q$ is defined in \eqref{nuq}; for the sample sizes we took $n=50$ and $200$ observations. The results confirm our theoretical 
findings in Theorem \ref{thm:one} and \ref{thm:f}
and can be summarized as follows:
\begin{itemize}
\item[(1)] The tests provide a reasonable approximation of the nominal level $\alpha$ when $\Delta=\|\beta_0\|_{L^2}^2$, for both scalar and functional response cases.
\item[(2)] The rejection probabilities are close to zero when $\Delta<\|\beta_0\|_{L^2}^2$ (interior of the null hypothesis).
\item[(3)] In the cases where $\Delta>\|\beta_0\|_{L^2}^2$ (interior of the alternative), the empirical rejection probabilities increases with $\Delta$, and larger sample size ($n=200$) attained higher empirical rejection probabilities.
\end{itemize}

We conclude this section with an investigation 
of the sensitivity of the tests with respect to
the choice of the parameter $\nu_0$ and the measure $\omega$ in the definition of the statistic $\hat{ \mathbb{V}}_n$ in \eqref{vn}.
For the sake of brevity we restrict ourselves to the scalar-on-function model. In
Figure~\ref{fig:s11} we display the empirical rejection probabilities of the test \eqref{det5} for the relevant hypotheses in \eqref{rh} under settings~(S1) and (i), at nominal level $\alpha=0.05$, using the sample size $n=200$ and different choices of $\nu_0$ and $Q$ in \eqref{nuq}: $\nu_0=1/4,1/2$ and $Q=15,25,35$. The results show that the empirical rejection probabilities remain relatively stable despite different values of parameters $\nu_0$ and $Q$ are used.

\begin{figure}[h]
\centering
(S1) $\vcenter{\hbox{\includegraphics[width=6cm]{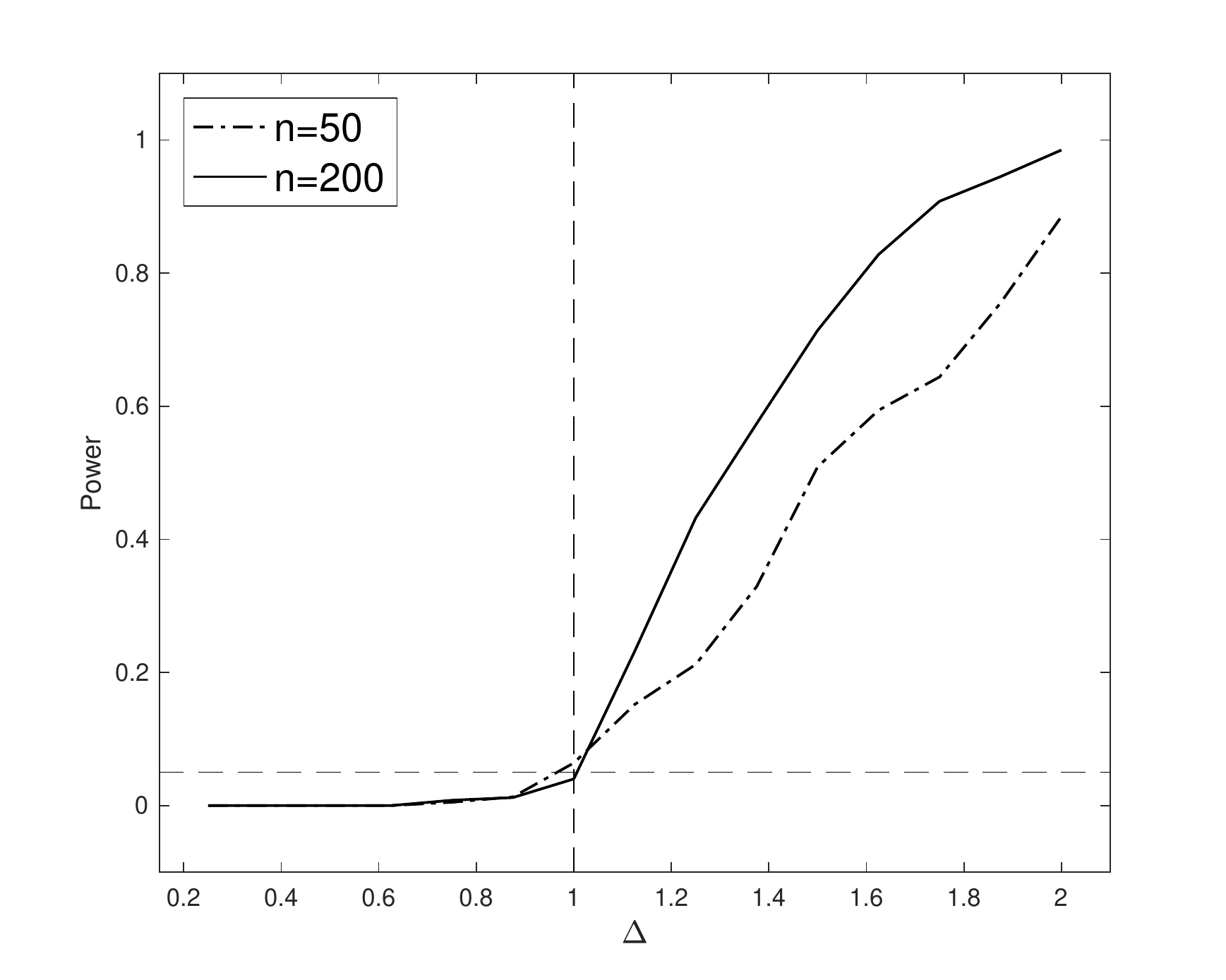}\includegraphics[width=6cm]{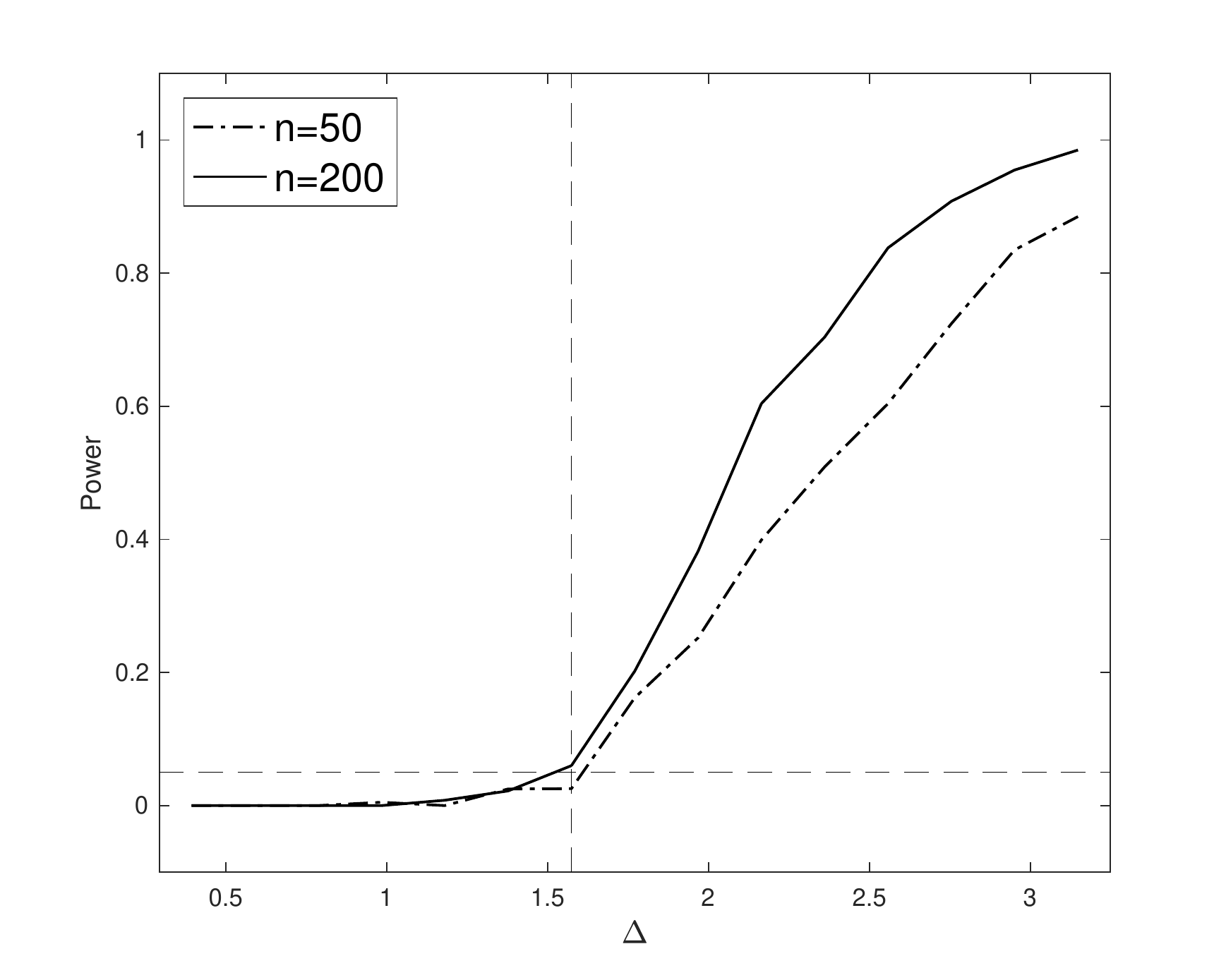}}}$~\\
(S2) $\vcenter{\hbox{\stackunder[5pt]{\includegraphics[width=6cm]{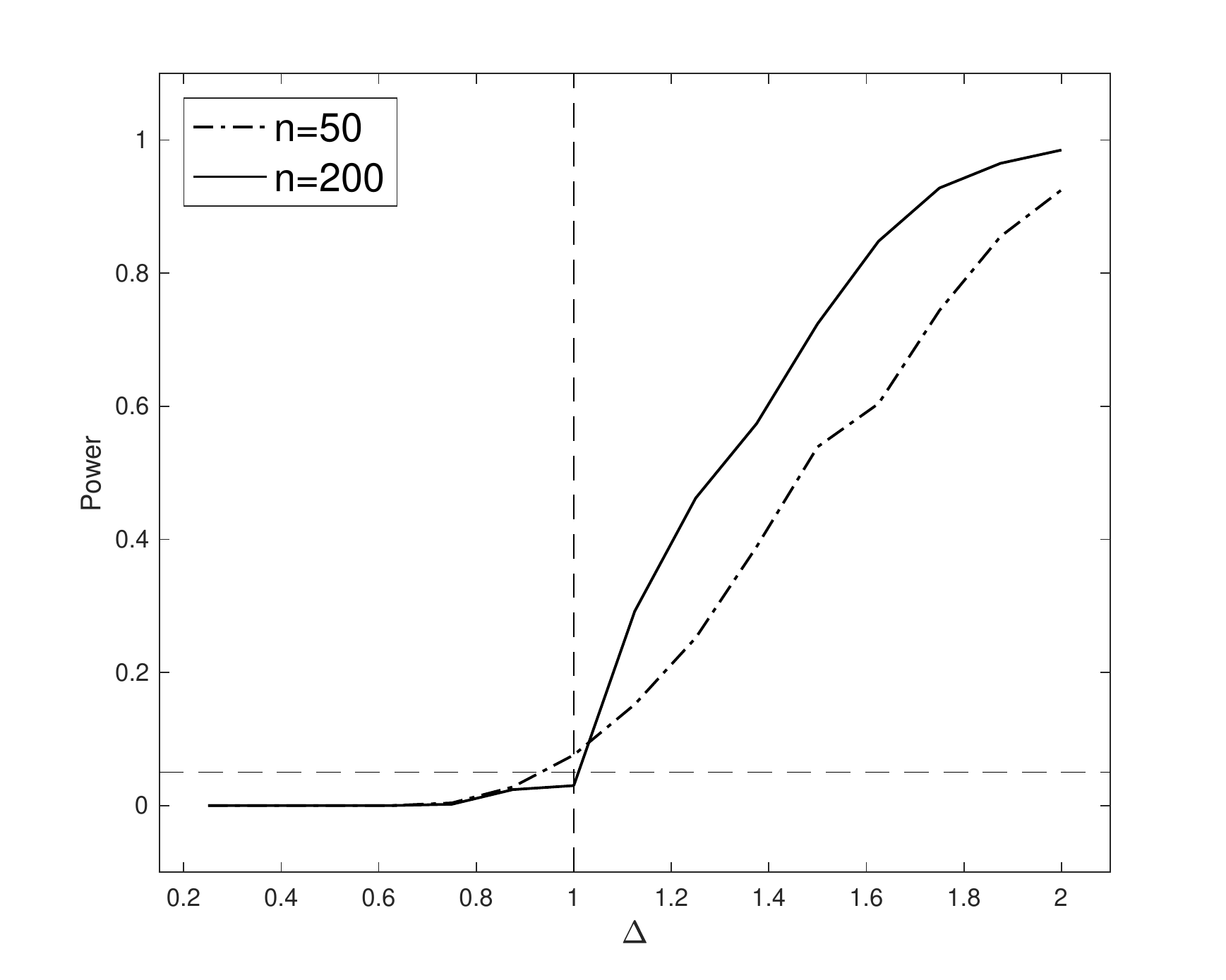}}{(i)}\stackunder[5pt]{\includegraphics[width=6cm]{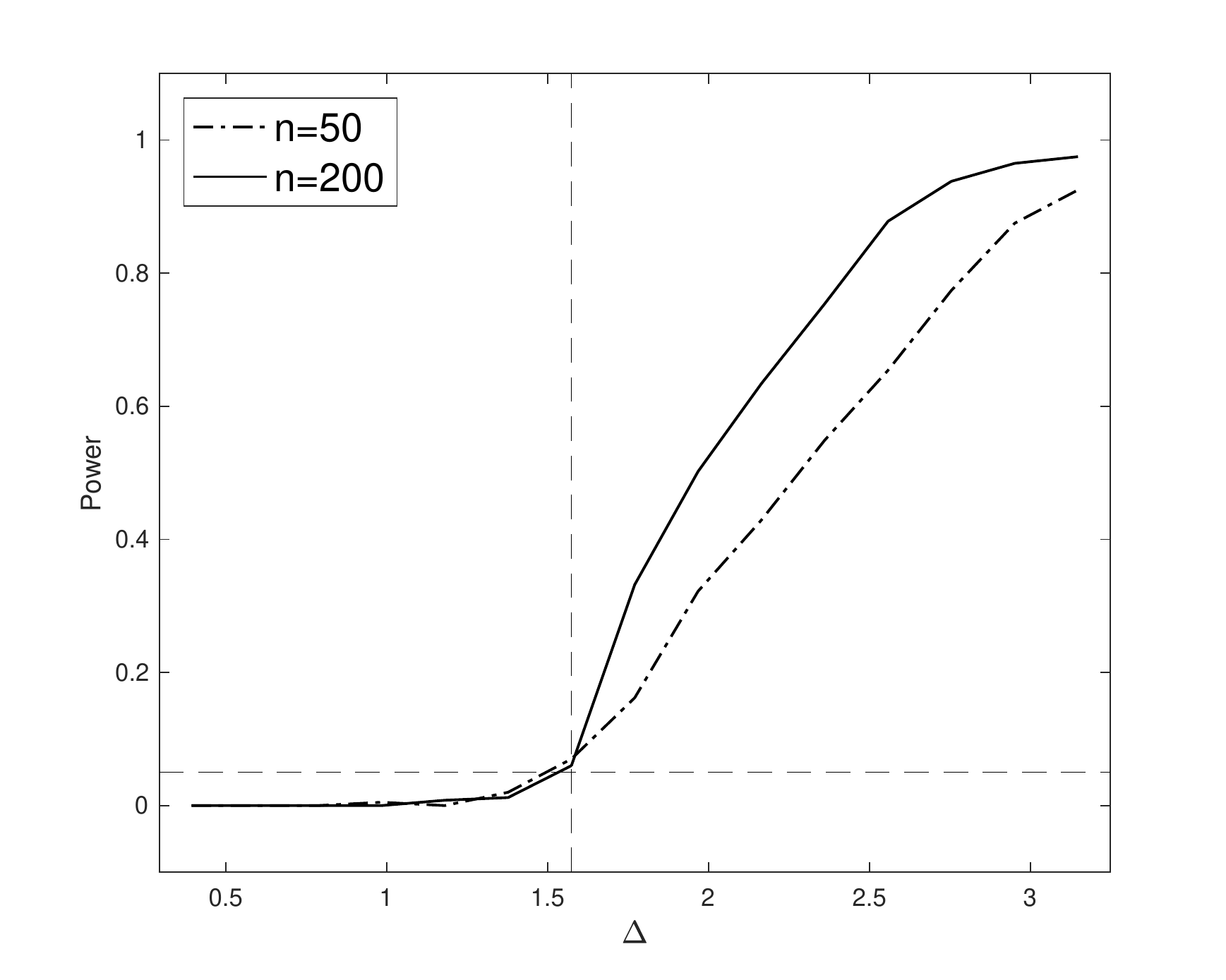}}{(ii)}}}$
\caption{\it Empirical rejection probabilities  
of the test \eqref{det5} for the relevant hypotheses in \eqref{rh} at nominal level $\alpha=0.05$, under settings~(S1) and (S2) (first and second row), and settings~(i) and (ii) (first and second column).
The horizontal dashed line is the nominal level 0.05; the vertical dashed line is the true value $d_0=\int_0^1|\beta(s)|^2ds$.\label{fig:sf}}
\end{figure}

\begin{figure}[h]
\centering
(F1) $\vcenter{\hbox{\includegraphics[width=6cm]{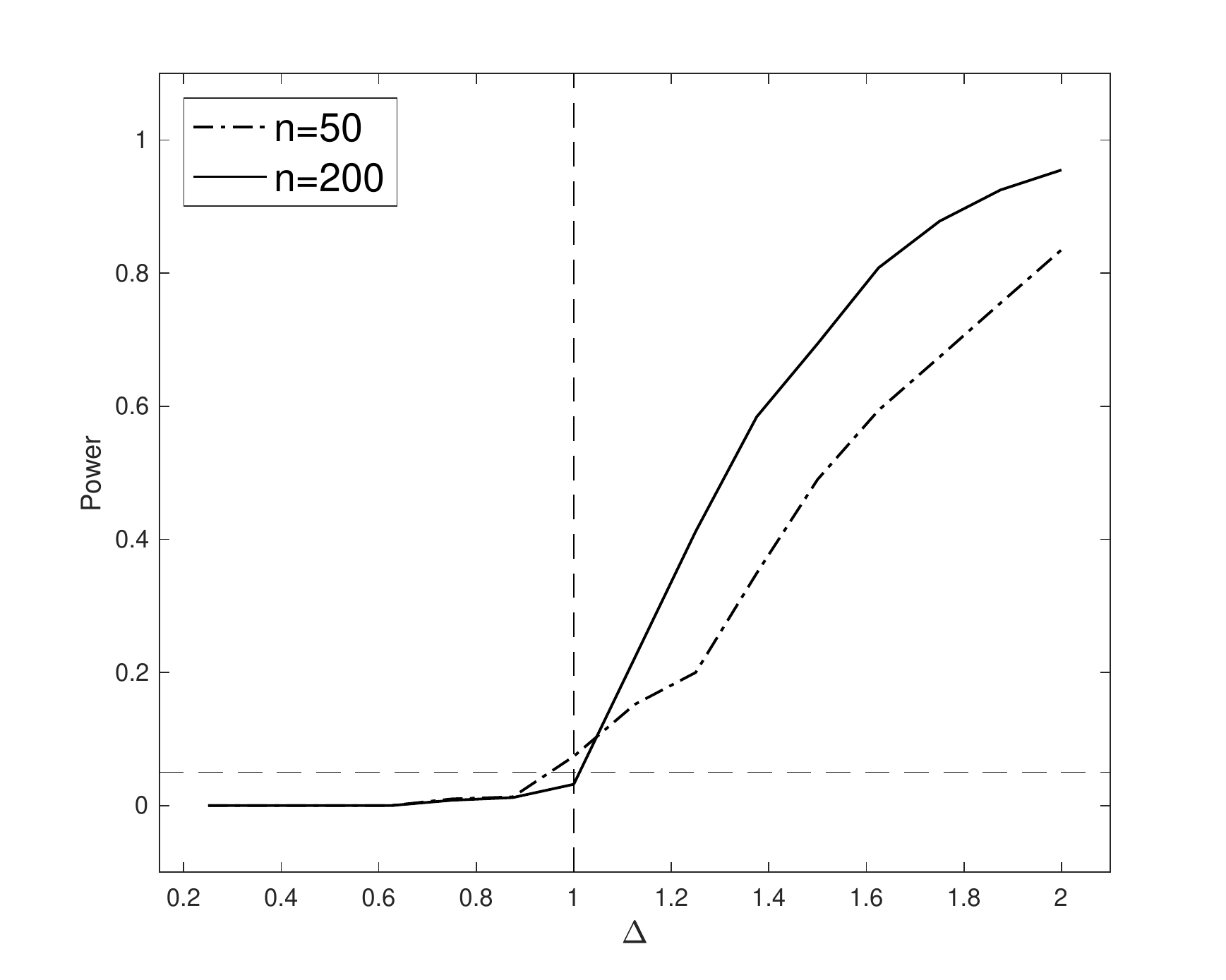}\includegraphics[width=6cm]{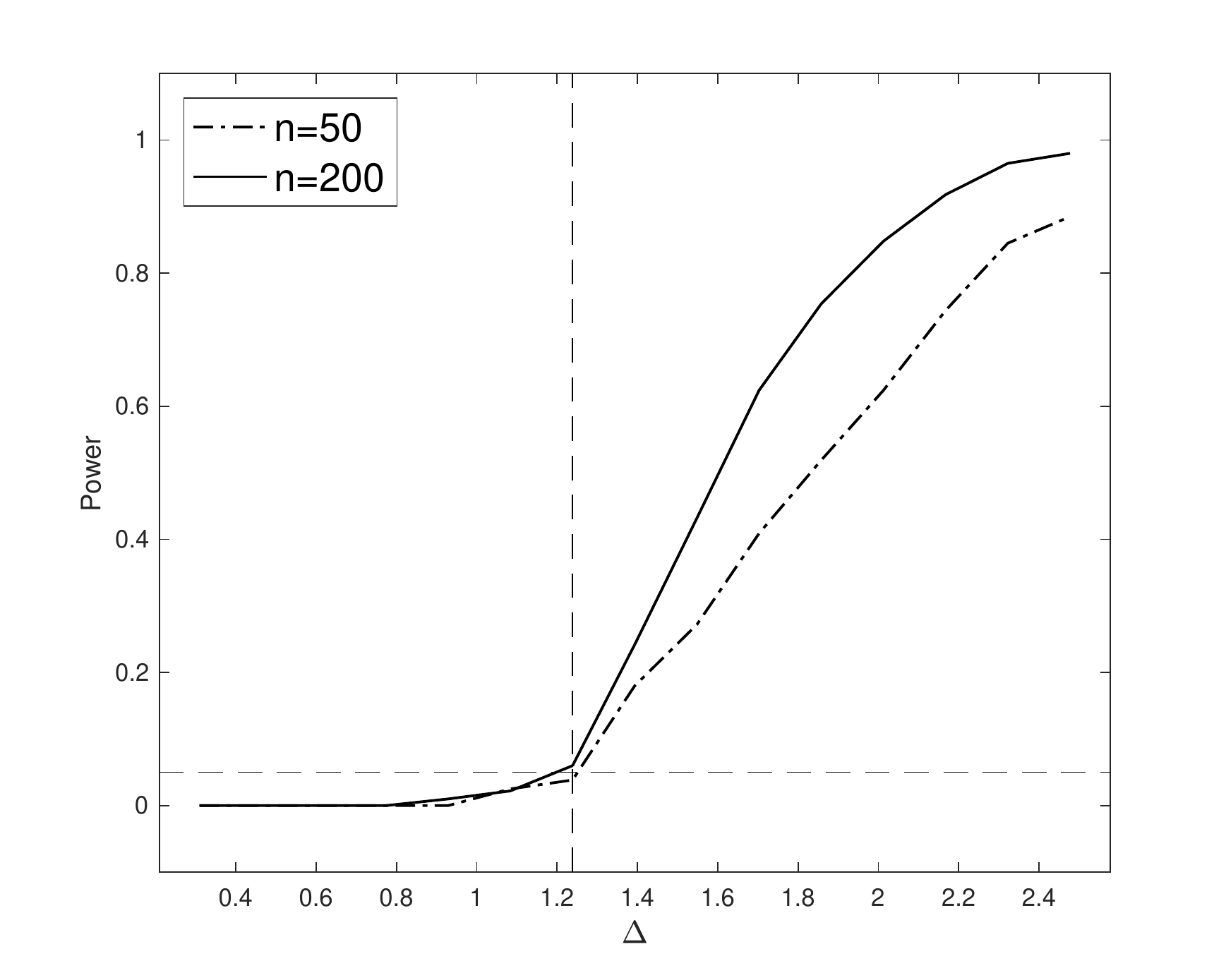}}}$~\\
(F2) $\vcenter{\hbox{\stackunder[5pt]{\includegraphics[width=6cm]{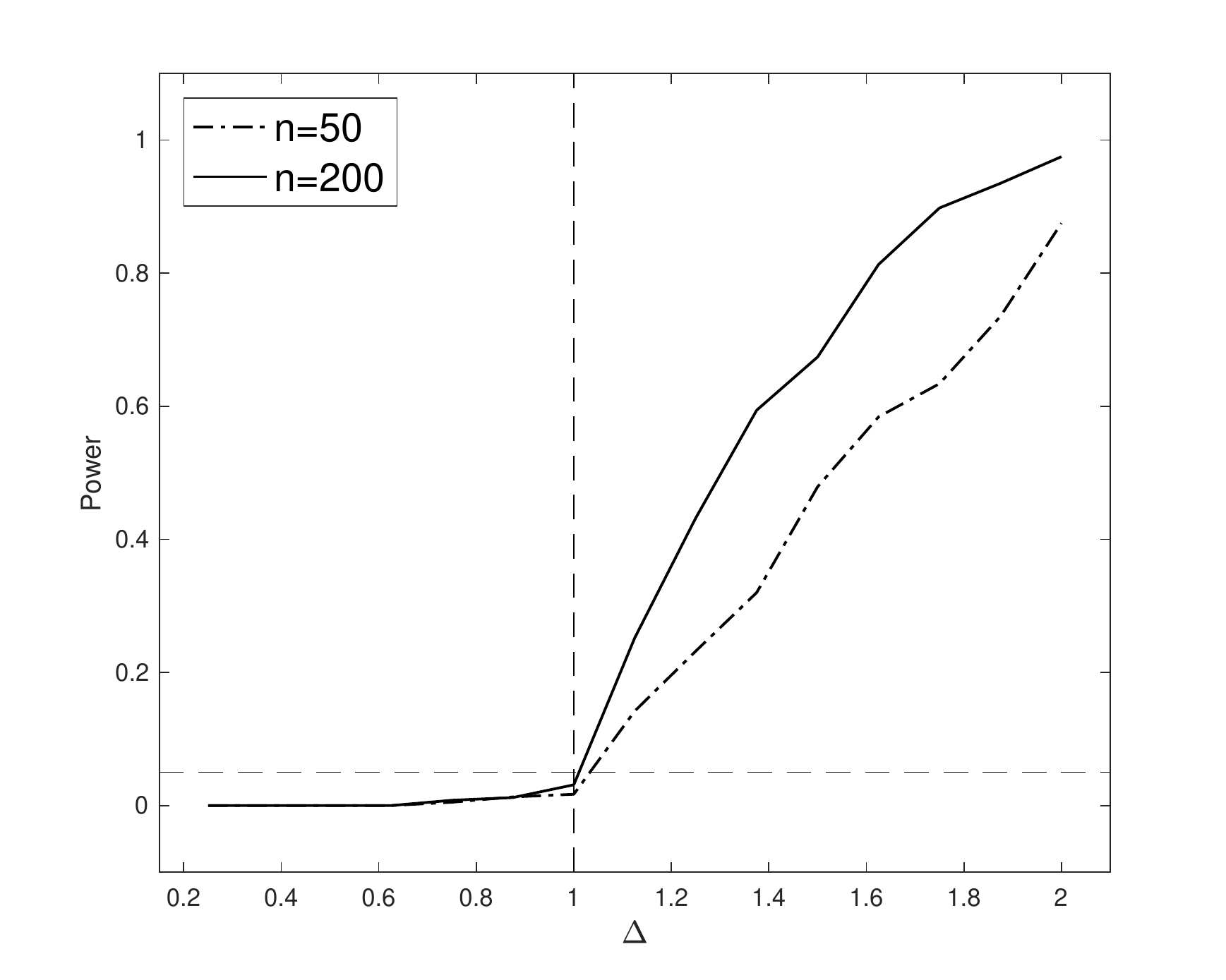}}{(i)}\stackunder[5pt]{\includegraphics[width=6cm]{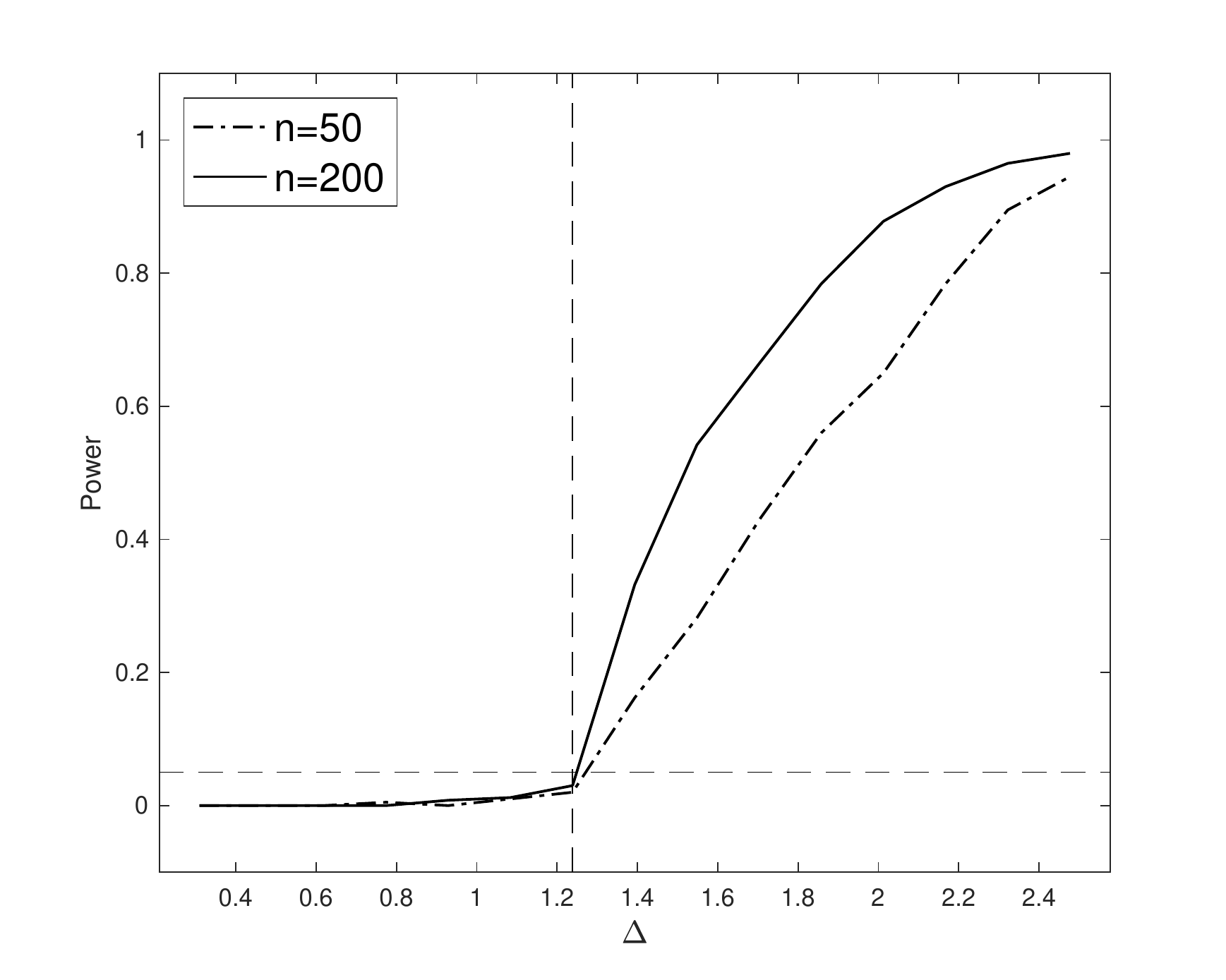}}{(ii)}}}$
\caption{\it Empirical rejection probabilities 
of the test \eqref{det6} for the relevant hypotheses in \eqref{hf} at nominal level $\alpha=0.05$, under settings~(F1) and (F2) (first and second row), and settings~(i) and (ii) (first and second column).
The horizontal dashed line is the nominal level 0.05; the vertical dashed line is the true value $d_0^f=\int_0^1\int_0^1|\beta(s,t)|^2dsdt$.\label{fig:ff}}
\end{figure}

\begin{figure}[h]
\centering
\includegraphics[width=6cm]{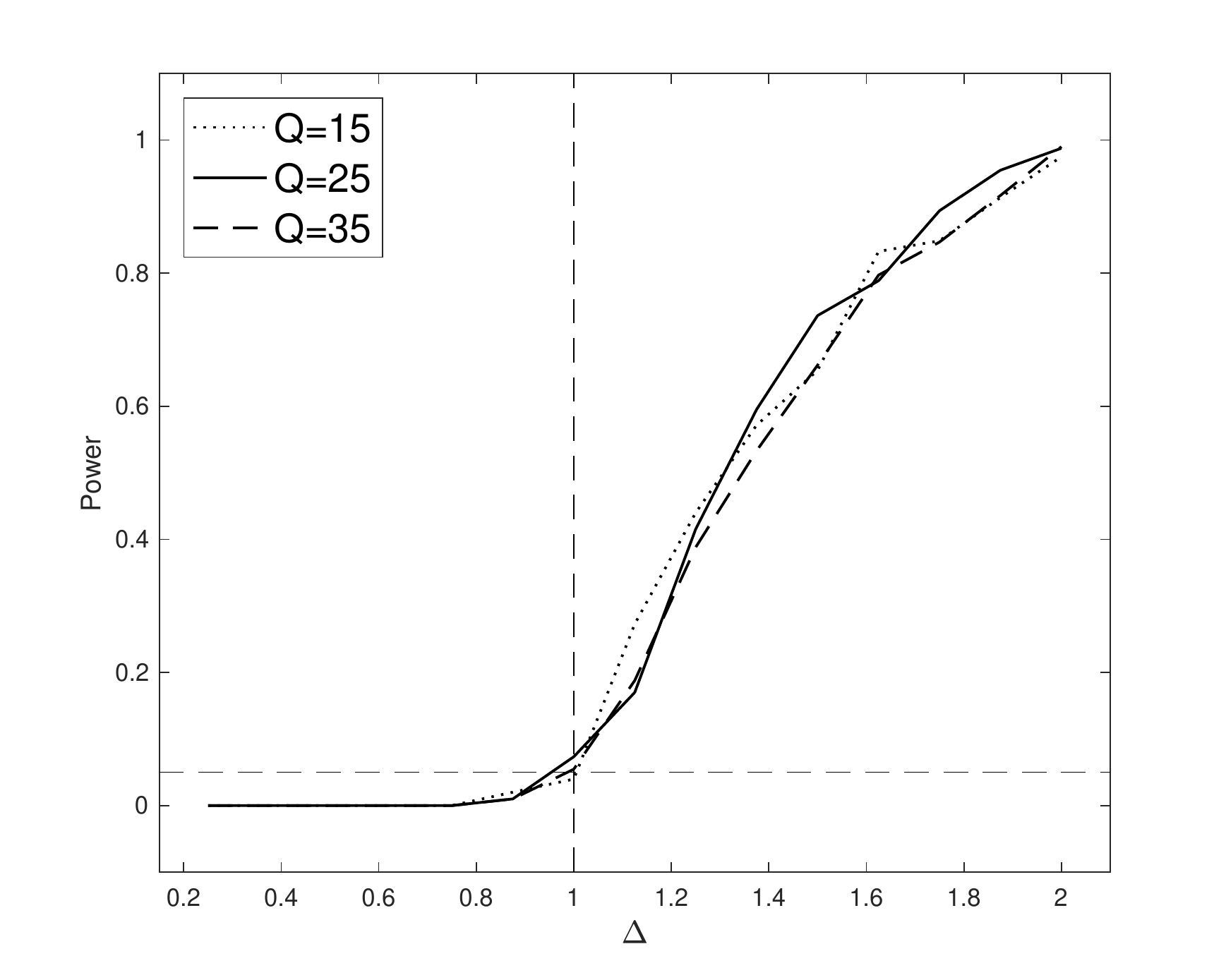}\includegraphics[width=6cm]{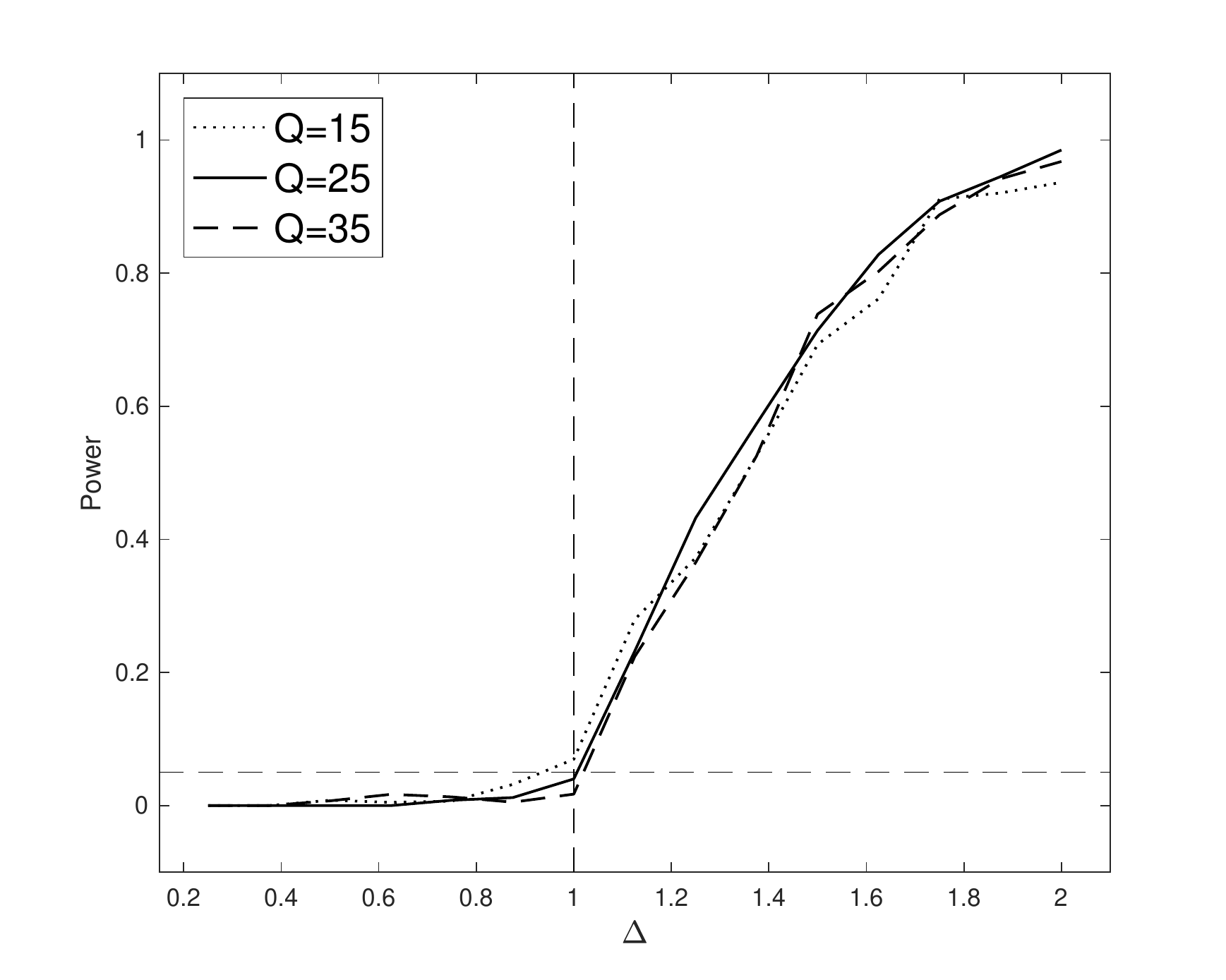}
\caption{\it Empirical rejection probabilities  
of the test \eqref{det5} for the relevant hypotheses in \eqref{rh} at nominal level $\alpha=0.05$, using sample size $n=200$, under settings~(S1) and (i), with $\nu_0=1/4$ (left), $\nu_0=1/2$ (right) and with $Q=15,25,35$ in \eqref{nuq}. 
The horizontal dashed line is the nominal level 0.05; the vertical dashed line is the true value $d_0=\int_0^1|\beta(s)|^2ds$.\label{fig:s11}}
\end{figure}

\clearpage

\subsection{Data example}\label{sec:realdata}

We applied our method to the Australian weather data available from the website of Australian Bureau of Meteorology at \nolinkurl{http://www.bom.gov.au/climate/data/}. We downloaded the daily observations of the maximum temperature and rainfall at the Melbourne Airport station (Station number 086282; Latitude: $37.67^\circ$S; Longitude: $144.83^\circ$E) from the year 1971 to 2020.
In this case, for $1\leq i\leq 50$, the predictor curve $X_i$ and the response curve $Y_i$ are supported on $[0,1]$ and represent the daily maximum temperature and daily rainfall for the $i$-th year, respectively. Following \cite{dettejrssb2020}, we obtained the curves $X_i$ and $Y_i$ by projecting the daily observations onto the linear space spanned by $49$ Fourier basis functions on $[0,1]$. We considered the relevant hypotheses \eqref{hf}, and took $\nu_0=1/2$ and $Q=10$ in \eqref{nuq}. Table~\ref{tab:weather} displays the decisions of our test with different values of $\Delta$ and nominal level $\alpha=0.10$, $0.05$ and $0.01$. 
For example, the largest  value of $\Delta$ such that the
test \eqref{det6} 
rejects
the null hypothesis in \eqref{hf}
at level $\alpha =0.05$ is given by $\Delta = 0.32$.
Alternatively, if one wants to avoid the specification 
of the threshold $\Delta$  for a test (see the discussion in Remark~\ref{delta}), we can provide 
a one-sided and two-sided confidence interval for the quantity 
$\int_0^1\int_0^1 |\beta_0(s,t)|^2 ds dt$, which are  given by $[0,0.44]$ and $[0.18,0.52]$, respectively 
(confidence level $95\%$).

\begin{table}[h]
\caption{Decisions of the test  \eqref{det6} for the relevant hypotheses \eqref{hf} with different values of $\Delta$ and confidence levels using the Canadian weather data, where ``R" stands for rejection of the null hypothesis and ``-" stands for not rejecting the null hypothesis.}
\centering
\begin{tabular}{c|cccccc} 
\cline{1-7}
$\Delta$&0.19&0.20&0.32&0.33& 0.41&0.42 \\  
\cline{1-7}
1\% &  R & -& - & - & - & - \\
5\%& R  & R  &    R&      - &- &  -\\
10\%&    R &     R & R &     R & R &     - \\
\cline{1-7}
\end{tabular}
\label{tab:weather}
\end{table}


\baselineskip=13pt
\vspace{1cm}

\begin{center}
{\large\textbf{Acknowledgements}}
\end{center}

\noindent 
The authors would like to thank Tim Kutta for helpful discussions and pointing out important references.

\begingroup
\renewcommand{\section}[2]{\subsection#1{#2}}
{\centering

}

\endgroup

\clearpage

\appendix

\setcounter{page}{1}

\renewcommand{\thesection}{\Alph{section}}

\pagestyle{fancy}
\fancyhf{}
\rhead{\textbf{\thepage}}
\lhead{\textbf{Online supplement}}

\begin{center}
{\bf \large Supplementary material for ``An  RKHS approach for pivotal  inference in functional linear regression"} 
\end{center}

\baselineskip=17pt
\begin{center}
Holger Dette, Jiajun Tang\\
Fakult\"at f\"ur Mathematik, Ruhr-Universit\"at Bochum, Bochum, Germany
\end{center}

\vspace{.5cm}

In this supplementary material we provide technical details of our theoretical results. 
In Section~\ref{app:proof} we provide the proofs of our theorems in our main article. In Section~\ref{app:aux:lem} we provide supporting lemmas that are used in the proofs in Section~\ref{app:proof}.
In the sequel, we use $c$ to denote a generic positive constant that might differ from line to line.

%
%
%

\section{Theoretical details of main results}\label{app:proof}

\subsection{Proof of Theorem~\ref{thm:bahadur}}\label{app:thm:bahadur}

We first prove in Lemma~\ref{lem:hatbeta} below the uniform convergence rate of the 
sequential RKHS estimator $\hat\beta_{n,\lambda}(\cdot,\nu)$ for the slope function
$\beta_0$  defined in \eqref{hatbeta} w.r.t~the $\|\cdot\|_K$-norm.


\begin{lemma}\label{lem:hatbeta}
Under Assumptions~\ref{a1}--\ref{a:rate}, we have, for any fixed (but arbitrary) $\nu_0\in(0,1]$,
\begin{align*}
\sup_{\nu\in[\nu_0,1]}\big\|\hat\beta_{n,\lambda}(\cdot,\nu)-\beta_0\big\|_K=O_p(\lambda^{1/2}+n^{-1/2}\lambda^{-(2a+1)/(4D)})\,.
\end{align*}

\end{lemma}


\begin{proof}
Define $S_{\lambda,\nu}(\beta)=\E\{S_{n,\lambda,\nu}(\beta)\}$ and $\mathcal D S_{\lambda,\nu}(\beta)=\E\{\mathcal D S_{n,\lambda,\nu}(\beta)\}$.
Moreover,  from \eqref{sn} it also follows that
\begin{equation}\label{s}
\begin{split}
&S_{\lambda,\nu}(\beta)=\E\{S_{n,\lambda,\nu}(\beta)\}=-\E\bigg[\bigg\{Y_0-\int_0^1X_0(s)\,\beta(s)\,ds\bigg\}\tau_{\lambda}(X_0)\bigg]+W_\lambda(\beta)\,,\\
&\mathcal D S_{\lambda,\nu}(\beta)\beta_1=\E\{\mathcal D S_{n,\lambda,\nu}(\beta)\beta_1\}=\E\bigg[\bigg\{\int_0^1X_0(s)\,\beta(s)\,ds\bigg\}\tau_{\lambda}(X_0)\bigg]+W_\lambda(\beta)\,.
\end{split}
\end{equation}
Recall from \eqref{sn} that, for $\nu\in[0,1]$ and for any $\beta_1,\beta_2\in\H$,
\begin{align*}
\l\mathcal D S_{\lambda,\nu}(\beta)\beta_1,\beta_2\r_K&=\E\Big\{\l\tau_{\lambda}(X_i),\beta_1\r_K\,\l\tau_{\lambda}(X_i),\beta_2\r_K\Big\}+ \l W_\lambda(\beta_1),\beta_2\r_K\\
&= \E\Big\{\l X_i,\beta_1\r_{L^2}\,\l X_i,\beta_2\r_{L^2}\Big\}+ \l W_\lambda(\beta_1),\beta_2\r_K\\
&=V(\beta_1,\beta_2)+\lambda J(\beta_1,\beta_2)=\l\beta_1,\beta_2\r_K=\l id(\beta_1),\beta_2\r_K\,,
\end{align*}
which implies that 
\begin{align}\label{prop:id}
\mathcal D S_{\lambda,\nu}(\beta)=id\,,
\end{align}
where $id$ denotes the identity operator on $\H$. 
Since the second-order Fr\'echet derivative $\mathcal D^2S_{\lambda,\nu}$ vanishes, there exists a unique solution to the estimating equation $S_{\lambda,\nu}(\beta)=0$. In addition, by the mean value theorem and \eqref{prop:id}, for any $\beta\in\H$, 
\begin{align*}
S_{\lambda,\nu}(\beta)= S_{\lambda,\nu}(\beta_0)+\mathcal D S_{\lambda,\nu}(\beta)(\beta-\beta_0)= S_{\lambda,\nu}(\beta_0)+(\beta-\beta_0)\,.
\end{align*}
Let $\beta_{\lambda,\nu}=\beta_0-S_{\lambda,\nu}(\beta_0)$. We deduce that $S_{\lambda,\nu}(\beta_{\lambda,\nu})=S_{\lambda,\nu}(\beta_0)+(\beta_{\lambda,\nu}-\beta_0)=0$,
so that $\beta_{\lambda,\nu}$ is the unique solution to the estimating equation $S_{\lambda,\nu}(\beta)=0$. Moreover, in view of \eqref{s}, for any $\nu\in[\nu_0,1]$,
\begin{align}\label{0}
&\|\beta_{\lambda,\nu}-\beta_0\|_K=\|S_{\lambda,\nu}(\beta_0)\|_K\notag\\
&=\bigg\|-\E\bigg[\bigg\{Y_0-\int_0^1X_0(s)\,\beta_0(s)\,ds\bigg\}\tau_{\lambda}(X_0)\bigg]+\,W_\lambda(\beta_0)\bigg\|_K=\|W_\lambda(\beta_0)\|_K\,.
\end{align}
 Therefore, by the Cauchy-Schwarz inequality, we deduce that
\begin{align}\label{deltalambdabeta}
&\sup_{\nu\in[\nu_0,1]}\big\|\beta_{\lambda,\nu}-\beta_0\big\|_K= \Vert W_\lambda(\beta_0)\Vert_K=\sup_{\Vert\gamma\Vert_K=1}|\l W_\lambda(\beta_0),\gamma\r_K|=\sup_{\Vert\gamma\Vert_K=1}\lambda |J(\beta_0,\gamma)|\notag\\
&\leq\sup_{\Vert\gamma\Vert_K=1}\Big\{\sqrt{\lambda J(\beta_0,\beta_0)}\sqrt{\lambda J(\gamma,\gamma)}\Big\}\leq\sup_{\Vert\gamma\Vert_K=1}\Big\{\sqrt{\lambda J(\beta_0,\beta_0)}\,\Vert\gamma\Vert_K\Big\}\notag\\
&=\sqrt{\lambda J(\beta_0,\beta_0)}=O(\lambda^{1/2})\,.
\end{align}
Since $\Vert\hat\beta_{n,\lambda}(\cdot,\nu)-\beta_0\Vert_K\leq\Vert\beta_{\lambda,\nu}-\beta_0\Vert_K+\Vert\hat\beta_{n,\lambda}(\cdot,\nu)-\beta_{\lambda,\nu}\Vert_K$, we then proceed to show the rate of $\Vert\hat\beta_{n,\lambda}(\cdot,\nu)-\beta_{\lambda,\nu}\Vert_K$. For $\nu\in[\nu_0,1]$, let 
\begin{align*}
F_{n,\nu}(\beta)=\beta-S_{n,\lambda,\nu}(\beta_{\lambda,\nu}+\beta)\,.
\end{align*}
Since $\mathcal D^2S_{\lambda,\nu}$ vanishes we obtain from  \eqref{prop:id}
\begin{align}\label{th}
F_{n,\nu}(\beta)=\mathcal DS_{\lambda,\nu}(\beta_{\lambda,\nu})\beta-S_{n,\lambda,\nu}(\beta_{\lambda,\nu}+\beta)    =I_{1,n,\nu}(\beta)+I_{2,n,\nu}(\beta)-S_{n,\lambda,\nu}(\beta_{\lambda,\nu})\,,
\end{align}
where  $\mathcal DS_{n,\lambda,\nu}$  is defined in \eqref{sn} and 
\begin{align}
&I_{1,n,\nu}(\beta)=-\{S_{n,\lambda,\nu}(\beta_{\lambda,\nu}+\beta)-S_{n,\lambda,\nu}(\beta_{\lambda,\nu})-\mathcal DS_{n,\lambda,\nu}(\beta_{\lambda,\nu})\beta\}\,,\notag\\
&I_{2,n,\nu}(\beta)=-\{\mathcal DS_{n,\lambda,\nu}(\beta_{\lambda,\nu})\beta-\mathcal DS_{\lambda,\nu}(\beta_{\lambda,\nu})\beta\}\,.\label{i1}
\end{align}
First, for $I_{1,n,\nu}(\beta)$ in \eqref{i1}, in view of $S_{n,\lambda,\nu}$ and $\mathcal DS_{n,\lambda,\nu}$ defined in \eqref{sn}, we find
\begin{align}\label{zero}
I_{1,n,\nu}(\beta)&=\frac{1}{\lfloor n\nu\rfloor}\sum_{i=1}^{\lfloor n\nu\rfloor}\bigg[Y_i-\int_0^1\{\beta_{\lambda,\nu}(s)+\beta(s)\}\,X_i(s)\,ds\bigg]\tau_{\lambda}(X_i)\notag\\
&-\frac{1}{\lfloor n\nu\rfloor}\sum_{i=1}^{\lfloor n\nu\rfloor}\bigg\{Y_i-\int_0^1\beta(s)\,X_i(s)\,ds\bigg\}\,\tau_{\lambda}(X_i)\notag\\
&+\frac{1}{\lfloor n\nu\rfloor}\sum_{i=1}^{\lfloor n\nu\rfloor}\bigg\{\int_0^1\beta_{\lambda,\nu}(s)\,X_i(s)\,ds\bigg\}\,\tau_{\lambda}(X_i)=0\,.
\end{align}
For the second term $I_{2,n,\nu}(\beta)$ in \eqref{i1}, denote the event
\begin{align}\label{enc}
\EE_n(c)=\Big\{\max_{1\leq i\leq n}\|X_i\|_{L^2}\leq c\log n\Big\}\,.
\end{align}
By Assumption~\ref{a:subg} and Markov's inequality, if we take $c>3/\varpi>0$, we have
\begin{align*}
\P\{\EE_n^{\rm c}(c)\}\leq n\,\P\big(\|X_0\|_{L^2}\leq c\log n\big)\leq n^{1-c\varpi}\,\E\{\exp(\varpi\|X_0\|_{L^2})\}=o(n^{-2})\,.
\end{align*}
Then, it suffices to confine the proof on the event $\EE_n(c)$. In view of \eqref{sn} and \eqref{s},
\begin{align}\label{i22nv}
I_{2,n,\nu}(\beta)&=\mathcal DS_{n,\nu}(\beta_{\lambda,\nu})\beta-\mathcal DS_{\nu}(\beta_{\lambda,\nu})\beta\notag\\
&=-\frac{1}{\lfloor n\nu\rfloor}\sum_{i=1}^{\lfloor n\nu\rfloor}\Bigg[\tau_{\lambda}(X_i)\int_0^1\beta(s)X_i(s)ds-\E\bigg\{\tau_{\lambda}(X_i)\int_0^1\beta(s)X_i(s)ds\bigg\}\Bigg]\notag\\
&=I_{2,1,n,\nu}(\beta)+I_{2,2,n,\nu}(\beta)\,,
\end{align}
where
\begin{align} \nonumber
&I_{2,1,n,\nu}(\beta)=
\E\bigg[\tau_{\lambda}(X_i)\int_0^1\beta(s)X_i(s)ds\times\one\{\EE_n^{\rm c}(c)\}\bigg]\,,\\
&I_{2,2,n,\nu}(\beta)=-\frac{1}{\lfloor n\nu\rfloor}\sum_{i=1}^{\lfloor n\nu\rfloor}\Bigg(\tau_{\lambda}(X_i)\int_0^1\beta(s)X_i(s)ds\,\one\{\EE_n(c)\}
\label{det7}
\\
& ~~~~~ ~~~~~~~~~ - \E\bigg[\tau_{\lambda}(X_i)\int_0^1\beta(s)X_i(s)ds\,\one\{\EE_n(c)\}\bigg]\Bigg)\,. \nonumber 
\end{align}
For the first term $I_{2,1,n,\nu}$ in \eqref{i22nv}, by the Cauchy-Schwarz inequality and Lemma~\ref{lem:s4}, we have
\begin{align}\label{e}
\|I_{2,1,n,\nu}(\beta)\|_K&=\Bigg\|\E\bigg[\tau_{\lambda}(X_0)\int_0^1\beta(s)X_0(s)ds\times\one\{\EE_n^{\rm c}(c)\}\bigg]\Bigg\|_K\notag\\
&\leq\sup_{\|\gamma\|_K=1}\bigg\l\gamma,\E\bigg[\tau_{\lambda}(X_0)\int_0^1\beta(s)X_0(s)ds\times\one\{\EE_n^{\rm c}(c)\}\bigg]\bigg\r_K\notag\\
&=\sup_{\|\gamma\|_K=1}\E\bigg[\int_0^1\gamma(s)X_0(s)ds\times\int_0^1\beta(s)X_0(s)ds\times\one\{\EE_n^{\rm c}(c)\}\bigg]\notag\\
&\leq \Big[\P\{\EE_n^{\rm c}(c)\}\Big]^{1/2}\times\E\big(\l X_0,\beta\r_{L^2}^4\big)^{1/4}\times\sup_{\|\gamma\|_K=1}\E\big(\l X_0,\gamma\r_{L^2}^4\big)^{1/4}\notag\\
&\leq c\,\Big[\P\{\EE_n^{\rm c}(c)\}\Big]^{1/2}\times\E\big(\l X_0,\beta\r_{L^2}^2\big)^{1/2}\times\sup_{\|\gamma\|_K=1}\E\big(\l X_0,\gamma\r_{L^2}^2\big)^{1/2}\notag\\
&= c\,\Big[\P\{\EE_n^{\rm c}(c)\}\Big]^{1/2}\times\E\big(\l\tau_{\lambda}( X_0),\beta\r_{K}^2\big)^{1/2}\times\sup_{\|\gamma\|_K=1}\E\big(\l\tau_{\lambda}( X_0),\gamma\r_{K}^2\big)^{1/2}\notag\\
&\leq c\,\Big[\P\{\EE_n^{\rm c}(c)\}\Big]^{1/2}\times\E\|\tau_{\lambda}(X_0)\|_K^2\times\|\beta\|_K\notag\\
&\leq c\,\times o(n^{-1})\times\lambda^{-1/(2D)}\|\beta\|_K=o(1)\,\|\beta\|_K\,.
\end{align}
Therefore, we deduce that
\begin{align}\label{i21}
\sup_{\nu\in[\nu_0,1]}\|I_{2,1,n,\nu}(\beta)\|_K=o(1)\,\|\beta\|_K\,.
\end{align}
For the second term $I_{2,2,n,\nu}(\beta)$ in \eqref{i22nv}, for $a,D$ in Assumption~\ref{a201} and $c_K$ in Lemma~\ref{lem:3.1} in Section~\ref{app:aux:lem}, let $p_n=c_K^{-2}\lambda^{(2a+1)/(2D)-1}$. In order to apply Lemma~\ref{lem:3.1} in Section~\ref{app:aux:lem},
we shall rescale $\beta$ such that the $L^2$-norm of its rescaled version is bounded by $1$, that is
\begin{align}\label{tildeb}
\tilde\beta=\left\{
\begin{array}{ll}
\big(c_K\lambda^{-(2a+1)/(4D)}\Vert\beta\Vert_K\big)^{-1}\beta&\text{ if }\beta\neq0\,,\\
0&\text{ if }\beta=0\,
\end{array}\right.
\end{align}
where    $c_K$  is the constant in Lemma~\ref{lem:3.1}.
We have $\Vert\tilde\beta\Vert_{L^2}\leq c_K\lambda^{-(2a+1)/(4D)}\Vert\tilde\beta\Vert_K\leq 1$, since $\Vert\tilde\beta\Vert_K\leq (c_K\lambda^{-(2a+1)/(4D)})^{-1}$ in view of Lemma~\ref{lem:3.1}. In addition, observing \eqref{inner}, it follows that 
$$
J(\tilde\beta,\tilde\beta)\leq\lambda^{-1}\Vert\tilde\beta\Vert_K^2\leq c_K^{-2}\lambda^{(2a+1)/(2D)-1}=p_n.
$$
Therefore, 
\begin{equation}
\label{det8}
\tilde\beta\in\mathcal{F}_{p_n}:=\{\beta\in\H:\Vert\beta\Vert_{L^2}\leq 1,J(\beta,\beta)\leq p_n\}. \end{equation}
For the event $\EE_n(c)$ defined in \eqref{enc} and for any $\beta\in\H$, let
\begin{align}\label{hnk}
%
&\tilde H_{n,\nu}(\beta)=\frac{1}{\sqrt{\lfloor n\nu\rfloor}}\sum_{i=1}^{{\lfloor n\nu\rfloor}}\bigg(\tau_{\lambda}(X_i)\l\beta,X_i\r_{L^2}\,\one\{\EE_n(c)\}-\E\Big[\tau_{\lambda}(X_i)\l\beta,X_i\r_{L^2}\,\one\{\EE_n(c)\}\Big]\bigg)\,.
\end{align}
Note  that, for $\nu\in[\nu_0,1]$,
\begin{align}\label{tildehh}
\sup_{\beta\in\mathcal F_{p_n}}\|\tilde H_{n,\nu}(\beta)\|_K&=\frac{1}{\sqrt{\lfloor n\nu\rfloor/n}}\,\sup_{\beta\in\mathcal F_{p_n}}\|H_{n,\lfloor n\nu\rfloor}(\beta)\|_K\notag\\
&\leq \nu_0^{-1/2}\max_{1\leq k\leq n}\sup_{\beta\in\mathcal F_{p_n}}\|H_{n,k}(\beta)\|_K\times\{1+o(1)\}\,,
\end{align}
where, for the $\EE_n(c)$ in \eqref{enc}, $H_{n,k}$ is defined  by 
\begin{align}\label{hnk2}
%
&H_{n,k}(\beta)=\frac{1}{\sqrt{n}}\sum_{i=1}^{k}\Bigg(\tau_{\lambda}(X_i)\int_0^1\beta(s)X_i(s)ds\,\one\{\EE_n(c)\}-\E\bigg[\tau_{\lambda}(X_i)\int_0^1\beta(s)X_i(s)ds\,\one\{\EE_n(c)\}\bigg]\Bigg)\,.
\end{align}
Therefore, observing  that $n^{-1/2}=o(p_n^{1/(2m)})$ by Assumption~\ref{a:rate}, combining \eqref{tildehh} and Lemma~\ref{lem:hnbeta} yields
with probability tending to one,
\begin{align*}
\sup_{\nu\in[\nu_0,1]}\,\sup_{\tilde\beta\in\mathcal F_{p_n}}\Vert \tilde H_{n,\nu}(\tilde\beta)\Vert_K&\leq c\,\big(p_n^{1/(2m)}+n^{-1/2}\big)\big(\lambda^{-1/(2D)}\log n\big)^{1/2}\\
&\leq c\,p_n^{1/(2m)}\lambda^{-1/(4D)}(\log n)^{1/2}\,,
\end{align*}
where $c>0$ depends on $\nu_0$. In view of \eqref{tildeb}, we deduce from the above equation that, for the $\beta$ in \eqref{tildeb}, with probability tending to one,
\begin{align*}
\sup_{\nu\in[\nu_0,1]}\Vert \tilde H_{n,\nu}(\beta)\Vert_K&\leq \big(c_K\lambda^{-(2a+1)/(4D)}\Vert\beta\Vert_K\big)\sup_{\nu\in[\nu_0,1]}\,\sup_{\tilde\beta\in\mathcal F_{p_n}}\Vert \tilde H_{n,\nu}(\tilde\beta)\Vert_K\\
&\leq c\,p_n^{1/(2m)}\lambda^{-(a+1)/(2D)}(\log n)^{1/2}\Vert\beta\Vert_K\,.
\end{align*}
Observing  that $p_n=O(\lambda^{(2a+1)/(2D)-1})$ and \eqref{det7}, we 
thus deduce  that, with probability tending to one,
\begin{align}\label{i1nrate}
\sup_{\nu\in[\nu_0,1]}\Vert I_{2,2,n,\nu}(\beta)\Vert_K&\leq n^{-1/2}\sup_{\nu\in[\nu_0,1]}\Vert \tilde H_{n,\nu}(\beta)\Vert_K\leq c\,n^{-1/2}p_n^{1/(2m)}\lambda^{-(a+1)/(2D)}(\log n)^{1/2}\Vert\beta\Vert_K\notag\\
&\leq c\,n^{-1/2}\lambda^{-\varsigma}(\log n)^{1/2}\,\Vert\beta\Vert_K=o(1)\,\Vert\beta\Vert_K\,,
\end{align}
where we used Assumption~\ref{a:rate} in the last step. Therefore, combining \eqref{i22nv}, \eqref{i21} and \eqref{i1nrate} yields that, as $n\to\infty$,
\begin{align}\label{sk}
\sup_{\nu\in[\nu_0,1]}\Vert I_{2,n,\nu}(\beta)\Vert_K&=o(1)\,\|\beta\|_K\,.
\end{align}
We now consider the term $-S_{n,\lambda,\nu}(\beta_{\lambda,\nu})$ in \eqref{th}.
Recalling the definition of   $\tau_\lambda$ in \eqref{tau} and observing that 
 $S_{\lambda,\nu}(\beta_{\lambda,\nu})=0$ and $\E\{\e_0\tau_{\lambda}(X_0)\}=0$, in view of \eqref{sn}, we find
\begin{align*} 
&-S_{n,\lambda,\nu}(\beta_{\lambda,\nu})=-\{S_{n,\lambda,\nu}(\beta_{\lambda,\nu})-S_{\lambda,\nu}(\beta_{\lambda,\nu})\}\notag\\
&=\frac{1}{{\lfloor n\nu\rfloor}}\sum_{i=1}^{\lfloor n\nu\rfloor}\Bigg(\tau_{\lambda}(X_i)\bigg\{Y_i-\int_0^1\beta_{\lambda,\nu}(s)X_i(s)ds\bigg\}-\E\bigg[\tau_{\lambda}(X_i)\bigg\{Y_i-\int_0^1\beta_{\lambda,\nu}(s)X_i(s)ds\bigg\}\bigg]\Bigg)\notag\\
&=\frac{1}{{\lfloor n\nu\rfloor}}\sum_{i=1}^{\lfloor n\nu\rfloor}\e_i\,\tau_{\lambda}(X_i)+\frac{1}{{\lfloor n\nu\rfloor}}\sum_{i=1}^{\lfloor n\nu\rfloor}\Bigg(\tau_{\lambda}(X_i)\int_0^1\big\{\beta_0(s)-\beta_{\lambda,\nu}(s)\big\}X_i(s)ds\notag\\
&\hspace{4cm}-\E\bigg[\tau_{\lambda}(X_i)\int_0^1\big\{\beta_0(s)-\beta_{\lambda,\nu}(s)\big\}X_i(s)ds\bigg]\Bigg)\notag\\
&=\frac{1}{{\lfloor n\nu\rfloor}}\sum_{i=1}^{\lfloor n\nu\rfloor}\e_i\,\tau_{\lambda}(X_i)-I_{2,n,\nu}(\beta_0-\beta_{\lambda,\nu})\,,
\end{align*}
where $I_{2,n,\nu}$ is defined in \eqref{i22nv}.
Therefore, we deduce from the above equation and \eqref{sk} that
\begin{align}\label{sss}
\sup_{\nu\in[\nu_0,1]}\Vert S_{n,\lambda,\nu}(\beta_{\lambda,\nu})\Vert_K^2&\leq 2\sup_{\nu\in[\nu_0,1]}\,\bigg\|\frac{1}{{\lfloor n\nu\rfloor}}\sum_{i=1}^{\lfloor n\nu\rfloor}\e_i\tau_{\lambda}(X_i)\bigg\|_K^2+\sup_{\nu\in[\nu_0,1]}\|I_{2,n,\nu}(\beta_0-\beta_{\lambda,\nu})\|_K^2\notag\\
&\leq 2\sup_{\nu\in[\nu_0,1]}\,\bigg\|\frac{1}{{\lfloor n\nu\rfloor}}\sum_{i=1}^{\lfloor n\nu\rfloor}\e_i\tau_{\lambda}(X_i)\bigg\|_K^2+o(1)\,\|\beta_0-\beta_{\lambda,\nu}\|_K^2\,.
\end{align}
For the first term in \eqref{sss}, 
by direct calculations, we find
\begin{align}\label{e0}
&\bigg\|\frac{1}{{\lfloor n\nu\rfloor}}\sum_{i=1}^{\lfloor n\nu\rfloor}\tau_{\lambda}(X_i)\e_i\bigg\|_K^2=\frac{1}{{\lfloor n\nu\rfloor}^2}\sum_{i_1=1}^{\lfloor n\nu\rfloor}\sum_{i_2=1}^{\lfloor n\nu\rfloor}\sum_{k=1}^\infty\sum_{\ell=1}^\infty\Bigg\l\frac{\l \e_{i_1}X_{i_1},\phi_k\r_{L^2}}{1+\lambda\rho_k}\phi_k,\,\frac{\l \e_{i_2}X_{i_2},\phi_\ell\r_{L^2}}{1+\lambda\rho_\ell}\phi_\ell\Bigg\r_K\notag\\
&=\frac{1}{{\lfloor n\nu\rfloor}^2}\sum_{i_1=1}^{\lfloor n\nu\rfloor}\sum_{i_2=1}^{\lfloor n\nu\rfloor}\sum_{k=1}^\infty\frac{1}{1+\lambda\rho_k}\,\big\l \e_{i_1}X_{i_1},\phi_k\big\r_{L^2}\big\l \e_{i_2}X_{i_2},\phi_k\big\r_{L^2}\notag\\
&=\sum_{k=1}^\infty\frac{1}{1+\lambda\rho_k}\bigg(\frac{1}{{\lfloor n\nu\rfloor}}\sum_{i=1}^{\lfloor n\nu\rfloor}\big\l \e_{i}X_{i},\phi_k\big\r_{L^2}\bigg)^2=\frac{1}{{\lfloor n\nu\rfloor}}\sum_{k=1}^\infty\frac{1}{1+\lambda\rho_k}\bigg\l\frac{1}{\sqrt{\lfloor n\nu\rfloor}}\sum_{i=1}^{\lfloor n\nu\rfloor}\e_{i}X_{i},\phi_k\bigg\r_{L^2}^2\notag\\
&\leq\frac{1}{{\lfloor n\nu\rfloor}}\,\bigg\|\frac{1}{\sqrt{\lfloor n\nu\rfloor}}\sum_{i=1}^{\lfloor n\nu\rfloor}\e_{i}X_{i}\bigg\|_{L^2}^2\times\sum_{k=1}^\infty\frac{\|\phi_k\|_{L^2}^2}{1+\lambda\rho_k}\,.
\end{align}
Denote the long-run covariance function
\begin{align}\label{cxe}
C_{X\e}(s,t)=\sum_{\ell=-\infty}^{+\infty}\cov\{\e_0X_0(s),\e_\ell X_\ell(t)\}\,.
\end{align}
Observing Lemmas~\ref{lem:cxe} and \ref{lem:cxe2}, we have that $C_{X\e}\in L^2([0,1]^2)$ and $\int_0^1 C_{X\e}(s,s)ds<\infty$. By Assumption~\ref{a1}, we have that $C_{X\e}$ is positive definite. Let $\{\breve\zeta_j\}_{j=1}^\infty$ and $\{\breve\psi_j\}_{j=1}^\infty$ denote the eigenvalues and the corresponding eigenfunctions of the covariance kernel $C_{X\e}$, such that $\sum_{j=1}^\infty\breve\zeta_j<\infty$.  In addition, since the $X_i$'s and the $\e_i$'s are independent, it is easy to see that the series $\{\e_i X_i\}_{i\in\mathbb Z}$ is $m$-approximable by $\{\e_{i,\ell} X_{i,\ell}\}_{i,\ell\in\mathbb Z}$. By Theorem~1.1 in \cite{berkes2013}, there exists a Gaussian process $\{\Gamma_{X\e}(s,\nu)\}_{s\in[0,1],\nu\in[0,1]}$ in $\mathcal F$ defined in \eqref{f}, given by
\begin{align*}
\Gamma_{X\e}(s,\nu)=\sum_{j=1}^\infty\sqrt{\breve\zeta_j}\,W_j(\nu)\,\breve\psi_j(s)\,,
\end{align*}
such that
\begin{align*}
\sup_{\nu\in[0,1]}\,\bigg\|\Gamma_{X\e}(\cdot,\nu)-\frac{1}{\sqrt {n}}\sum_{i=1}^{\lfloor n\nu\rfloor}\e_{i}X_{i}\bigg\|_{L^2}^2=o_p(1)\,.
\end{align*}
Here, $\{W_{j}\}_{j=1}^\infty$ is a series of i.i.d.~Wiener processes.  Note that $\E\big\{\sup_{\nu\in[0,1]}W_j^2(\nu)\big\}<\infty$, so that
\begin{align*}
\E\bigg\{\sup_{\nu\in[0,1]}\|\Gamma_{X\e}(\cdot,\nu)\|_{L^2}^2\bigg\}\leq\sum_{j=1}^\infty\breve\zeta_j\,\E\bigg\{\sup_{\nu\in[0,1]}W_j^2(\nu)\bigg\}<\infty\,.
\end{align*}
Therefore, in view of \eqref{e0}, we deduce from the above finding that
\begin{align}\label{11}
&\sup_{\nu\in[\nu_0,1]}\,\bigg\|\frac{1}{\lfloor n\nu\rfloor}\sum_{i=1}^{\lfloor n\nu\rfloor}\tau_{\lambda}(X_i)\e_i\bigg\|_K^2\leq\sup_{\nu\in[\nu_0,1]}\Bigg\{\frac{1}{{\lfloor n\nu\rfloor}}\bigg\|\frac{1}{\sqrt {\lfloor n\nu\rfloor}}\sum_{i=1}^{\lfloor n\nu\rfloor}\e_{i}X_{i}\bigg\|_{L^2}^2\Bigg\}\times\sum_{k=1}^\infty\frac{\|\phi_k\|_{L^2}^2}{1+\lambda\rho_k}\notag\\
&\leq n^{-1}\nu_0^{-2}\sup_{\nu\in[0,1]}\bigg\|\frac{1}{\sqrt {n}}\sum_{i=1}^{\lfloor n\nu\rfloor}\e_{i}X_{i}\bigg\|_{L^2}^2\times\sum_{k=1}^\infty\frac{\|\phi_k\|_{L^2}^2}{1+\lambda\rho_k}\times\{1+o_p(1)\}\notag\\
&\leq n^{-1}\nu_0^{-2}\Bigg\{\sup_{\nu\in[0,1]}\,\bigg\|\Gamma_{X\e}(\cdot,\nu)-\frac{1}{\sqrt {n}}\sum_{i=1}^{\lfloor n\nu\rfloor}\e_{i}X_{i}\bigg\|_{L^2}^2+\sup_{\nu\in[0,1]}\|\Gamma_{X\e}(\cdot,\nu)\|_{L^2}^2\Bigg\}\notag\\
&\qquad\times\,\sum_{k=1}^\infty\frac{\|\phi_k\|_{L^2}^2}{1+\lambda\rho_k}\times\{1+o_p(1)\}\notag\\
&=O_p(n^{-1}\lambda^{-(2a+1)/(2D)})\,.
\end{align}
Consequently, combining the above finding and \eqref{deltalambdabeta} and \eqref{sss}, we obtain that
\begin{align}\label{snlambda}
\sup_{\nu\in[\nu_0,1]}\Vert S_{n,\lambda,\nu}(\beta_{\lambda,\nu})\Vert_K=O_p\big(n^{-1/2}\lambda^{-(2a+1)/(4D)}\big)+o(\lambda^{1/2})\,.
\end{align}
Next, let $q_n=c(n^{-1/2}\lambda^{-(2a+1)/(4D)}+\lambda^{1/2})$ and denote by  $\mathcal{B}(r)=\{\gamma\in\mathcal{H},\Vert\gamma\Vert_K\leq r\}$ denote the $\Vert\cdot\Vert_K$-ball with radius $r>0$ in $\H$. In view of \eqref{i1nrate}, for any $\beta\in\mathcal{B}(q_n)$, with probability tending to one, $\Vert I_{2,n,\nu}(\beta)\Vert_K\leq \Vert\beta\Vert_K/2\leq q_n/2$. Therefore, in view of \eqref{zero}, \eqref{sk} and \eqref{snlambda}, for $F_n(\beta)$ defined in \eqref{th}, with probability tending to one, for any $\beta\in\mathcal{B}(q_n)$,
\begin{align*}
\sup_{\nu\in[\nu_0,1]}\Vert F_{n,\nu}(\beta)\Vert_K&\leq\sup_{\nu\in[\nu_0,1]}\Vert I_{2,n,\nu}(\beta)\Vert_K+\sup_{\nu\in[\nu_0,1]}\Vert S_{n,\lambda,\nu}(\beta_{\lambda,\nu})\Vert_K\\
&\leq c\,n^{-1/2}\lambda^{-(2a+1)/(4D)}+q_n/2\leq q_n\,,
\end{align*}
which indicates that $F_{n,\nu}\{\mathcal{B}(q_n)\}\subset\mathcal{B}(q_n)$ uniformly in $\nu\in[\nu_0,1]$. Observing \eqref{th}--\eqref{zero}, we have, for any $\beta_1,\beta_2\in\mathcal B(q_n)$, $F_{n,\nu}(\beta_1)-F_{n,\nu}(\beta_2)=I_{2,n,\nu}(\beta_1)-I_{2,n,\nu}(\beta_2)$. Due to \eqref{sk}, with probability tending to one,
\begin{align*}
&\sup_{\nu\in[\nu_0,1]}\Vert F_{n,\nu}(\beta_1)-F_{n,\nu}(\beta_2)\Vert_K=\sup_{\nu\in[\nu_0,1]}\Vert I_{2,n,\nu}(\beta_1)-I_{2,n,\nu}(\beta_2)\Vert_K\leq \Vert\beta_1-\beta_2\Vert_K/2\,,
\end{align*}
which indicates that $F_{n,\nu}$ is a contraction mapping on $\mathcal{B}(q_n)$ uniformly in $\nu\in[\nu_0,1]$. By the Banach contraction mapping theorem, there exists a unique element $\beta_\nu^*\in\mathcal{B}_n$ such that $\beta_\nu^*=F_{n,\nu}(\beta_\nu^*)=\beta_\nu^*-S_{n,\lambda,\nu}(\beta_{\lambda,\nu}+\beta_\nu^*)$. Letting $\hat\beta_{n,\lambda}(\cdot,\nu)=\beta_{\lambda,\nu}+\beta_\nu^*$, we have $S_{n,\lambda,\nu}\{\hat\beta_{n,\lambda}(\cdot,\nu)\}=0$, which implies that $\hat\beta_{n,\lambda}(\cdot,\nu)$ is the estimator defined by \eqref{hatbeta}. Moreover, we have, with probability tending to one,
\begin{align*}
\sup_{\nu\in[\nu_0,1]}\|\hat\beta_{n,\lambda}(\cdot,\nu)-\beta_{\lambda,\nu}\|_K=\sup_{\nu\in[\nu_0,1]}\|\beta_\nu^*\|_K\leq q_n\,.
\end{align*}
In view of \eqref{deltalambdabeta},
\begin{align*}
\sup_{\nu\in[\nu_0,1]}\|\hat\beta_{n,\lambda}(\cdot,\nu)-\beta_0\|_K&\leq\sup_{\nu\in[\nu_0,1]}\|\beta_{\lambda,\nu}-\beta_0\|_K+\sup_{\nu\in[\nu_0,1]}\|\hat\beta_{n,\lambda}(\cdot,\nu)-\beta_{\lambda,\nu}\|_K\notag\\
&=O_p(\lambda^{1/2}+q_n)=O_p\big(\lambda^{1/2}+n^{-1/2}\lambda^{-(2a+1)/(4D)}\big)\,,
\end{align*}
which completes the proof.


\end{proof}

\subsubsection*{Proof of Theorem~\ref{thm:bahadur}.}
       
Now, we provide the proof of Theorem~\ref{thm:bahadur} using Lemma~\ref{lem:hatbeta}. To be precise, we define    
\begin{equation}\label{snnu}
\begin{split}
&S_{n,\nu}(\beta)=-\frac{1}{\lfloor n\nu\rfloor}\sum_{i=1}^{\lfloor n\nu\rfloor}\left\{Y_i-\int_0^1X_i(s)\,\beta(s)\,ds\right\}\,\tau_{\lambda}(X_i)\,,\\
&S_\nu(\beta)=-\E\bigg[\bigg\{Y_0-\int_0^1X_0(s)\,\beta(s)\,ds\bigg\}\tau_{\lambda}(X_0)\bigg]\, ,
\end{split}
\end{equation}
and  $\Delta_\nu\beta=\hat\beta_{n,\lambda}(\cdot,\nu)-\beta_0+W_\lambda(\beta_0)$ for the sake of notational convenience.
Since $\mathcal D^2S_{\lambda,\nu}$ vanishes and $\mathcal DS_{\lambda,\nu}(\beta_0)=id$ by \eqref{prop:id}, we have 
\begin{align*}
S_{\lambda,\nu}\{\hat\beta_{n,\lambda}(\cdot,\nu)\}-S_{\lambda,\nu}(\beta_0)=\mathcal DS_{\lambda,\nu}(\beta_0)\Delta_\nu\beta=\Delta_\nu\beta\,.
\end{align*}
Since $S_{n,\lambda,\nu}\{\hat\beta_{n,\lambda}(\cdot,\nu)\}=0$, we deduce from this equation that
\begin{align}\label{co}
&\hat\beta_{n,\lambda}(\cdot,\nu)-\beta_0+S_{n,\lambda,\nu}(\beta_0)=\Delta_\nu\beta+S_{n,\lambda,\nu}(\beta_0)\notag\\
&=-S_{n,\lambda,\nu}\{\hat\beta_{n,\lambda}(\cdot,\nu)\}+S_{n,\lambda,\nu}(\beta_0)+S_{\lambda,\nu}\{\hat\beta_{n,\lambda}(\cdot,\nu)\}-S_{\lambda,\nu}(\beta_0)\notag\\
&=-S_{n,\nu}\{\hat\beta_{n,\lambda}(\cdot,\nu)\}+S_{n,\nu}(\beta_0)+S_\nu\{\hat\beta_{n,\lambda}(\cdot,\nu)\}-S_{\nu}(\beta_0)\,,
%
\end{align}
where  $S_{n,\nu}$ and $S_{\nu}$ are defined in \eqref{snnu}. 
Let $r_n=\lambda^{1/2}+n^{-1/2}\lambda^{-(2a+1)/(4D)}$. For $c_1>0$, consider  the event $\mathcal M_n=\big\{\sup_{\nu\in[\nu_0,1]}\Vert\Delta_\nu\beta\Vert_K\leq c_1r_n\big\}$. By Lemma~\ref{lem:hatbeta}, we obtain that $\P(\mathcal M_n)$ tends to one if the constant $c_1>0$
is chosen sufficiently 
large. For $c_K>0$ in Lemma~\ref{lem:3.1}, let $q_n=c_1c_K\lambda^{-(2a+1)/(4D)} r_n$ and let
\begin{align*}
p_n=c_1^2\,q_n^{-2}\lambda^{-1}r_n^2=c_1^2\,\big(c_1c_K\lambda^{-(2a+1)/(4D)} r_n\big)^{-2}\lambda^{-1}r_n^2=c_K^{-2}\,\lambda^{(-2D+2a+1)/(2D)}\,.
\end{align*}
Note that $p_n\geq 1$ for $n$ large enough. In order to apply Lemma~\ref{lem:hnbeta}, we shall rescale $\Delta_\nu\beta$ such that the $L^2$-norm of its rescaled version is bounded by $1$. Let $\tilde\Delta_\nu\beta=q_n^{-1}\Delta_\nu\beta$. By Lemma~\ref{lem:3.1}, we have that, on the event $\mathcal M_n$,
\begin{align*}
\Vert\tilde\Delta_\nu\beta\Vert_{L^2} & \leq c_K\lambda^{-(2a+1)/(4D)} \Vert\tilde\Delta_\nu\beta\Vert_K \\
& \leq c_Kq_n^{-1}\lambda^{-(2a+1)/(4D)} \Vert\Delta_\nu\beta\Vert_K\leq c_1c_Kq_n^{-1}\lambda^{-(2a+1)/(4D)} r_n\leq 1\,.
\end{align*}
In addition, since $J(\Delta_\nu\beta,\Delta_\nu\beta)\leq\lambda^{-1}\Vert\Delta_\nu\beta\Vert_K^2$, we have
\begin{align*}
J(\tilde\Delta_\nu\beta,\tilde\Delta_\nu\beta)\leq q_n^{-2}J(\Delta_\nu\beta,\Delta_\nu\beta)\leq q_n^{-2}\lambda^{-1}\Vert\Delta_\nu\beta\Vert_K^2\leq c_1^2\,q_n^{-2}\lambda^{-1}r_n^2=p_n\,. 
\end{align*}
Hence, we have shown that $\tilde\Delta_\nu\beta\in\mathcal{F}_{p_n}$, where $\mathcal{F}_{p_n}$
is defined in \eqref{det8}.

Recall from \eqref{hnk} that, for the event $\EE_n(c)$ defined in \eqref{enc}, for any $\beta\in\H$,
\begin{align*}
&\tilde H_{n,\nu}(\beta)=\frac{1}{\sqrt{\lfloor n\nu\rfloor}}\sum_{i=1}^{{\lfloor n\nu\rfloor}}\bigg(\tau_{\lambda}(X_i)\l\beta,X_i\r_{L^2}\,\one\{\EE_n(c)\}-\E\Big[\tau_{\lambda}(X_i)\l\beta,X_i\r_{L^2}\,\one\{\EE_n(c)\}\Big]\bigg)\,.
\end{align*}
In view of \eqref{snnu} and 
\eqref{co} we thus obtain on the event $\EE_n(c)$ that 
\begin{align}\label{coo}
&\hat\beta_{n,\lambda}(\cdot,\nu)-\beta_0+S_{n,\lambda,\nu}(\beta_0)\notag\\
&=-S_{n,\nu}\{\hat\beta_{n,\lambda}(\cdot,\nu)\}+S_{n,\nu}(\beta_0)+S\{\hat\beta_{n,\lambda}(\cdot,\nu)\}-S_{\nu}(\beta_0)\notag\\
&=\frac{1}{{\lfloor n\nu\rfloor}}\sum_{i=1}^{{\lfloor n\nu\rfloor}}\bigg[\tau_{\lambda}(X_i)\int_0^1X_i(s)\,\Delta_\nu\beta(s)\,ds-\E\left\{\tau_{\lambda}(X_i)\int_0^1X_i(s)\,\Delta_\nu\beta(s)\,ds\right\}\bigg]\notag\\
&\leq \frac{1}{\sqrt{\lfloor n\nu\rfloor}}\, \tilde H_{n,\nu}(\Delta_\nu\beta)-\E\left\{\tau_{\lambda}(X_i)\int_0^1X_i(s)\,\Delta_\nu\beta(s)\,ds\,\one\{\EE(c)^{\rm c}\}\right\}\,.
\end{align}
Note that following arguments similar to the ones used in \eqref{e}, we deduce that
\begin{align}\label{8}
&\bigg\|\E\left\{\tau_{\lambda}(X_i)\int_0^1X_i(s)\,\Delta_\nu\beta(s)\,ds\,\one\{\EE(c)^{\rm c}\}\right\}\bigg\|_K\notag\\
&\leq c\times o(n^{-1})\times\lambda^{-1/(2D)}\times\|\Delta_\nu\beta\|_K\times\{1+o(1)\}\leq c\,n^{-1}\lambda^{-1/(2D)}r_n=o(v_n)\,.
\end{align}
Since $\tilde\Delta_\nu\beta\in\mathcal{F}_{p_n}$, by applying Lemma~\ref{lem:hnbeta}, observing \eqref{tildehh}, we deduce that
\begin{align*}
\sup_{\nu\in[\nu_0,1]}\,\sup_{\tilde\Delta_\nu\beta\in\mathcal F_{p_n}}\Vert \tilde H_{n,\nu}(\tilde\Delta_\nu\beta)\Vert_K&=O_p\big\{\big(p_n^{1/(2m)}+n^{-1/2}\big)\big(\lambda^{-1/(2D)}\log n\big)^{1/2}\big\}\\
&=O_p\big\{p_n^{1/(2m)}\lambda^{-1/(4D)}(\log n)^{1/2}\big\}\,.
\end{align*}
Consequently, for the $\Delta_\nu\beta$ in \eqref{coo}, it follows with probability tending to one,
\begin{align*} 
&n^{-1/2}\sup_{\nu\in[\nu_0,1]}
\Vert \tilde H_{n,\nu}(\Delta_\nu\beta)\Vert_K
\leq n^{-1/2}q_n\sup_{\nu\in[\nu_0,1]}\,\sup_{\tilde\Delta_\nu\beta\in\mathcal F_{p_n}}\Vert \tilde H_{n,\nu}(\tilde\Delta_\nu\beta)\Vert_K\notag\\
&\leq c\,n^{-1/2}q_n\,p_n^{1/(2m)}\lambda^{-1/(4D)}(\log n)^{1/2}\notag\\
&\leq c\,n^{-1/2}(\lambda^{-(2a+1)/(4D)}\, r_n)\,\lambda^{(-2D+2a+1)/(4Dm)}\,\lambda^{-1/(4D)}(\log n)^{1/2}\notag\\
&=c\,n^{-1/2}\lambda^{-\varsigma}\big(\lambda^{1/2}+n^{-1/2}\lambda^{-(2a+1)/(4D)}\big)(\log n)^{1/2}\,,
\end{align*}
for the constant $\varsigma>0$ in Assumption~\ref{a:rate}. Combining the above result with \eqref{coo} and \eqref{8} yields that
\begin{align}\label{b}
&\sup_{\nu\in[\nu_0,1]}\big\Vert\hat\beta_{n,\lambda}(\cdot,\nu)-\beta_0+W_\lambda(\beta_0)+S_{n,\lambda,\nu}(\beta_0)\big\Vert_K\notag\\
&=O_p\big\{n^{-1/2}\lambda^{-\varsigma}(\lambda^{1/2}+n^{-1/2}\lambda^{-(2a+1)/(4D)})(\log n)^{1/2}\big\}=O_p(v_n)\,.
\end{align}
Observing $S_{n,\lambda,\nu}(\beta_0)$ defined in \eqref{snlambdanubeta0}, we therefore deduce from the above equation that
\begin{align}\label{cc}
&\sup_{\nu\in[\nu_0,1]}\,\bigg\Vert\nu\{\hat\beta_{n,\lambda}(\cdot,\nu)-\beta_0+W_\lambda(\beta_0)\}-\frac{1}{n}\sum_{i=1}^{\lfloor n\nu \rfloor}\e_i\tau_{\lambda}(X_i)\bigg\Vert_K\notag\\
&\leq\sup_{\nu\in[\nu_0,1]}\Big\{\nu\big\Vert\hat\beta_{n,\lambda}(\cdot,\nu)-\beta_0+W_\lambda(\beta_0)+S_{n,\lambda,\nu}(\beta_0)\big\Vert_K\Big\}\notag\\
&\qquad+\sup_{\nu\in[\nu_0,1]}\bigg\{\bigg|\frac{\nu}{\lfloor n\nu\rfloor}-\frac{1}{n}\bigg|\times\bigg\Vert\sum_{i=1}^{\lfloor n\nu \rfloor}\e_i\,\tau_{\lambda}(X_i)\bigg\Vert_K\bigg\}+\sup_{\nu\in[\nu_0,1]}\Big\{\nu\Vert W_\lambda(\beta_0)\Vert_K\Big\}\notag\\
&\leq\sup_{\nu\in[\nu_0,1]}\big\Vert\hat\beta_{n,\lambda}(\cdot,\nu)-\beta_0+W_\lambda(\beta_0)+S_{n,\lambda,\nu}(\beta_0)\big\Vert_K\notag\\
&\qquad+n^{-1}\sup_{\nu\in[\nu_0,1]}\bigg\Vert\frac{1}{\lfloor n\nu\rfloor}\sum_{i=1}^{\lfloor n\nu \rfloor}\e_i\,\tau_{\lambda}(X_i)\bigg\Vert_K+\Vert W_\lambda(\beta_0)\Vert_K\,.
\end{align}
Observing \eqref{11} and Assumption~\ref{a:rate} we find
\begin{align*} 
n^{-1}\sup_{\nu\in[\nu_0,1]}\bigg\Vert\frac{1}{\lfloor n\nu\rfloor}\sum_{i=1}^{\lfloor n\nu \rfloor}\e_i\,\tau_{\lambda}(X_i)\bigg\Vert_K=O_p(n^{-3/2}\lambda^{-(2a+1)/(4D)})=o_p(v_n)\,.
\end{align*}
The proof is therefore complete by combining the above equation with \eqref{b} and \eqref{cc}.

\subsection{Proof of Theorem~\ref{thm:wip}}\label{proof:thm:wip}

The proof is now performed in two steps.
First, in Lemma~\ref{lem:m}, we will 
show that the $U_i$'s are $L^2$-$m$-approximable (see Assumptions~(1.1)--(1.4) in \citealp{berkes2013}). Second, in Lemma~\ref{lem:l2} we will show that $C_{U,\lambda}$ defined in \eqref{cu} satsifies  $\sup_{\lambda>0}\int_0^1\int_0^1|C_{U,\lambda}(s,t)|dsdt<\infty$. 
Then, Theorem~\ref{thm:wip} is proved by the arguments as given in 
the proof of Theorem~1.1 in \cite{berkes2013}, which are omitted for the sake of brevity.

\begin{lemma}\label{lem:m}
Under the assumptions of Theorem~\ref{thm:wip}, the series $\{U_i\}_{i\in\mathbb Z}$ defined in \eqref{u} is $L^2$-$m$-approximable w.r.t.~the series $\{U_{i,\ell}\}_{i,\ell\in\mathbb Z}$ 
 uniformly in $\lambda>0$, where 
\begin{align*} 
U_{i,\ell}=\lambda^{(2a+1)/(2D)}\,\e_{i,\ell}\,\tau_{\lambda}(X_{i,\ell})=\lambda^{(2a+1)/(2D)}\e_{i,\ell}\sum_{k=1}^\infty\frac{\l X_{i,\ell},\phi_k\r_{L^2}}{1+\lambda\rho_k}\phi_k\,,\quad (i,\ell\in\mathbb Z)\, .
\end{align*}

\end{lemma}

\begin{proof}
By Lemmas~\ref{lem:s4} and \ref{lem:3.1} in Section~\ref{app:aux:lem},
we obtain that there exists a constant $c>0$ such that
\begin{align}\label{1}
\Vert\tau_{\lambda}(X_i)\Vert_{L^2}\leq c\lambda^{-(2a+1)/(4D)}\Vert\tau_{\lambda}(X_i)\Vert_K\leq c\lambda^{-(2a+1)/(2D)}\|X_i\|_{L^2}\,.
\end{align}
This together with the fact that $\|U_i\|_{L^2}\leq \lambda^{(2a+1)/(2D)}|\e_i|\cdot\|\tau_{\lambda}(X_i)\|_{L^2}$ implies that $U_i\in L^2([0,1])$ uniformly in $\lambda>0$. In addition, $\E(U_i)\equiv0$, and,  by \eqref{1}, for any $\delta\in(0,1)$,
\begin{align*}
\E\|U_i\|_{L^2}^{2+\delta}\leq \lambda^{(2a+1)(2+\delta)/(2D)}\E|\e_i|^{2+\delta}\,\E\|\tau_{\lambda}(X_i)\|_{L^2}^{2+\delta}\leq c\,\lambda^{(2a+1)(1+\delta)/(2D)}\E|\e_i|^{2+\delta}\,\E\|X_i\|_{L^2}^{2+\delta}<\infty\,,
\end{align*}
where in the last step we have used Assumptions~\ref{a3.1} and \ref{a5.2}.
Moreover, note that $m$-approximable series are strictly stationary (see, for example, \citealp{hormann2010}). Hence, by applying Assumption~\ref{a:m} and \eqref{1}, we find that, uniformly in $\lambda>0$,
\begin{align*}
&\sum_{\ell=1}^\infty\big(\E\|U_i-U_{i,\ell}\|_{L^2}^{2+\delta}\big)^{1/\kappa}=\sum_{\ell=1}^\infty\Big\{\lambda^{(2a+1)(2+\delta)/(2D)}\,\E\big\|\e_i\,\tau_{\lambda}(X_i)-\e_{i,\ell}\,\tau_{\lambda}(X_{i,\ell})\big\|_{L^2}^{2+\delta}\Big\}^{1/\kappa}\\
&\leq 2^{(1+\delta)/\kappa}\,\sum_{\ell=1}^\infty\Big\{\lambda^{(2a+1)(2+\delta)/(2D)}\times\E|\e_i-\e_{i,\ell}|^{2+\delta}\times\E\|\tau_{\lambda}(X_i)\|_{L^2}^{2+\delta}\Big\}^{1/\kappa}\\
&\quad+2^{(1+\delta)/\kappa}\,\sum_{\ell=1}^\infty\Big\{\lambda^{(2a+1)(2+\delta)/(2D)}\times\E|\e_{i,\ell}|^{2+\delta}\times\E\|\tau_{\lambda}(X_i)-\tau_{\lambda}(X_{i,\ell})\|_{L^2}^{2+\delta}\Big\}^{1/\kappa}\\
&=2^{(1+\delta)/\kappa}\,\Big\{\lambda^{(2a+1)(2+\delta)/(2D)}\E\|\tau_{\lambda}(X_i)\|_{L^2}^{2+\delta}\Big\}^{1/\kappa}\,\sum_{\ell=1}^\infty\big(\E|\e_i-\e_{i,\ell}|^{2+\delta}\big)^{1/\kappa}\\
&\quad+2^{(1+\delta)/\kappa}\,\big(\E|\e_{0}|^{2+\delta}\big)^{1/\kappa}\times\bigg\{\lambda^{(2a+1)(2+\delta)/(2D)}\sum_{\ell=1}^\infty\E\|\tau_{\lambda}(X_i-X_{i,\ell})\|_{L^2}^{2+\delta}\bigg\}^{1/\kappa}\\
&\leq 2^{(1+\delta)/\kappa}\,\Big(\E\|X_i\|_{L^2}^{2+\delta}\Big)^{1/\kappa}\,\sum_{\ell=1}^\infty\big(\E|\e_i-\e_{i,\ell}|^{2+\delta}\big)^{1/\kappa}\\
&\quad+2^{(1+\delta)/\kappa}\,\big(\E|\e_{0}|^{2+\delta}\big)^{1/\kappa}\bigg(\sum_{\ell=1}^\infty\E\|X_i-X_{i,\ell}\|_{L^2}^{2+\delta}\bigg)^{1/\kappa}<\infty\,.
\end{align*}
Now, we have shown that the series $\{U_i\}_{i\in\mathbb Z}$ is $L^2$-$m$-approximable uniformly in $\lambda>0$. 

\end{proof}



\begin{lemma}\label{lem:l2}
Under the assumptions of Theorem~\ref{thm:wip}, we have  $$\sup_{\lambda>0}\int_0^1\int_0^1\{C_{U,\lambda}(s,t)\}^2\,ds\,dt<\infty\,,
$$
where  $C_{U,\lambda}$ is defined in \eqref{cu}.
\end{lemma}

\begin{proof}

Note that by Assumption~\ref{a:m}, for $\ell\geq1$, $\e_{0,\ell}$ and $\e_{-\ell}$ are independent; $\tau_{\lambda}(X_{0,\ell})$ and $\tau_{\lambda}(X_{-\ell})$ are independent. Note that $\E(\e_\ell)=0$ for any $\ell\in\mathbb Z$. Hence we deduce that, for $\ell\geq1$, $\E(\e_{0,\ell}\e_{-\ell})=0$, so that
\begin{align}\label{p1}
\E(\e_0\e_{-\ell})=\E\big\{(\e_0-\e_{0,\ell})\e_{-\ell}\big\}+\E(\e_{0,\ell}\e_{-\ell})=\E\big\{(\e_0-\e_{0,\ell})\e_{-\ell}\big\}\,.
\end{align}
In addition, for $\ell\geq1$, we have
\begin{align*}
&\E\big\{\tau_{\lambda}(X_{0})(s)\,\tau_{\lambda}(X_{-\ell})(t)\big\}\\
&=\E\Big[\big\{\tau_{\lambda}(X_{0})(s)-\tau_{\lambda}(X_{0,\ell})(s)\big\}\,\tau_{\lambda}(X_{-\ell})(t)\Big]+\E\big\{\tau_{\lambda}(X_{0,\ell})(s)\,\tau_{\lambda}(X_{-\ell})(t)\big\}\\
&=\E\Big[\big\{\tau_{\lambda}(X_{0})(s)-\tau_{\lambda}(X_{0,\ell})(s)\big\}\,\tau_{\lambda}(X_{-\ell})(t)\Big]+\E\big\{\tau_{\lambda}(X_{0})(s)\}\times\E\big\{\tau_{\lambda}(X_{0})(t)\big\}\,.
\end{align*}
Since $\E\{\e_\ell\,\tau(X_\ell)\}\equiv0$, combining the above equation and \eqref{p1} implies that, for $\ell\geq1$,
\begin{align*}
&\cov\big\{\e_{0}\, \tau_{\lambda}(X_{0})(s)\,,\e_{-\ell} \,\tau_{\lambda}(X_{-\ell})(t)\big\}\\
&=\E\big\{(\e_{0}-\e_{0,\ell})\e_{-\ell}\big\}\,\E\Big[\big\{\tau_{\lambda}(X_{0})(s)-\tau_{\lambda}(X_{0,\ell})(s)\big\}\times\tau_{\lambda}(X_{-\ell})(t)\Big]\\
&\quad+\E\big\{(\e_{0}-\e_{0,\ell})\e_{-\ell}\big\}\,\E\big\{\tau_{\lambda}(X_{0})(s)\}\times\E\big\{\tau_{\lambda}(X_{0})(t)\big\}\,.
\end{align*}
Therefore, we deduce from the above equation that
\begin{align}\label{intcu}
&\int_0^1\int_0^1\{C_{U,\lambda}(s,t)\}^2dsdt\notag\\
&=\lambda^{2(2a+1)/D}\int_0^1\int_0^1\bigg[\cov\big\{\e_0\, \tau_{\lambda}(X_0)(s)\,,\e_0 \,\tau_{\lambda}(X_0)(t)\big\}\notag\\
&\hspace{4cm}+2\sum_{\ell=1}^{+\infty}\cov\big\{\e_0\, \tau_{\lambda}(X_0)(s)\,,\e_{-\ell} \,\tau_{\lambda}(X_{-\ell})(t)\big\}\bigg]^2dsdt\notag\\
&=\lambda^{2(2a+1)/D}\int_0^1\int_0^1\bigg[\E(\e_0^2)\,\E\big\{\tau_{\lambda}(X_0)(s)\times\tau_{\lambda}(X_0)(t)\big\}\notag\\
&\quad+2\sum_{\ell=1}^{+\infty}\E\big\{(\e_{0}-\e_{0,\ell})\e_{-\ell}\big\}\,\E\Big[\big\{\tau_{\lambda}(X_{0})(s)-\tau_{\lambda}(X_{0,\ell})(s)\big\}\times\tau_{\lambda}(X_{-\ell})(t)\Big]\notag\\
&\quad+2\,\E\big\{\tau_{\lambda}(X_{0})(s)\}\times\E\big\{\tau_{\lambda}(X_{0})(t)\}\sum_{\ell=1}^{+\infty}\E\big\{(\e_{0}-\e_{0,\ell})\e_{-\ell}\big\}\bigg]^2dsdt\notag\\
&\leq3I_1+12I_2+12I_3\,,
\end{align}
where
\begin{align*}
&I_1=\lambda^{2(2a+1)/D}\{\E(\e_0^2)\}^2\int_{[0,1]^2}\Big[\E\big\{\tau_{\lambda}(X_0)(s)\times\tau_{\lambda}(X_0)(t)\big\}\Big]^2dsdt\,,\\
&I_2=\lambda^{2(2a+1)/D}\int_{[0,1]^2}\Bigg(\sum_{\ell=1}^{+\infty}\E\big\{(\e_{0}-\e_{0,\ell})\e_{-\ell}\big\}\,\E\Big[\big\{\tau_{\lambda}(X_{0})(s)-\tau_{\lambda}(X_{0,\ell})(s)\big\}\tau_{\lambda}(X_{-\ell})(t)\Big]\Bigg)^2dsdt\,,\\
&I_3=\lambda^{2(2a+1)/D}\bigg[\sum_{\ell=1}^{+\infty}\E\big\{(\e_{0}-\e_{0,\ell})\e_{-\ell}\big\}\bigg]^2\,\int_{[0,1]^2}\Big[\E\big\{\tau_{\lambda}(X_{0})(s)\}\times\E\big\{\tau_{\lambda}(X_{0})(t)\big\}\Big]^2dsdt\,.
\end{align*}
For the first term $I_1$, note that
\begin{align*}
&\E\big\{\tau_{\lambda}(X_{0})(s)\times\tau_{\lambda}(X_{0})(t)\big\}\\
&=\E\Bigg[\bigg\{\sum_{k_1=1}^\infty\frac{\l X_{0},\phi_{k_1}\r_{L^2}}{1+\lambda\rho_{k_1}}\phi_{k_1}(s)\bigg\}\bigg\{\sum_{k_2=1}^\infty\frac{\l X_{0},\phi_{k_2}\r_{L^2}}{1+\lambda\rho_{k_2}}\phi_{k_2}(t)\bigg\}\Bigg]\\
&=\sum_{k_1=1}^\infty\sum_{k_2=1}^\infty\frac{\phi_{k_1}(s)\phi_{k_2}(t)}{(1+\lambda\rho_{k_1})(1+\lambda\rho_{k_2})}\E\Big(\l X_{0},\phi_{k_1}\r_{L^2}\l X_{0},\phi_{k_2}\r_{L^2}\Big)\\
&=\sum_{k_1=1}^\infty\sum_{k_2=1}^\infty\frac{\phi_{k_1}(s)\phi_{k_2}(t)}{(1+\lambda\rho_{k_1})(1+\lambda\rho_{k_2})}\int_0^1\int_0^1C_X(t_1,t_2)\phi_{k_1}(t_1)\phi_{k_2}(t_2)dt_1dt_2=\sum_{k=1}^\infty\frac{\phi_{k}(s)\phi_{k}(t)}{(1+\lambda\rho_{k})^2}\,.
\end{align*}
Hence, by the Cauchy-Schwarz inequality, we find
\begin{align*}
I_1&=\lambda^{2(2a+1)/D}\{\E(\e_0^2)\}^2\int_0^1\int_0^1\bigg\{\sum_{k=1}^\infty\frac{\phi_{k}(s)\phi_{k}(t)}{(1+\lambda\rho_{k})^2}\bigg\}^2dsdt\\
&=\lambda^{2(2a+1)/D}\{\E(\e_0^2)\}^2\sum_{k_1=1}^\infty\sum_{k_2=1}^\infty\frac{1}{(1+\lambda\rho_{k_1})^2(1+\lambda\rho_{k_2})^2}\int_0^1\int_0^1\phi_{k_1}(s)\phi_{k_2}(s)\phi_{k_1}(t)\phi_{k_2}(t)dsdt\\
&\leq\lambda^{2(2a+1)/D}\{\E(\e_0^2)\}^2\sum_{k_1=1}^\infty\sum_{k_2=1}^\infty\frac{\|\phi_{k_1}\|_{L^2}^2\,\|\phi_{k_2}\|_{L^2}^2}{(1+\lambda\rho_{k_1})^2(1+\lambda\rho_{k_2})^2}\\
&=\lambda^{2(2a+1)/D}\{\E(\e_0^2)\}^2\bigg\{\sum_{k=1}^\infty\frac{\|\phi_{k}\|_{L^2}^2}{(1+\lambda\rho_{k})^2}\bigg\}^2=O(\lambda^{(2a+1)/D})\,.
\end{align*}

For the second term $I_2$, by the Cauchy-Schwarz inequality, we have
\begin{align*}
&\Bigg(\sum_{\ell=1}^{+\infty}\E\big\{(\e_{0}-\e_{0,\ell})\e_{-\ell}\big\}\,\E\Big[\big\{\tau_{\lambda}(X_{0})(s)-\tau_{\lambda}(X_{0,\ell})(s)\big\}\times\tau_{\lambda}(X_{-\ell})(t)\Big]\Bigg)^2\\
&\leq \Bigg(\sum_{\ell=1}^{+\infty}\E\big\{(\e_{0}-\e_{0,\ell})\e_{-\ell}\big\}\,\Big[\E\big\{\tau_{\lambda}(X_{0})(s)-\tau_{\lambda}(X_{0,\ell})(s)\big\}^2\Big]^{1/2}\times\Big[\E\{\tau_{\lambda}(X_{-\ell})(t)\}^2\Big]^{1/2}\Bigg)^2\\
&\leq\Bigg\{\sum_{\ell=1}^{+\infty} \E(\e_{0}-\e_{0,\ell})^2 \,\E(\e_{0}^2) \Bigg\}\times\Bigg[\sum_{\ell=1}^{+\infty}\E\big\{\tau_{\lambda}(X_{0})(s)-\tau_{\lambda}(X_{0,\ell})(s)\big\}^2\Bigg]\times\E\{\tau_{\lambda}(X_{0})(t)\}^2\\
&=\E(\e_0)^2\times\Bigg\{\sum_{\ell=1}^{+\infty}\E(\e_{0}-\e_{0,\ell})^2\Bigg\}\times\Bigg[\sum_{\ell=1}^{+\infty}\E\big\{\tau_{\lambda}(X_{0})(s)-\tau_{\lambda}(X_{0,\ell})(s)\big\}^2\Bigg]\times\E\{\tau_{\lambda}(X_{0})(t)\}^2\,.
\end{align*}
Therefore, observing \eqref{1} and Assumption~\ref{a:m}, we deduce from the above equation that
\begin{align*}
I_2&\leq\lambda^{2(2a+1)/D}\,\E(\e_0)^2\times\Bigg\{\sum_{\ell=1}^{+\infty}\E(\e_{0}-\e_{0,\ell})^2\Bigg\}\\
&\qquad\times\int_0^1\Bigg[\sum_{\ell=1}^{+\infty}\E\big\{\tau_{\lambda}(X_{0})(s)-\tau_{\lambda}(X_{0,\ell})(s)\big\}^2\Bigg]ds\times\int_0^1\E\{\tau_{\lambda}(X_{0})(t)\}^2dt\\
&=\lambda^{2(2a+1)/D}\,\E(\e_0)^2\Bigg\{\sum_{\ell=1}^{+\infty}\E(\e_{0}-\e_{0,\ell})^2\Bigg\}\Bigg[\sum_{\ell=1}^{+\infty}\E\big\|\tau_{\lambda}(X_{0})-\tau_{\lambda}(X_{0,\ell})\big\|_{L^2}^2\Bigg]\,\E\|\tau_{\lambda}(X_0)\|_{L^2}^2\\
&\leq c\,\E(\e_0)^2\times\Bigg\{\sum_{\ell=1}^{+\infty}\E(\e_{0}-\e_{0,\ell})^2\Bigg\}\times\E\|X_0\|_{L^2}^2\times\Bigg(\sum_{\ell=1}^{+\infty}\E\|X_{0}-X_{0,\ell}\|_{L^2}^2\Bigg)<\infty\,.
\end{align*}

For the third term $I_3$, by the Cauchy-Schwarz inequality, \eqref{1} and Assumption~\ref{a:m}, we deduce that
\begin{align*}
I_3&\leq\lambda^{2(2a+1)/D}\,\big(\E\|\tau_{\lambda}(X_0)\|_{L^2}^2\big)^2\times\bigg[\sum_{\ell=1}^{+\infty}\E\big\{(\e_{0}-\e_{0,\ell})\e_{-\ell}\big\}\bigg]^2\\
&\leq \lambda^{2(2a+1)/D}\,\big(\E\|\tau_{\lambda}(X_0)\|_{L^2}^2\big)^2\times\E(\e_0)^2\times\Bigg\{\sum_{\ell=1}^{+\infty}\E(\e_{0}-\e_{0,\ell})^2\Bigg\}\\
&\leq c\,\big(\E\|X_0\|_{L^2}^2\big)^2\times\E(\e_0)^2\times\Bigg\{\sum_{\ell=1}^{+\infty}\E(\e_{0}-\e_{0,\ell})^2\Bigg\}<\infty\,.
\end{align*}
In conclusion, we deduce from \eqref{intcu} that $\sup_{\lambda>0}\int_0^1\int_0^1 \{C_{U,\lambda}(s,t)\}^2dsdt<\infty$, which completes the proof.
\end{proof}

\subsection{Proof of Theorem~\ref{thm:2.1}}\label{proof:thm:2.1}
We first deal with the bias term $ W_\lambda(\beta_0)$ in \eqref{bahadur}. Observing \eqref{wphi}, we deduce that
\begin{align*}
W_\lambda(\beta_0)=\sum_{k=1}^\infty V(\beta,\phi_{k})\,W_\lambda(\phi_{k})=\lambda\,\sum_{k=1}^\infty \frac{\rho_{k}\,\phi_{k}\,V(\beta_0,\phi_{k})}{1+\lambda\rho_{k}}\, ,
\end{align*}
and, using Assumption~\ref{a3.3}, we conclude
\begin{align*}
\|W_\lambda(\beta_0)\|_K&=\lambda\,\bigg\{\sum_{k=1}^\infty\frac{\rho_{k}^2\,V^2(\beta_0,\phi_{k})}{1+\lambda\rho_{k}}\bigg\}^{1/2}\leq\lambda\,\bigg\{\sum_{k=1}^\infty \rho_k^2\,V^2(\beta_0,\phi_k)\bigg\}^{1/2}=O(\lambda)\,.
\end{align*}
Note that by Lemma~\ref{lem:3.1} in Section~\ref{app:aux:lem} and Assumption~\ref{a:rate},
\begin{align}\label{bias}
\sqrt n\lambda^{(2a+1)/(2D)}\|W_\lambda(\beta_0)\|_{L^2}&\leq c\sqrt n\lambda^{(2a+1)/(4D)}\|W_\lambda(\beta_0)\|_{K}\notag\\
&=O\big(\sqrt n\lambda^{1+(2a+1)/(4D)}\big)=o(1)\,.
\end{align}
Next, applying Theorem~\ref{thm:bahadur} and Lemma~\ref{lem:3.1} in Section~\ref{app:aux:lem}, we find
\begin{align}\label{t}
&\sup_{\nu\in[\nu_0,1]}\,\bigg\|\sqrt n\lambda^{(2a+1)/(2D)}\bigg[\nu\{\hat\beta_{n,\lambda}(\cdot,\nu)-\beta_0+W_\lambda(\beta_0)\}-\frac{1}{n}\sum_{i=1}^{\lfloor n\nu\rfloor}\e_i\tau_{\lambda}(X_i)\bigg]\bigg\|_{L^2}^2\notag\\
&=n\,\lambda^{(2a+1)/D}\sup_{\nu\in[\nu_0,1]}\,\bigg\|\nu\{\hat\beta_{n,\lambda}(\cdot,\nu)-\beta_0+W_\lambda(\beta_0)\}-\frac{1}{n}\sum_{i=1}^{\lfloor n\nu\rfloor}\e_i\tau_{\lambda}(X_i)\bigg\|_{L^2}^2\notag\\
&\leq c_K\,n\,\lambda^{(2a+1)/(2D)}\sup_{\nu\in[\nu_0,1]}\,\bigg\|\nu\{\hat\beta_{n,\lambda}(\cdot,\nu)-\beta_0+W_\lambda(\beta_0)\}-\frac{1}{n}\sum_{i=1}^{\lfloor n\nu\rfloor}\e_i\tau_{\lambda}(X_i)\bigg\|_{K}^2\notag\\
&= O_p\big(n\,\lambda^{(2a+1)/(2D)}\,v_n^2\big)\notag\\
&=O_p\Big\{n\,\lambda^{(2a+1)/(2D)}\times n^{-1}\lambda^{-2\varsigma}\big(\lambda+n^{-1}\lambda^{-(2a+1)/(2D)}\big)(\log n)\Big\}\notag\\
&=O_p\Big\{\lambda^{-2\varsigma+(2a+1)/(2D)}\big(\lambda+n^{-1}\lambda^{-(2a+1)/(2D)}\big)(\log n)\Big\}\notag\\
&=O_p\Big\{\big(\lambda^{-2\varsigma+(2D+2a+1)/(2D)}+n^{-1}\lambda^{-2\varsigma}\big)(\log n)\Big\}=o_p(1)\,,
\end{align}
%
where we used  Assumption~\ref{a:rate} in the last step.
By Theorem~\ref{thm:wip}, there exists a 
Gaussian process
$\{\Gamma(s,\nu)\}_{s\in[0,1],\nu\in[\nu_0,1]}$ in $\mathcal F$ defined in  the set \eqref{f} such that
\begin{align*}
\sup_{\nu\in[0,1]}\,\bigg\|n^{-1/2}\lambda^{(2a+1)/(2D)}\sum_{i=1}^{\lfloor n\nu\rfloor}\e_i\tau_{\lambda}(X_i)-\Gamma(\cdot,\nu)\bigg\|_{L^2}^2=\sup_{\nu\in[0,1]}\,\bigg\|\frac{1}{\sqrt n}\sum_{i=1}^{\lfloor n\nu\rfloor}U_i-\Gamma(\cdot,\nu)\bigg\|_{L^2}^2=o_p(1)\,.
\end{align*}
Combining the above finding with \eqref{bias} and \eqref{t} yields
\begin{align}\label{gamma}
&\sup_{\nu\in[\nu_0,1]}\,\bigg\|\sqrt n\lambda^{(2a+1)/(2D)}\nu\big\{\hat\beta_{n,\lambda}(\cdot,\nu)-\beta_0\big\}-\Gamma(\cdot,\nu)\bigg\|_{L^2}^2\notag\\
&\leq3\sup_{\nu\in[\nu_0,1]}\,\bigg\|\sqrt n\lambda^{(2a+1)/(2D)}\bigg[\nu\{\hat\beta_{n,\lambda}(\cdot,\nu)-\beta_0+W_\lambda(\beta_0)\}-\frac{1}{n}\sum_{i=1}^{\lfloor n\nu\rfloor}\e_i\tau_{\lambda}(X_i)\bigg]\bigg\|_{L^2}^2\notag\\
&\quad+3\sup_{\nu\in[\nu_0,1]}\,\bigg\|\sqrt n\lambda^{(2a+1)/(2D)}\nu W_\lambda(\beta_0)\bigg\|_{L^2}^2\notag\\
&\quad+3\sup_{\nu\in[0,1]}\,\bigg\|n^{-1/2}\lambda^{(2a+1)/(2D)}\sum_{i=1}^{\lfloor n\nu\rfloor}\e_i\tau_{\lambda}(X_i)-\Gamma(\cdot,\nu)\bigg\|_{L^2}^2=o_p(1)\,.
\end{align}
Recall  the definition of  $C_{U}$ in Assumption~\ref{a34}, and let $\{\kappa_j\}_{j=1}^\infty$ and $\{\psi_j\}_{j=1}^\infty$ denote the eigenvalues and eigenfunctions of $C_{U}$, respectively, that is,
\begin{align}\label{cueigen}
C_{U}(s,t)=\sum_{j=1}^\infty\kappa_j\,\psi_j(s)\,\psi_j(t)\,,\qquad \int_0^1 C_{U}(s,t)\,\psi_j(s)\,ds=\kappa_j\,\psi_j(t)\,.
\end{align} 
Following Theorem~1.1 in \cite{berkes2013}, we have
\begin{align*}
\Gamma(s,\nu)=\sum_{j=1}^\infty\sqrt{\kappa_j}\,\psi_j(s)\,W_j(\nu)~\qquad s\in[0,1]\,,\ \nu\in[\nu_0,1]\,,
\end{align*}
where the $W_j$'s are i.i.d.~standard Brownian motions on $[0,1]$. Note that $\E\big\{\sup_{\nu\in[\nu_0,1]}W_j^2(\nu)\big\}<\infty$, so that
\begin{align}\label{intgamma}
\E\bigg\{\sup_{\nu\in[0,1]}\|\Gamma(\cdot,\nu)\|_{L^2}^2\bigg\}\leq\sum_{j=1}^\infty\kappa_j\,\E\bigg\{\sup_{\nu\in[0,1]}W_j^2(\nu)\bigg\}<\infty\,.
\end{align}
Furthermore, observing \eqref{gamma}, we deduce from direct calculations that
\begin{align}\label{1234}
&\hat\G_n(\nu)=\sqrt n\lambda^{(2a+1)/(2D)}\nu^2 \int_0^1\big\{\hat\beta_{n,\lambda}^{\,2}(s, \nu)- \beta_0^{2}(s)\big\}\,ds\notag\\
&=\sqrt{n}\lambda^{(2a+1)/(2D)}\Bigg[\nu^2\int_0^1\big\{\hat\beta_{n,\lambda}(s, \nu)- \beta_0(s)\big\}^{2}\, ds+2\nu\int_0^1\nu\big\{\hat\beta_{n,\lambda}(s, \nu)- \beta_0(s)\big\}\,\beta_0(s)\, ds\Bigg]\notag\\
&=2 \nu\int_0^1  \beta_0(s)\, \Gamma(s,\nu)\,ds+I_{1,n}(\nu)+I_{2,n}(\nu)+I_{3,n}(\nu)+I_{4,n}(\nu)\,,
\end{align}
where
\begin{align*}
I_{1,n}(\nu)&=\frac{1}{\sqrt{n}\lambda^{(2a+1)/(2D)}} \int_0^1\bigg[\sqrt n\lambda^{(2a+1)/(2D)}\nu \big\{\hat\beta_{n,\lambda}(s,\nu)-\beta_0(s)\big\}-\Gamma(s, \nu)\bigg]^{2}\, ds\,,\\
I_{2,n}(\nu)&=\frac{1}{\sqrt{n}\lambda^{(2a+1)/(2D)}} \int_0^1 \Gamma^{2}(s, \nu)\, ds\,, \\
I_{3,n}(\nu)&=\frac{2}{\sqrt{n}\lambda^{(2a+1)/(2D)}} \int_0^1\bigg[\sqrt n\lambda^{(2a+1)/(2D)}\nu\big\{\hat\beta_{n,\lambda}(s,\nu)-\beta_0(s)\big\}-\Gamma(s, \nu)\bigg] \Gamma(s, \nu)\, ds \,,\\
I_{4,n}(\nu)&=2 \nu \int_0^1 \beta_0(s)\bigg[\sqrt n\lambda^{(2a+1)/(2D)}\nu\big\{\hat\beta_{n,\lambda}(s,\nu)-\beta_0(s)\big\}-\Gamma(s, \nu)\bigg]\,ds\,.
\end{align*}
Note that \eqref{intgamma} implies that $\sup_{\nu\in[0,1]}\|\Gamma(\cdot,\nu)\|_{L^2}^2=O_p(1)$. Therefore, observing that $n\lambda^{(2a+1)/D}\to\infty$ as $n\to\infty$ (see Assumption~\ref{a:rate}),  \eqref{gamma} and the Cauchy-Schwarz inequality, it follows that
\begin{align*}
\sup_{\nu\in[\nu_0,1]}\big\{|I_{1,n}(\nu)|+|I_{2,n}(\nu)|+|I_{3,n}(\nu)|+|I_{4,n}(\nu)|\big\}=o_p(1)\,
\end{align*}
 as $n\to\infty$.
Consequently, in view of \eqref{1234}, we have that, 
\begin{align}\label{main}
\sup_{\nu\in[\nu_0,1]}\,\bigg|\hat\G_n(\nu)-2 \nu\int_0^1  \beta_0(s)\, \Gamma(s,\nu)\,ds\bigg|=o_p(1)\,.
\end{align}
This proves the finite-dimensional convergence, that is, for any $k\in\mathbb N_+$ and $\nu_1,\ldots,\nu_k\in[\nu_0,1]$,
\begin{align}\label{finite}
\big(\hat\G_n(\nu_1),\ldots,\hat\G_n(\nu_k)\big)\converged\bigg(2 \nu_1\int_0^1  \beta_0(s)\, \Gamma(s,\nu_1)\,ds\,,\ldots,\,2 \nu_k\int_0^1  \beta_0(s)\, \Gamma(s,\nu_k)\,ds\bigg)\,.
\end{align}
Next, we shall show the tightness of the process $\{\hat\G_n(\nu)\}_{\nu\in[\nu_0,1]}$. To achieve this, we shall show that the process $\{\hat\G_n(\nu)\}_{\nu\in[\nu_0,1]}$ is asymptotically uniformly equicontinuous in probability
(see Lemma~1.5.7 in \citealp{vaart1996}).
By the Cauchy-Schwarz inequality and \eqref{main}, we deduce that
\begin{align}\label{01}
&\sup_{\substack{|\nu_1-\nu_2|<\delta\\\nu_1,\nu_2\in[\nu_0,1]}}|\hat\G_n(\nu_1)-\hat\G_n(\nu_2)|\notag\\
&\leq\sup_{\substack{|\nu_1-\nu_2|<\delta\\\nu_1,\nu_2\in[\nu_0,1]}} \bigg|2 \nu_1\int_0^1  \beta_0(s)\, \Gamma(s,\nu_1)\,ds-2 \nu_2\int_0^1  \beta_0(s)\, \Gamma(s,\nu_2)\,ds\bigg|\notag\\
&\hspace{2cm}+2\sup_{\nu\in[\nu_0,1]}\bigg|\hat\G_n(\nu)-2 \nu\int_0^1  \beta_0(s)\, \Gamma(s,\nu)\,ds\bigg|\notag\\
&\leq\sup_{\substack{|\nu_1-\nu_2|<\delta\\\nu_1,\nu_2\in[\nu_0,1]}} \bigg\{2|\nu_1-\nu_2|\times\bigg|\int_0^1  \beta_0(s)\, \Gamma(s,\nu_1)\,ds\bigg|\bigg\}\notag\\
&\hspace{2cm}+2\sup_{\substack{|\nu_1-\nu_2|<\delta\\\nu_1,\nu_2\in[\nu_0,1]}}\bigg|\int_0^1\beta_0(s)\{\Gamma(s,\nu_1)-\Gamma(s,\nu_2)\}ds\bigg|+o_p(1)\notag\\
&\leq2\delta\,\|\beta_0\|_{L^2}\sup_{\nu\in[\nu_0,1]}\bigg|\int_0^1\Gamma^2(s,\nu)\,ds\bigg|^{1/2}\notag\\
&\hspace{2cm}+2\|\beta_0\|_{L^2}\sup_{\substack{|\nu_1-\nu_2|<\delta\\\nu_1,\nu_2\in[0,1]}}\bigg|\int_0^1\{\Gamma(s,\nu_1)-\Gamma(s,\nu_2)\}^2\,ds\bigg|^{1/2}+o_p(1)\,.
\end{align}
By Lemma~2.1 in \cite{berkes2013}, we have
\begin{align}\label{02}
\sup_{\nu\in[0,1]}\int_0^1\Gamma^2(s,\nu)\,ds<\infty\qquad\as
\end{align}
In addition, in view of \eqref{cueigen}, note that the $\psi_j$'s are orthogonal in $L^2([0,1])$, so that
\begin{align}\label{03}
&\sup_{\substack{|\nu_1-\nu_2|<\delta\\\nu_1,\nu_2\in[0,1]}}\int_0^1\{\Gamma(s,\nu_1)-\Gamma(s,\nu_2)\}^2\,ds\notag\\
&=\sup_{\substack{|\nu_1-\nu_2|<\delta\\\nu_1,\nu_2\in[0,1]}}\int_0^1\bigg[\sum_{j=1}^\infty\sqrt{\kappa_j}\,\psi_j(s)\,\{W_j(\nu_1)-W_j(\nu_2)\}\bigg]^2\,ds\notag\\
&=\sup_{\substack{|\nu_1-\nu_2|<\delta\\\nu_1,\nu_2\in[0,1]}}\,\sum_{j=1}^\infty\kappa_j\,\{W_j(\nu_1)-W_j(\nu_2)\}^2\notag\\
&\leq\sup_{\substack{|\nu_1-\nu_2|<\delta\\\nu_1,\nu_2\in[0,1]}}\,\{W_1(\nu_1)-W_1(\nu_2)\}^2\,\sum_{j=1}^\infty\kappa_j=o_p(1)\,,
\end{align}
where the last step is due to the modulus of continuity
of Brownian motions and the fact that $\sum_{j=1}^\infty\kappa_j<\infty$. Therefore, combining \eqref{01}--\eqref{03}, we deduce that, for any $e>0$,
\begin{align*}
\lim_{\delta\downarrow0}\,\limsup_{n\to\infty}\,\P\Bigg\{\sup_{\substack{|\nu_1-\nu_2|<\delta\\\nu_1,\nu_2\in[\nu_0,1]}}|\hat\G_n(\nu_1)-\hat\G_n(\nu_2)|>e\Bigg\}=0\,.
\end{align*}
This proves the tightness of the process $\{\hat\G_n(\nu)\}_{\nu\in[\nu_0,1]}$. Together with \eqref{finite}, by Lemma~1.5.4 in \cite{vaart1996}, this implies that
\begin{align*}
\{\hat\G_n(\nu)\}_{\nu\in[\nu_0,1]} \weakconverge \bigg\{2\nu \int_0^1  \beta_0(s)\, \Gamma(s, \nu)\, ds\bigg\}_{\nu\in[\nu_0,1]}\hspace{1cm}\text{in }\ell^\infty([\nu_0,1])\,.
\end{align*}
In addition, observing \eqref{cueigen}, we have for  the Gaussian process $\{\Gamma(s,\nu)\}_{s\in[0,1],\nu\in[\nu_0,1]}$,
\begin{align*}
&\cov\{\Gamma(s_1, \nu_1),\Gamma(s_2, \nu_2) \}=\cov\bigg\{\sum_{j_1=1}^\infty\sqrt{\kappa_{j_1}}\,\psi_{j_1}(s)\,W_{j_1}(\nu_1),\sum_{{j_2}=1}^\infty\sqrt{\kappa_{j_2}}\,\psi_{j_2}(s)\,W_{j_2}(\nu_2)\bigg\}\\
&=\cov\{W_j(\nu_1),W_j(\nu_2)\}\sum_{j=1}^\infty\kappa_j\,\psi_j(s_1)\,\psi_j(s_2)=(\nu_1\wedge\nu_2)\,C_{U}(s_1,s_2)\,,
\end{align*}
where   $C_{U}$ is defined in Assumption~\ref{a34} and 
$s_1,s_2\in[0,1]$ and $\nu_1,\nu_2\in[\nu_0,1]$.
Hence, we deduce that
\begin{align*}
&\cov\left\{\int_0^1 \beta_0(s_1)\, \Gamma(s_1, \nu_1) \,ds\,, \int_0^1 \beta_0(s_2) \Gamma(s_2, \nu_2) d s_2\right\}\\
&=\int_0^1\int_0^1\cov\{\Gamma(s_1, \nu_1),\Gamma(s_2, \nu_2) \}\,\beta_0(s_1)\,\beta_0(s_2)\,ds_1\,ds_2\\
&=(\nu_1 \wedge\nu_2) \int_0^1\int_0^1C_{U}(s_1,s_2)\,\beta_0(s_1)\beta_0(s_2)\,ds_1\,ds_2=(\nu_1 \wedge\nu_2)\,\sigma_d^2\,.
\end{align*}
where $\sigma_d^2$ is defined in \eqref{sigmad2}. This implies that
\begin{align*}
\int_0^1 \beta_0(s)\, \Gamma(s, \nu) \,ds \stackrel{d.}{=}2\sigma_d\,\BB(\nu)\,,
\end{align*}
where $\BB$ denotes a standard Brownian motion. Hence, combining the above finding with \eqref{main} yields
\begin{align*}
\{\hat\G_n(\nu)\}_{\nu\in[\nu_0,1]}\weakconverge \{2\sigma_d\,\nu\,\BB(\nu)\}_{\nu\in[\nu_0,1]}
\qquad \text{in }\ell^\infty([\nu_0,1])\,,
\end{align*}
which completes the proof.

\subsection{Proof of Theorem~\ref{thm:pivot}}\label{proof:thm:pivot}

Observing $\hat\G_n(\nu)$ defined in \eqref{gnnu}, we have
\begin{align*}
&\sqrt{n}\lambda^{(2a+1)/(2D)}\hat\VV_n\\
&=\sqrt{n}\lambda^{(2a+1)/(2D)}\bigg[\int_{\nu_0}^1\bigg|\nu^2\int_0^1\big\{\hat\beta^{\,2}_{n,\lambda}(s,\nu)-\hat\beta^{\,2}_{n,\lambda}(s,1)\big\}\,ds\bigg|^2\,\omega(d\nu)\bigg]^{1/2}\\
&=\sqrt{n}\lambda^{(2a+1)/(2D)}\bigg[\int_{\nu_0}^1\bigg|\nu^2\int_0^1\big\{\hat\beta^{\,2}_{n,\lambda}(s,\nu)-\beta_0^2(s)\big\}\,ds-\nu^2\int_0^1\big\{\hat\beta^{\,2}_{n,\lambda}(s,1)-\beta_0^2(s)\big\}\,ds\bigg|^2\,\omega(d\nu)\bigg]^{1/2}\\
%
%
&=\bigg\{\int_{\nu_0}^1\big|\hat\G_n(\nu)-\nu^2\,\hat\G_n(1)\big|^2\,\omega(d\nu)\bigg\}^{1/2}\,.
\end{align*}
In addition, for $\hat\TT_n$ defined in \eqref{tn},
\begin{align*}
&\sqrt{n}\lambda^{(2a+1)/(2D)}(\hat\TT_n-d_0)=\sqrt{n}\lambda^{(2a+1)/(2D)}\int_0^1\big\{\hat\beta^{\,2}_{n,\lambda}(s,1)-\beta_0^2(s)\big\}\,ds=\hat\G_n(1)\,.
\end{align*}
Therefore, by the continuous mapping theorem we obtain that
\begin{align*}
&\sqrt{n}\lambda^{(2a+1)/(2D)}\big((\hat\TT_n-d_0)\,,\hat\VV_n\big)=\Bigg(\hat\G_n(1)\,, \bigg\{\int_{\nu_0}^1\big|\hat\G_n(\nu)-\nu^2\,\hat\G_n(1)\big|^2\,\omega(d\nu)\bigg\}^{1/2}\Bigg)\,.
\end{align*}
Observing that the map from $\ell^\infty([\nu_0,1])$ to $\mathbb R$ defined by
\begin{align*}
\mathbb H\mapsto\frac{\mathbb H(1)}{\scaleobj{1.2}{\big\{}\int_{\nu_0}^1\big|\mathbb H(\nu)-\nu^2\,\mathbb H(1)\big|^2\,\omega(d\nu)\scaleobj{1.2}{\big\}^{\scaleobj{0.83}{1/2}}}}
\end{align*}
is continuous as long as $\int_{\nu_0}^1\big|{\mathbb H}(\nu)-\nu^2\,{\mathbb H}(1)\big|^2\omega(d\nu)\neq0$. Since $\sigma_d^2\neq0$, we therefore deduce from Theorem~\ref{thm:2.1} and the continuous mapping theorem that
\begin{align*}
\frac{\hat\TT_n-d_0}{\hat\VV_{n}}\converged\frac{2\sigma_d\,\BB(1)}{2\sigma_d\scaleobj{1.2}{\big\{}\int_{\nu_0}^1|\nu\,\BB(\nu)-\nu^2\,\BB(1)|^2\,\omega(d\nu)\scaleobj{1.2}{\big\}^{\scaleobj{0.83}{1/2}}}}=\frac{\BB(1)}{\scaleobj{1.2}{\big\{}\int_{\nu_0}^1|\nu\,\BB(\nu)-\nu^2\,\BB(1)|^2\,\omega(d\nu)\scaleobj{1.2}{\big\}^{\scaleobj{0.83}{1/2}}}}=\mathbb W\,,
\end{align*}
which completes the proof.

\subsection{Proof of Theorem~\ref{thm:one}}\label{proof:thm:one}

When $d_0=\int_0^1|\beta_0(t)|^2dt=0$, we have $\hat\TT_n=\hat\TT_n-d_0=o_p(1)$ and $\hat\VV_n=o_p(1)$, which implies that $\hat\TT_n- \mathcal Q_{1-\alpha}(\WW)\hat\VV_n=o_p(1)$, so that
\begin{align*}
\lim_{n\to\infty}\P\big\{\hat\TT_n> \mathcal Q_{1-\alpha}(\WW)\hat\VV_n+\Delta\big\}=\lim_{n\to\infty}\P\big\{\hat\TT_n- \mathcal Q_{1-\alpha}(\WW)\hat\VV_n>\Delta\big\}=0\,.
\end{align*}
When $0<\int_0^1|\beta_0(t)|^2dt<\Delta$, we have
\begin{align*}
\lim_{n\to\infty}\P\big\{\hat\TT_n> \mathcal Q_{1-\alpha}(\WW)\hat\VV_n+\Delta\big\}=\lim_{n\to\infty}\P\bigg\{\frac{\hat\TT_n-d_0}{\hat\VV_n}> \mathcal Q_{1-\alpha}(\WW)+\frac{\sqrt{n}\lambda^{(2a+1)/(2D)}(\Delta-d_0)}{\sqrt{n}\lambda^{(2a+1)/(2D)}\hat\VV_n}\bigg\}\,.
\end{align*}
Note that $\sqrt{n}\lambda^{(2a+1)/(2D)}\hat\VV_n=O_p(1)$ and $\sqrt{n}\lambda^{(2a+1)/(2D)}\to+\infty$ as $n\to\infty$ according to Assumption~\ref{a:rate}, so that $\sqrt{n}\lambda^{(2a+1)/(2D)}(\Delta-d_0)\to+\infty$. Hence, the result in \eqref{eq:rele} in the case where $\int_0^1|\beta_0(t)|^2dt<\Delta$ follows.

When $d_0=\int_0^1|\beta_0(t)|^2dt=\Delta$, we have
\begin{align*}
\lim_{n\to\infty}\P\big\{\hat\TT_n> \mathcal Q_{1-\alpha}(\WW)\hat\VV_n+\Delta\big\}=\lim_{n\to\infty}\P\bigg\{\frac{\hat\TT_n-d_0}{\hat\VV_n}> \mathcal Q_{1-\alpha}(\WW)\bigg\}=\alpha\,.
\end{align*}

When $d_0=\int_0^1|\beta_0(t)|^2dt>\Delta$, we have $\sqrt{n}\lambda^{(2a+1)/(2D)}(\Delta-d_0)\to-\infty$ as $n\to\infty$, so that
\begin{align*}
\P\big\{\hat\TT_n> \mathcal Q_{1-\alpha}(\WW)\hat\VV_n+\Delta\big\}=\P\bigg\{\frac{\hat\TT_n-d_0}{\hat\VV_n}> \mathcal Q_{1-\alpha}(\WW)+\frac{\sqrt{n}\lambda^{(2a+1)/(2D)}(\Delta-d_0)}{\sqrt{n}\lambda^{(2a+1)/(2D)}\hat\VV_n}\bigg\}\to0\,,
\end{align*}
which completes the proof.

\subsection{Proof of the results in Section~\ref{sec:two}}\label{app:two}

\subsubsection{Remaining assumptions for Theorem~\ref{thm:gnt}}\label{app:a:two}

For $j=1,2$, let $C_{X,j}(s,t)=\cov\{X_{j,0}(s),X_{j,0}(t)\}$ denote the covariance function of the predictor of the $j$-th sample. For any $\beta_1,\beta_2\in\H$ and for $j=1,2$, define
\begin{align*}
V_j(\beta_1,\beta_2)=\int_0^1\int_0^1C_{X,j}(s,t)\,\beta_1(s)\,\beta_2(t)\,ds\,dt\,.
\end{align*}
For each $j=1,2$ and for $\beta_1,\beta_2\in\H$, let
\begin{align*}
\l\beta_1,\beta_2\r_j=V_j(\beta_1,\beta_2)+\lambda_j J(\beta_1,\beta_2)
\end{align*}
define an inner product on $\H$ and let $\|\cdot\|_j$ denotes its corresponding norm.

\begin{assumption}\label{a1t}
The covariance functions $C_{X,1}$ and $C_{X,2}$ are continuous on $[0,1]^2$. For any $\gamma\in L^2([0,1])$ and for each $j=1,2$, $\int_0^1 C_{X,j}(s,s')\gamma(s)ds\equiv0$ implies that $\gamma\equiv0$.  
\end{assumption}

\begin{assumption}\label{a201t}
There exists a sequence of functions $\{\phi_{j,k}\}_{j=1,2;k\geq1}$ in $\H$, such that $\Vert \phi_{j,k}\Vert_{\infty}\leq c\, k^{a_j}$ for any $k\geq 1$ and $j=1,2$, and
\begin{align*}
V_j(\phi_{j,k},\phi_{j,k'})=\delta_{kk'}\,,\qquad J(\phi_{j,k},\phi_{j,k'})=\rho_{j,k}\,\delta_{kk'}\,,
\end{align*}
where $a_j\geq 0$, $c>0$ are constants, $\delta_{kk'}$ is the Kronecker delta and $\rho_{j,k}$ is such that $\rho_{j,k}\asymp k^{2D_j}$ for some constant $D_j>a_j+1/2$. Furthermore, for $j=1,2$, any $\beta\in\H$ admits the expansion $\beta=\sum_{k=1}^\infty V_j(\beta,\phi_{j,k})\phi_{j,k}$ with convergence in $\H$ with respect to the norm $\Vert\cdot\Vert_j$.
\end{assumption}

\begin{assumption}\label{a:subgt}~

\begin{enumerate}[label={\rm(\ref*{a:subgt}.\arabic*)},ref={\rm\ref*{a:subgt}.\arabic*},series=a3t,nolistsep,leftmargin=1.6cm]
\item\label{a3.1t} There exists a constant $\varpi>0$ such that $\max_{j=1,2}\E\{\exp(\varpi\|X_{j,0}\|_{L^2}^2)\}<\infty$.

\item\label{a3.2t} 
For any $\beta\in\H$, $\E\big(\l X_{j,0},\beta\r_{L^2}^4\big)\leq c_0\big\{\E\big(\l X_{j,0},\beta\r_{L^2}^2\big)\big\}^2$, for some constant $c_0>0$ and for $j=1,2$.

\item\label{a3.3t}The true slope functions $\beta_1$ and $\beta_2$ are such that $\sum_{k=1}^\infty \rho_{1,k}^2V_1^2(\beta_1,\phi_{1,k})<\infty$ and $\sum_{k=1}^\infty \rho_{2,k}^2V_2^2(\beta_2,\phi_{2,k})<\infty$ 
\item For $(s,t)\in[0,1]^2$ and $j=1,2$, the limit $C_{U,j}(s,t)=\lim_{\lambda\downarrow0}C_{U,\lambda,j}(s,t)$ exists.

\end{enumerate}
\end{assumption}

\begin{assumption}\label{a:ratet}
For the constants $a_1,a_2,D_1,D_2$ in Assumption~\ref{a201t}
and the regularization parameters
$\lambda_1$ and $\lambda_2$ in \eqref{betat},  for $j=1,2$, $\lambda_j=o(1)$, $n_j^{-1}\lambda_j^{-(2a_j+1)/D_j}=o(1)$, $ n_j\lambda_j^{2+(2a_j+1)/(2D_j)}=o(1)$. In addition, for $j=1,2$, $n_j^{-1}\lambda_j^{-2\varsigma_j}\log n=o(1)$ and $\lambda_j^{-2\varsigma_j+(2D_j+2a_j+1)/(2D_j)}\log n=o(1)$  as $n\to\infty$, where $\varsigma_j=(2D_j-2a_j-1)/(4D_jm)+(a_j+1)/(2D_j)>0$.
\end{assumption}

\begin{assumption}\label{a:mt}~
For all $i\in\mathbb Z$, $(X_{j,i},Y_{j,i})$ is generated by the model \eqref{model01} and follows the following assumptions.
\begin{enumerate}[label={\rm(\ref*{a:mt}.\arabic*)},ref={\rm\ref*{a:mt}.\arabic*},series=Qbetat,nolistsep,leftmargin=1.6cm]

\item For all $i\in\mathbb Z$ and $j=1,2$,
$X_{j,i}=g_j(\ldots,\xi_{j,i-1},\xi_{j,i})$ and $\e_{j,i}=h_j(\ldots,\eta_{j,i-1},\eta_{j,i})$, for some deterministic measurable functions $g_j:\mathcal S^{\infty}\mapsto L^2([0,1])$ and $h_j:\mathbb{R}\mapsto\mathbb{R}$, where $\mathcal S$ is some measurable space,  $\xi_{j,i}=\xi_{j,i}(t,\upomega)$ is jointly measurable in $(t,\upomega)$, for $i\in\mathbb Z$ and $j=1,2$. The $\xi_{j,i}$'s and the $\eta_{j,i}$'s are i.i.d.

\item For any $s\in[0,1]$ and $j=1,2$, $\E\{X_{j,0}(s)\}=\E(\e_{j,0})=0$. For some $\delta\in(0,1)$ and $j=1,2$, $\E|\e_{j,0}|^{2+\delta}<\infty$.


\item For $j=1,2$, the sequences $ \{X_{j,i}\}_{i\in\mathbb Z}$ and $\{\e_{j,i}\}_{i\in\mathbb Z}$ can be approximated by $\ell$-dependent sequences $\{X_{j,i,\ell}\}_{i,\ell\in\mathbb Z}$ and $\{\e_{j,i,\ell}\}_{i,\ell\in\mathbb Z}$, respectively, in the sense that, for some $\kappa>2+\delta$,
\begin{align*}
\sum_{\ell=1}^\infty\big(\E\|X_{j,i}-X_{j,i,\ell}\|_{L^2}^{2+\delta}\big)^{1/\kappa}<\infty\,,\qquad\sum_{\ell=1}^\infty\big(\E|\e_{j,i}-\e_{j,i,\ell}|^{2+\delta}\big)^{1/\kappa}<\infty\,.
\end{align*}
Here, $X_{j,i,\ell}=g(\xi_{j,i},\xi_{j,i-1},\ldots,\xi_{j,i-\ell+1},\boldsymbol{\xi}_{j,i,\ell}^*)$ and $\e_{j,i,\ell}=h(\eta_{j,i},\eta_{j,i-1},\ldots,\eta_{j,i-\ell+1},\boldsymbol{\eta}_{j,i,\ell}^*)$, where $\boldsymbol{\xi}_{j,i,\ell}^*=(\xi^*_{j,i,\ell,i-\ell},\xi^*_{j,i,\ell,i-\ell-1},\ldots)$ and $\boldsymbol{\eta}_{j,i,\ell}^*=(\eta^*_{j,i,\ell,i-\ell},\eta^*_{j,i,\ell,i-\ell-1},\ldots)$ and where the $\xi^*_{j,i,\ell,k}$'s and the $\eta^*_{j,i,\ell,k}$'s are independent copies of $\xi_0$ and $\eta_0$, and are independent of $\{\xi_{j,i}\}_{i\in\mathbb Z}$ and $\{\eta_{j,i}\}_{i\in\mathbb Z}$, respectively.

\end{enumerate}

\end{assumption}

\subsubsection{Proof of Theorem~\ref{thm:gnt}}\label{proof:thm:gnt}

For $\beta_1,\beta_2\in\mathcal{H}$, let $W_{\lambda_j}$ denote the linear self-adjoint operator such that $\l W_{\lambda_j}\beta_1,\beta_2\r_j=\lambda_j J(\beta_1,\beta_2)$.
Let $\|\cdot\|_j$ define the norm corresponding to the inner product $\l\cdot,\cdot\r_j$ defined in \eqref{lrj}.
For $j=1,2$, let
\begin{align*}
v_{n_j}=n_j^{-1/2}\lambda_j^{-(2D_j-2a_j-1)/(4D_jm)-(a_j+1)/(2D_j)}(\lambda_j^{1/2}+n_j^{-1/2}\lambda_j^{-(2a_j+1)/(4D_j)})(\log n_j)^{1/2}\,.
\end{align*}
\noindent{\bf Step 1.} By Theorem~\ref{thm:bahadur}, we have, for $j=1,2$,
\begin{align*}
&\sup_{\nu\in[\nu_0,1]}\,\bigg\Vert\nu\Big\{\hat\beta_{n_j,\lambda_j}(\cdot,\nu)-\beta_1+W_{\lambda_j}(\beta_j)\Big\}-\frac{1}{n_j}\sum_{i=1}^{\lfloor n_j\nu \rfloor}\e_{j,i}\,\tau_j(X_{j,i})\bigg\Vert_j=O_p(v_{n_j})\,.
\end{align*}
In addition, observing \eqref{bias}, for $j=1,2$, $\sqrt n_j\lambda_j^{(2a_j+1)/(2D_j)}\|W_{\lambda_j}(\beta_j)\|_{L^2}=o(1)$. Then, by Lemma~\ref{lem:3.1} and Assumption~\ref{a:ratet},
\begin{align*}
&\sup_{\nu\in[\nu_0,1]}\,\bigg\Vert\nu\Big\{\hat\beta_{n_1,n_2}(\cdot,\nu)-(\beta_1-\beta_2)\Big\}-\frac{1}{n_1}\sum_{i=1}^{\lfloor n_1\nu \rfloor}\e_{1,i}\,\tau_1(X_{1,i})+\frac{1}{n_2}\sum_{i=1}^{\lfloor n_2\nu \rfloor}\e_{2,i}\,\tau_2(X_{2,i})\bigg\Vert_{L^2}\\
&=O_p\big(\lambda_1^{-(2a_1+1)/(4D_1)}v_{n_1}+\lambda_2^{-(2a_2+1)/(4D_2)}v_{n_2}\big)=o_p(1)\,.
\end{align*}
\noindent{\bf Step 2.} By Theorem~\ref{thm:wip}, there exists a Gaussian process $\{\tilde\Gamma(s,\nu)\}_{s\in[0,1],\nu\in[0,1]}$ in $\mathcal F$ defined in \eqref{f}, with covariance function
\begin{align*}
\cov\big\{\tilde\Gamma(s_1,\nu_1),\tilde\Gamma(s_2,\nu_2)\big\}=(\nu_1\wedge\nu_2)\big\{C_{U,1}(s_1,s_2)+\gamma\, C_{U,2}(s_1,s_2)\big\}\,,
\end{align*}
such that
\begin{align*}
\sup_{\nu\in[0,1]}\,\Bigg\|\frac{1}{\sqrt{n_1}\lambda_1^{(2a_1+1)/(2D_1)}}\sum_{i=1}^{\lfloor n_1\nu \rfloor}\e_{1,i}\,\tau_1(X_{1,i})+\frac{1}{\sqrt{n_2}\lambda_2^{(2a_2+1)/(2D_2)}}\sum_{i=1}^{\lfloor n_2\nu \rfloor}\e_{2,i}\,\tau_2(X_{2,i})&-\tilde\Gamma(\cdot,\nu)\Bigg\|_{L^2}^2\\
&=o_p(1)\,.
\end{align*}

\noindent{\bf Step 3.} Following the proof of Theorem~\ref{thm:2.1}, we deduce that
\begin{align*}
\sup_{\nu\in[\nu_0,1]}\bigg|\hat\G_{n_1,n_2}(\nu)-2\nu\int_0^1\{\beta_1(s)-\beta_2(s)\}\,\tilde\Gamma(s,\nu)\,ds\bigg|=o_p(1)\,.
\end{align*}
In addition, for the $\sigma_{1,2}$ defined in \eqref{sigma12},
\begin{align*}
2\nu\int_0^1\{\beta_1(s)-\beta_2(s)\}\,\tilde\Gamma(s,\nu)\,ds\stackrel{d.}{=}2\sigma_{1,2}\,\nu\,\BB(\nu)\,.
\end{align*}
Moreover, the tightness of the process $\{\hat\G_{n_1,n_2}(\nu)\}_{\nu\in[\nu_0,1]}$ can be proved by following the proof of Theorem~\ref{thm:2.1} in Section~\ref{proof:thm:2.1}. The proof is therefore complete.

\subsection{Proof of Theorem~\ref{thm:gnf}}\label{app:thm:gnf}

It follows from Assumption~\ref{a201f} that, for any $\beta\in\H_f$ and $k,\ell\geq1$,
$$
\l\beta,\phi_{k\ell}\r_f=\sum_{k',\ell'=1}^\infty V_f(\beta,\phi_{k'\ell'})\l\phi_{k\ell},\phi_{k'\ell'}\r_f=(1+\lambda\rho_{k\ell})V_f(\beta,\phi_{k\ell})
$$
so that
\begin{align}\label{expansionf}
\beta=\sum_{k,\ell=1}^\infty V_f(\beta,\phi_{k\ell})\,\phi_{k\ell}=\sum_{k,\ell=1}^\infty \frac{\l\beta,\phi_{k\ell}\r_f}{1+\lambda\rho_{k\ell}}\,\phi_{k\ell}\,.
\end{align}
For $\beta_1,\beta_2\in\mathcal{H}_f$, let $W_\lambda^f$ denote the linear self-adjoint operator such that $\l W^f_\lambda\beta_1,\beta_2\r_f=\lambda J_f(\beta_1,\beta_2)$. By definition, 
the functions $\{\phi_{k\ell}\}_{k,\ell\geq 1}$ in Assumption~\ref{a201f} satisfy $\l W_\lambda^f\phi_{k\ell},\phi_{k'\ell'}\r_f=\lambda J_f(\phi_{k\ell},\phi_{k'\ell'})=\lambda\rho_{k\ell}\,\delta_{kk'}\,\delta_{\ell\ell'}$, so that in view of \eqref{expansionf},
\begin{align}\label{wphif}
W_\lambda^f(\phi_{k\ell})=\sum_{k',\ell'}\frac{\l W_\lambda^f(\phi_{k\ell}),\phi_{k'\ell'}\r_K}{1+\lambda\rho_{k'\ell'}}\phi_{k'\ell'}=\frac{\lambda\,\rho_{k\ell}\,\phi_{k\ell}}{1+\lambda\rho_{k\ell}}\,.
\end{align}
The following lemma establishes the uniform Bahadur representation for function-on-function linear regression. The proof follows similar arguments as the proof of Theorem~\ref{thm:bahadur} of this article and Theorem~3.1 in \cite{dettetang}, and is therefore omitted for the sake of brevity.

\begin{lemma}\label{lem:vn}
Suppose the assumptions of Theorem~\ref{thm:gnf} are satisfied. Then, for any fixed (but arbitrary) $\nu_0\in(0,1]$, we have
\begin{align*}
&\sup_{\nu\in[\nu_0,1]}\,\bigg\Vert\nu\Big\{\hat\beta_{n,\lambda}(\,\cdot\,;\nu)-\beta_0+W_\lambda^f(\beta_0)\Big\}-\frac{1}{n}\sum_{i=1}^{\lfloor n\nu \rfloor}\tau_\lambda^f (X_{i}\otimes\e_i)\bigg\Vert_f=O_p(\tilde v_{n})\,,
\end{align*}
where $\tilde v_{n}=n^{-1/2}\lambda^{-\varsigma_f}\big(\lambda^{1/2}+n^{-1/2}\lambda^{-(2a_f+1)/(4D_f)}\big)(\log n)^{1/2}$, and $W_\lambda^f$ is defined in \eqref{wphif}.
\end{lemma}

Let
\begin{align}\label{ff}
\mathcal F_f=\bigg\{g:[0,1]^2\times[0,1]\to\mathbb{R}\,\Big|\sup_{\nu\in[0,1]}\int_0^1\int_0^1|g(s,t;\nu)|^2\,ds\,dt<\infty\bigg\}\,.
\end{align}
The following lemma establishes the weak invariance principle in the context of function-on-function linear regression. The proof is in line with the proof of Theorem~\ref{thm:wip} and is therefore omitted for the sake of brevity.

\begin{lemma}\label{lem:vn2}
Suppose the assumptions of Theorem~\ref{thm:gnf} are satisfied. Then, there exists a Gaussian process $\{\Gamma_f(s,t;\nu)\}_{(s,t)\in[0,1]^2,\nu\in[0,1]}$ in $\mathcal F_f$ defined in \eqref{ff}, with covariance function 
\begin{align*}
\cov\big\{\Gamma_f(s_1,t_1;\nu_1)\,,\Gamma_f(s_2,t_2;\nu_2)\big\}=(\nu_1\wedge\nu_2)\,C_f\{(s_1,t_1),(s_2,t_2)\}\,,
\end{align*}
for $C_f$ in Assumption~\ref{a34f}, such that
\begin{align*}
\sup_{\nu\in[0,1]}\,\Bigg\|\frac{1}{\sqrt{n}\lambda^{(2a_f+1)/(2D_f)}}\sum_{i=1}^{\lfloor n\nu \rfloor}\tau_\lambda^f (X_i\otimes \e_{i})&-\Gamma_f(\cdot;\nu)\Bigg\|_{L^2}^2=o_p(1)\,.
\end{align*}
\end{lemma}
In addition, it can be shown that $\|W_\lambda^f(\beta_0)\|_{L^2}=o\big(n^{-1/2}\lambda^{-(2a_f+1)/(2D_f)}\big)$. Therefore, it can be deduced from Lemmas~\ref{lem:vn} and \ref{lem:vn2} that
\begin{align}\label{v1}
\sup_{\nu\in[\nu_0,1]}\bigg|\hat\G_{n_1,n_2}(\nu)-2\nu\int_0^1\int_0^1\beta_0(s,t)\,\tilde\Gamma(s,t;\nu)\,ds\bigg|=o_p(1)\,.
\end{align}
Moreover, for the $\sigma_{f}$ defined in \eqref{sigmaf}, we have
\begin{align}\label{v2}
2\nu\int_0^1\int_0^1\beta_0(s,t)\,\tilde\Gamma(s,t;\nu)\,ds\,dt\stackrel{d.}{=}2\sigma_{f}\,\nu\,\BB(\nu)\, ,
\end{align}
and following the arguments in the proof of Theorem~\ref{thm:2.1} it can be shown that the process $\{\hat\G_{f}(\nu)\}_{\nu\in[\nu_0,1]}$ is tight. 
The proof is therefore complete in view of \eqref{v1} and \eqref{v2}.


\newpage

\section{Auxiliary lemmas}\label{app:aux:lem}

\begin{lemma}\label{lem:s4}
Under Assumptions~\ref{a1} and \ref{a201}, there exists a constant $c>0$ such that, for any $x\in L^2([0,1])$,  $\|\tau_{\lambda}(x)\|_K^2\leq c\,\lambda^{-(2a+1)/(2D)}\|x\|_{L^2}^2$ and $\E\|\tau_{\lambda}(x)\|_K^2\leq c\,\lambda^{-1/(2D)}$.
\end{lemma}
\begin{proof}
Recall from \eqref{tau} that $\tau_{\lambda}(x)=\sum_{k=1}^\infty(1+\lambda\rho_{k})^{-1}\phi_{k}\l x,\phi_{k}\r_{L^2}$. Therefore,
\begin{align*}
\|\tau_{\lambda}(x)\|_K^2&=\bigg\l\sum_{k=1}^\infty\frac{\l x,\phi_{k}\r_{L^2}}{1+\lambda\rho_{k}}\phi_{k},\sum_{k'=1}^\infty\frac{\l x,\phi_{k'}\r_{L^2}}{1+\lambda\rho_{k'}}\phi_{k'}\bigg\r_K=\sum_{k=1}^\infty\frac{\l x,\phi_k\r_{L^2}^2}{1+\lambda\rho_k}\\
&\leq \|x\|_{L^2}^2\sum_{k=1}^\infty\frac{\|\phi_k\|_{L^2}^2}{1+\lambda\rho_k}\leq c\, \lambda^{-(2a+1)/(2D)}\|x\|_{L^2}^2\,.
\end{align*}
In addition, by Assumption~\ref{a201}, $\E\big(\l X,\phi_k\r_{L^2}^2\big)=1$, so that
\begin{align*}
\E\|\tau_{\lambda}(X)\|_K^2=\sum_{k=1}^\infty\frac{\E\big(\l X,\phi_k\r_{L^2}^2\big)}{1+\lambda\rho_k}=\sum_{k=1}^\infty\frac{1}{1+\lambda\rho_k}\leq c\, \lambda^{-1/(2D)}\,.
\end{align*}

\end{proof}

\begin{lemma}[Lemma~3.1 in \citealp{shang2015}]\label{lem:3.1}
Under Assumptions~\ref{a1} and \ref{a201}, there exists a constant $c_K>0$ such that for any $\beta\in\H$, $\|\beta\|_{L^2}^2\leq c_K\lambda^{-(2a+1)/(2D)}\|\beta\|_K^2$.
\end{lemma}
Recall from \eqref{cxe} the long-run covariance function of $\{ X_i\e_i\}_{i\in\mathbb Z}$ given by
\begin{align*}
C_{X\e}(s,t)=\sum_{\ell=-\infty}^{+\infty}\cov\{\e_0X_0(s),\e_\ell X_\ell(t)\}\,.
\end{align*}
Lemmas~\ref{lem:cxe} and \ref{lem:cxe2} below shows that $C_{X\e}\in L^2([0,1]^2)$ and $\int_0^1 C_{X\e}(t,t)dt<\infty$.

\begin{lemma}\label{lem:cxe}
Under Assumption~\ref{a:m}, we have $C_{X\e}\in L^2([0,1]^2)$.
\end{lemma}

\begin{proof}
Note that by Assumption~\ref{a:m}, for $\ell\geq1$, $\e_{0,\ell}$ and $\e_{-\ell}$ are independent; $ X_{0,\ell}$ and $ X_{-\ell}$ are independent. Since $\E(X)\equiv0$, we find that, for $\ell\geq1$,
\begin{align*}
\E\{X_0(s)\,X_\ell(t)\}&=\E\Big[\big\{X_{0}(s)-X_{0,\ell}(s)\big\}\,X_{-\ell}(t)\Big]+\E\big\{X_{0,\ell}(s)\,X_{-\ell}(t)\big\}\\
&=\E\Big[\big\{X_{0}(s)-X_{0,\ell}(s)\big\}\,X_{-\ell}(t)\Big]\,.
\end{align*}
We combine the above equation and \eqref{p1} and deduce that, for $\ell\geq1$,
\begin{align}\label{ss}
&\cov\big\{\e_{0}\,  X_{0}(s)\,,\e_{-\ell} \, X_{-\ell}(t)\big\}=\E\big\{(\e_{0}-\e_{0,\ell})\e_{-\ell}\big\}\,\E\Big[\big\{ X_{0}(s)- X_{0,\ell}(s)\big\}\times X_{-\ell}(t)\Big]\,.
\end{align}
Therefore, we find from the above equation that
\begin{align}\label{intcue}
&\int_0^1\int_0^1\{C_{X\e}(s,t)\}^2dsdt\notag\\
&= \int_0^1\int_0^1\bigg[\cov\big\{\e_0\,  X_0(s)\,,\e_0 \, X_0(t)\big\}+2\sum_{\ell=1}^{+\infty}\cov\big\{\e_0\,  X_0(s)\,,\e_{-\ell} \, X_{-\ell}(t)\big\}\bigg]^2dsdt\notag\\
&=\int_0^1\int_0^1\bigg(\E(\e_0^2)\,\E\big\{ X_0(s)\times X_0(t)\big\}\notag\\
&\hspace{2cm}+2\sum_{\ell=1}^{+\infty}\E\big\{(\e_{0}-\e_{0,\ell})\e_{-\ell}\big\}\,\E\Big[\big\{ X_{0}(s)- X_{0,\ell}(s)\big\}\times X_{-\ell}(t)\Big]\bigg)^2dsdt\notag\\
&\leq2I_1+8I_2\,,
\end{align}
where
\begin{align*}
&I_1= \{\E(\e_0^2)\}^2\int_0^1\int_0^1\Big[\E\big\{ X_0(s)\times X_0(t)\big\}\Big]^2dsdt\,,\\
&I_2= \int_0^1\int_0^1\Bigg(\sum_{\ell=1}^{+\infty}\E\big\{(\e_{0}-\e_{0,\ell})\e_{-\ell}\big\}\,\E\Big[\big\{ X_{0}(s)- X_{0,\ell}(s)\big\}\times X_{-\ell}(t)\Big]\Bigg)^2dsdt\,.
\end{align*}
For the first term $I_1$, we have $I_1=\{\E(\e_0^2)\}^2\times\|C_X\|_{L^2}^2<\infty$. For the second term $I_2$, by the Cauchy-Schwarz inequality, we find
\begin{align}\label{p2}
&\Bigg(\sum_{\ell=1}^{+\infty}\E\big\{(\e_{0}-\e_{0,\ell})\e_{-\ell}\big\}\,\E\Big[\big\{ X_{0}(s)- X_{0,\ell}(s)\big\}\times X_{-\ell}(t)\Big]\Bigg)^2\notag\\
&\leq \Bigg(\sum_{\ell=1}^{+\infty}\E\big\{(\e_{0}-\e_{0,\ell})\e_{-\ell}\big\}\,\Big[\E\big\{ X_{0}(s)- X_{0,\ell}(s)\big\}^2\Big]^{1/2}\times\Big[\E\{ X_{0}(t)\}^2\Big]^{1/2}\Bigg)^2\notag\\
&=\Bigg(\sum_{\ell=1}^{+\infty}\E\big\{(\e_{0}-\e_{0,\ell})\e_{-\ell}\big\}\,\Big[\E\big\{ X_{0}(s)- X_{0,\ell}(s)\big\}^2\Big]^{1/2}\Bigg)^2\times\E\{ X_{0}(t)\}^2\notag\\
&\leq\Bigg[\sum_{\ell=1}^{+\infty}\E\big\{(\e_{0}-\e_{0,\ell})\e_{-\ell}\big\}^2\Bigg]\times\Bigg[\sum_{\ell=1}^{+\infty}\E\big\{ X_{0}(s)- X_{0,\ell}(s)\big\}^2\Bigg]\times\E\{ X_{0}(t)\}^2\notag\\
&\leq\Bigg\{\sum_{\ell=1}^{+\infty}\E(\e_{0}-\e_{0,\ell})^2\times\E(\e_{0}^2)\Bigg\}\times\Bigg[\sum_{\ell=1}^{+\infty}\E\big\{ X_{0}(s)- X_{0,\ell}(s)\big\}^2\Bigg]\times\E\{ X_{0}(t)\}^2\notag\\
&=\E(\e_0^2)\times\Bigg(\sum_{\ell=1}^{+\infty}\E|\e_{0}-\e_{0,\ell}|^2\Bigg)\times\Bigg[\sum_{\ell=1}^{+\infty}\E\big\{ X_{0}(s)- X_{0,\ell}(s)\big\}^2\Bigg]\times\E\{ X_{0}(t)\}^2\,.
\end{align}
Therefore, by Assumption~\ref{a:m}, we deduce from the above equation that
\begin{align*}
I_2&\leq \,\E(\e_0)^2\times\Bigg(\sum_{\ell=1}^{+\infty}\E|\e_{0}-\e_{0,\ell}|^2\Bigg)\\
&\qquad\times\int_0^1\Bigg[\sum_{\ell=1}^{+\infty}\E\big\{ X_{0}(s)- X_{0,\ell}(s)\big\}^2\Bigg]ds\times\int_0^1\E\{ X_{0}(t)\}^2dt\\
&= \,\E(\e_0)^2\times\Bigg(\sum_{\ell=1}^{+\infty}\E|\e_{0}-\e_{0,\ell}|^2\Bigg)\times\Bigg(\sum_{\ell=1}^{+\infty}\E\|X_{0}-X_{0,\ell}\|_{L^2}^2\Bigg)\times\E\| X_0\|_{L^2}^2<\infty\,.
\end{align*}
In conclusion, we deduce from \eqref{intcue} that $\int_0^1\int_0^1 \{C_{X\e}(s,t)\}^2dsdt<\infty$, which completes the proof.

\end{proof}

\begin{lemma}\label{lem:cxe2}
Under Assumption~\ref{a:m}, we have $\int_0^1 C_{X\e}(s,s)ds<\infty$.
\end{lemma}

\begin{proof}
By the arguments similar to the ones used to obtain \eqref{ss}, it follows that, for $\ell\geq1$,
\begin{align*}
\cov\big\{\e_{0}\,  X_{0}(s)\,,\e_{-\ell} \, X_{-\ell}(s)\big\}&=\E\big\{(\e_{0}-\e_{0,\ell})\e_{-\ell}\big\}\,\E\Big[\big\{ X_{0}(s)- X_{0,\ell}(s)\big\}\times X_{-\ell}(s)\Big]\,.
\end{align*}
Therefore, following the proof of Lemma~\ref{lem:cxe}, we deduce that
\begin{align*}
&\int_0^1C_{X\e}(s,s)ds= \int_0^1\bigg[\cov\big\{\e_0\,  X_0(s)\,,\e_0 \, X_0(s)\big\}+2\sum_{\ell=1}^{+\infty}\cov\big\{\e_0\,  X_0(s)\,,\e_{-\ell} \, X_{-\ell}(s)\big\}\bigg]ds\notag\\
&= \int_0^1\bigg(\E(\e_0^2)\,\E\big\{ X_0(s)\big\}^2+2\sum_{\ell=1}^{+\infty}\E\big\{(\e_{0}-\e_{0,\ell})\e_{-\ell}\big\}\,\E\Big[\big\{ X_{0}(s)- X_{0,\ell}(s)\big\}  X_{-\ell}(s)\Big]\bigg)ds\notag\\
&=\E(\e_0^2)\,\E\|X_0\|_{L^2}^2+2\int_0^1\bigg(\sum_{\ell=1}^{+\infty}\E\big\{(\e_{0}-\e_{0,\ell})\e_{-\ell}\big\}\,\E\Big[\big\{ X_{0}(s)- X_{0,\ell}(s)\big\}  X_{-\ell}(s)\Big]\bigg)ds\notag\\
&\leq\E(\e_0^2)\,\E\|X_0\|_{L^2}^2\\
&\quad+2\int_0^1\Bigg(\E(\e_0^2)\times\Bigg(\sum_{\ell=1}^{+\infty}\E|\e_{0}-\e_{0,\ell}|^2\Bigg)\times\Bigg[\sum_{\ell=1}^{+\infty}\E\big\{ X_{0}(s)- X_{0,\ell}(s)\big\}^2\Bigg]\times\E\{ X_{0}(s)\}^2\Bigg)^{1/2}ds\,,
\end{align*}
where in the last step we applied \eqref{p2} by taking $t=s$. Therefore, by the Cauchy-Schwarz inequality and Assumption~\ref{a:m}, we conclude from the above equation that
\begin{align*}
&\int_0^1C_{X\e}(s,s)ds\leq\E(\e_0^2)\,\E\|X_0\|_{L^2}^2\\
&\quad+2\{\E(\e_0)^2\}^{1/2} \Bigg(\sum_{\ell=1}^{+\infty}\E|\e_{0}-\e_{0,\ell}|^2\Bigg)^{1/2}\int_0^1\Bigg(\Bigg[\sum_{\ell=1}^{+\infty}\E\big\{ X_{0}(s)- X_{0,\ell}(s)\big\}^2\Bigg]\times\E\{ X_{0}(s)\}^2\Bigg)^{1/2}ds\\
&\leq \E(\e_0^2)\,\E\|X_0\|_{L^2}^2+2\{\E(\e_0)^2\}^{1/2} \bigg(\sum_{\ell=1}^{+\infty}\E|\e_{0}-\e_{0,\ell}|^2\bigg)^{1/2}\big(\E\| X_0\|_{L^2}^2\big)^{1/2}\bigg(\sum_{\ell=1}^{+\infty}\E\|X_{0}-X_{0,\ell}\|_{L^2}^2\bigg)^{1/2}\\
&<\infty\,,
\end{align*}
which concludes the proof.

\end{proof}


Lemma~\ref{lem:hnbeta} below is used to prove Theorem~\ref{thm:bahadur}. For its statement recall 
the definition of $H_{n,k}$ in  \eqref{hnk2}.
%
%

\begin{lemma}\label{lem:hnbeta}
For $p_n\geq 1$, let $\mathcal F_{p_n}=\{\beta\in\H:\Vert\beta\Vert_{L^2}\leq 1,J(\beta,\beta)\leq p_n\}$. Then, under Assumptions~\ref{a1}--\ref{a:subg}, as $n\to\infty$,
\begin{align*}
&\max_{1\leq k\leq n}\,\sup_{\beta\in\mathcal{F}_{p_n}}\,\frac{\Vert H_{n,k}(\beta)\Vert_K}{p_n^{1/(2m)}\Vert\beta\Vert_{L^2}^{(m-1)/m}+n^{-1/2}}=O_p\big(\lambda^{-1/(2D)}\log n\big)^{1/2}\,.
\end{align*}

\end{lemma}

\begin{proof}

The proof of Lemma~\ref{lem:hnbeta} follows a modified argument of the proof of Lemma~3.4 in \cite{shang2015}. For any $x\in L^2([0,1])$, let $g(x,\beta)=\tau_{\lambda}(x)\int_0^1\beta(s)x(s)ds$.
We have
\begin{align*}
H_{n,k}(\beta_1)-H_{n,k}(\beta_2)&=\frac{1}{\sqrt n}\sum_{i=1}^{j}\Big(g(X_i,\beta_1-\beta_2)\times\one\{\EE_n(c)\}-\E\big[g(X_i,\beta_1-\beta_2)\times\one\{\EE_n(c)\}\big]\Big)\,.
%
\end{align*}
Note that, on the event $\EE_n(c)$ defined in \eqref{enc},
\begin{align*}
\Big|\l\beta_1-\beta_2,X_i\r_{L^2}\,\one\{\EE_n(c)\}-\E\big[\l\beta_1-\beta_2,X_i\r_{L^2}\,\one\{\EE_n(c)\}\big]\Big|\leq c\,\|\beta_1-\beta_2\|_{L^2}\,.
\end{align*}
Hence, we deduce that
\begin{align*}
&\Big\|g(X_i,\beta_1-\beta_2)\times\one\{\EE_n(c)\}-\E\big[g(X_i,\beta_1-\beta_2)\times\one\{\EE_n(c)\}\big]\Big\|_K^2\\
%
%
&\leq 2(\log n)^2\times\|\beta_1-\beta_2\|^2_{L^2}\times\|\tau_{\lambda}(X_i)\|_K^2\,.
\end{align*}
For $1\leq k\leq n$, let $W_{n,k}^2=n^{-1}\sum_{i=1}^k\|\tau_{\lambda}(X_i)\|_K^2$ and $\mathcal X_n=\{\|\tau_{\lambda}(X_i)\|_K\}_{i=1}^n$. 
By Lemma~\ref{lem:s4}, $\E(W_{n,k}^2)\leq(k/n)\E\|\tau_{\lambda}(X_i)\|_K^2\leq c\,\lambda^{-1/(2D)}$.
By Theorem 3.5 in \cite{pinelis1994}, for any $1\leq j\leq n$, for any $\beta_1,\beta_2\in\H$ and for $1\leq j\leq n$,
\begin{align}
&\P\left\{\Vert H_{n,k}(\beta_1)-H_{n,k}(\beta_2)\Vert_K\geq x\,\big|\mathcal{X}_n \right\}\notag\\
%
%
&=\P\bigg\{\bigg\|\frac{1}{\sqrt k}\sum_{i=1}^{k}\Big[g(X_i,\beta_1-\beta_2)-\E\big\{g(X_i,\beta_1-\beta_2)\big\}\Big]\bigg\|_K\geq \sqrt{n/k}\,x\,\big|\mathcal{X}_n \bigg\}\notag\\
&\leq2\exp\left(-\frac{nk^{-1}x^2}{2k^{-1}\sum_{i=1}^k\|\tau_{\lambda}(X_i)\|_K^2\,\Vert\beta_1-\beta_2\Vert_{L^2}^2}\right)\leq 2\exp\left(-\frac{x^2}{2\,W_{n,k}^2\,\Vert\beta_1-\beta_2\Vert_{L^2}^2}\right)\,.
\label{p}
\end{align}
%
Following the proof of Lemma~3.4 in \cite{shang2015} (see p.~13 of \citealp{shang2015b}), we deduce that, 
%
for any $1\leq k\leq n$,
\begin{align*}
\P\left\{\sup_{\beta\in\mathcal{F}_{p_n},\,\Vert\beta\Vert_{L^2}\leq\delta}\Vert H_{n,k}(\beta)\Vert_K\geq x\,\big|\mathcal{X}_n\right\}\leq 2\exp\big(-c_1^{-2}W_{n,k}^{-2}\, p_n^{-1/(2m)}\delta^{-2+1/m}x^2\big)\,.
\end{align*}
Taking $\gamma=1-1/(2m)$, $b_n=\sqrt{n}\,p_n^{1/(4m)}$, $\theta_n=b_n^{-1}$, $Q_n=\lfloor-\log_2\theta_n+\gamma-1\rfloor$ and $T_n=c_2(\lambda^{-1/(2D)}\log n)^{1/2}$, for some constant $c_2>0$ to be specified below, yields that
\begin{align*} 
&\P\left\{\max_{1\leq k\leq n}\,\sup_{\beta\in\mathcal{F}_{p_n},\,\Vert\beta\Vert_{L^2}\leq2}\,\frac{\sqrt{n}\Vert H_{n,k}(\beta)\Vert_K}{b_n\Vert\beta\Vert_{L^2}^{\gamma}+1}\geq T_n\,\big|\mathcal{X}_n\right\}\notag\\
&\leq\sum_{k=1}^n\P\left\{\sup_{\beta\in\mathcal{F}_{p_n},\,\Vert\beta\Vert_{L^2}\leq\theta_n^{1/\gamma}}\sqrt{n}\Vert H_{n,k}(\beta)\Vert_K\geq T_n\,\big|\mathcal{X}_n\right\}\notag\\
&\hspace{1cm}+\sum_{k=1}^n\sum_{j=0}^{Q_n}\P\left\{\sup_{\beta\in\mathcal{F}_{p_n},\,(\theta_n2^{j})^{1/\gamma}\leq\Vert\beta\Vert_{L^2}\leq(\theta_n2^{j+1})^{1/\gamma}}\,\frac{\sqrt{n}\Vert H_{n,k}(\beta)\Vert_K}{b_n\Vert\beta\Vert_{L^2}^{\gamma}+1}\geq T_n\,\big|\mathcal{X}_n\right\}\notag\\
&\leq\sum_{k=1}^n\P\left\{\sup_{\beta\in\mathcal{F}_{p_n},\,\Vert\beta\Vert_{L^2}\leq\theta_n^{1/\gamma}}\sqrt{n}\Vert H_{n,k}(\beta)\Vert_K\geq T_n\,\big|\mathcal{X}_n\right\}\notag\\
&\hspace{1cm}+\sum_{k=1}^n\sum_{j=0}^{Q_n}\P\left\{\sup_{\beta\in\mathcal{F}_{p_n},\,\Vert\beta\Vert_{L^2}\leq(\theta_n2^{j+1})^{1/\gamma}}\,\sqrt{n}\Vert H_{n,k}(\beta)\Vert_K\geq (b_n\theta_n2^j+1)T_n\,\big|\mathcal{X}_n\right\}\notag\\
&\leq 2\sum_{k=1}^n\exp\big(-c_1^{-2}W_{n,k}^{-2}\, p_n^{-1/(2m)}\theta_n^{(-1+1/m)/\gamma}n^{-1}T_n^2\big)\notag\\
&\hspace{1cm}+2\sum_{k=1}^n\sum_{j=0}^{Q_n}\exp\Big\{-c_1^{-2}W_{n,k}^{-2}\, p_n^{-1/m}(\theta_n2^{j+1})^{(-1+1/m)/\gamma}(b_n\theta_n2^j+1)^2n^{-1}T_n^2\Big\}\notag\\
&\leq2\sum_{k=1}^n\exp\big(-c_1^{-2}W_{n,k}^{-2}\,T_n^2\big)+2(Q_n+1)\sum_{k=1}^n\exp\big(-c_1^{-2}W_{n,k}^{-2}\,T_n^2/4\big)\notag\\
&\leq 2(Q_n+2)\sum_{k=1}^n\exp\big(-c_1^{-2}W_{n,k}^{-2}\,T_n^2/4\big)\leq 2(Q_n+2)\exp\big(\log n-c_1^{-2}W_{n,n}^{-2}\,T_n^2/4\big)\,.
\end{align*}
Denote the event $
\mathcal{A}_n=\{W_{n,n}^2\leq c_3 \lambda^{-1/(2D)}\}$ for some constant $c_3>0$. Since $\E(W_{n,n}^2)\leq c\lambda^{-1/(2D)}$, we have that, for $c_3$ large enough, $\P(\mathcal{A}_n)\to1$ as $n\to\infty$. On the event $\mathcal{A}_n$, by taking $c_2>2c_1c_3^{-1/2}$, as $n\to\infty$,
\begin{align*}
2(Q_n+2)\exp\big(\log n-c_1^{-2}W_{n,n}^{-2}\,T_n^2/4\big)\leq 2(Q_n+2)\exp\big\{\log n-c_1^{-2}c_2^2c_3\log n/4\big\}=o(1)\,,
\end{align*}
which together with \eqref{p} completes the proof.

\end{proof}

\baselineskip=13pt
{\centering

}


\begin{thebibliography}{99}









\bibitem[Berger and Delampady, 1987]{berger1987}Berger, J. O. and Delampady, M. (1987). Testing precise hypotheses. {\em Statist. Sci.}, {\bf 2}, 317--335.

\bibitem[Berkes et al., 2013]{berkes2013}Berkes, I., Horváth, L. and Rice, G. (2013). Weak invariance principles for sums of dependent random functions. \emph{Stochastic Process. Appl.}, \textbf{123}, 385--403.



\bibitem[Cai and Yuan, 2012]{caiyuan2012}Cai, T. T. and Yuan, M. (2012). Minimax and adaptive prediction for functional linear regression. \emph{J. Amer. Statist. Assoc.}, \textbf{107}, 1201--1216.

\bibitem[Cardot et al., 1999]{cardot1999}Cardot, H., Ferraty, F. and Sarda, P. (1999). Functional linear model. {\em Stat. Probab. Lett.}, {\bf45}, 11--22.

\bibitem[Cardot et al., 2003]{cardot2003}Cardot, H., Ferraty, F., Mas, A. and Sarda, P. (2003). Testing hypotheses in the functional linear model. {\em 	Scand. J. Stat.}, \textbf{30}, 241--255.

\bibitem[Chow and Liu, 1992]{chow1992}Chow, S. C. and Liu, P. J. (1992). {\em Design and Analysis of Bioavailability and Bioequivalence Studies}.
Marcel Dekker, New York.



\bibitem[Dehling~et~al., 2002]{dehling2002}Dehling, H., T. Mikosch, and M. S{\o}rensen (2002). {\em Empirical process techniques for dependent
data}. Birkh\"auser.





\bibitem[Dette et al., 2020a]{detteaos2020}Dette, H., Kokot, K. and Aue, A. (2020a). Functional data analysis in the Banach space of continuous functions. \emph{Ann. Stat.}, \textbf{48}, 1168--1192.



\bibitem[Dette et al., 2020b]{dettejrssb2020}Dette, H., Kokot, K. and Volgushev, S. (2020b). Testing relevant hypotheses in functional time series via self‐normalization. \emph{J. R. Stat. Soc. Series. B. Stat. Methodol.}, \textbf{82}, 629--660.



\bibitem[Dette and Tang, 2021]{dettetang}Dette, H and Tang, J. (2021) Statistical inference for function-on-function linear regression. \emph{arXiv preprint}. arXiv:2109.13603.

\bibitem[\protect\citeauthoryear{Ferraty and Vieu}{Ferraty and
  Vieu}{2010}]{FerratyVieu2010}
Ferraty, F. and P.~Vieu (2010).
\newblock {\em Nonparametric {F}unctional {D}ata {A}nalysis}.
\newblock Springer-Verlag, New York.

\bibitem[Fogarty and Small, 2014]{fogarty2014}Fogarty, C. B. and Small, D. S. (2014). Equivalence testing for functional data with an application to comparing pulmonary function devices. {\em Ann. Appl. Stat.}, {\bf 8}, 2002--2026.


\bibitem[Garc{\'\i}a-Portugu\'es et al., 2014]{garcia2014}Garc{\'\i}a-Portugu\'es, E., Gonz\'alez-Manteiga, W. and Febrero-Bande, M. (2014). A goodness-of-fit test for the functional linear model with scalar response. {\em J. Comput. Graph. Stat.}, \textbf{23}, 761--778.

\bibitem[Golub et al., 1979]{golub}Golub, G. H., Heath, M. and Wahba, G. (1979). Generalized cross-validation as a method for choosing a good ridge parameter. {\em Technometrics}, \textbf{21}, 215--223.



\bibitem[Hall and Horowitz, 2007]{hall2007} Hall, P. and Horowitz, J. L. (2007). Methodology and convergence rates for functional linear regression. \emph{Ann. Stat.}, \textbf{35}, 70--91.

\bibitem[Hao et al., 2021]{hao2021}Hao, M., Liu, K. Y., Xu, W. and Zhao, X. (2021). Semiparametric inference for the functional Cox model. \emph{J. Amer. Statist. Assoc.} \textbf{116}, 1319--1329.


\bibitem[Hilgert et al., 2013]{hilgert2013}Hilgert, N., Mas, A. and Verzelen, N. (2013). Minimax adaptive tests for the functional linear model. \emph{Ann. Stat.}, {\bf 41}, 838--869.

\bibitem[Hodges and Lehmann, 1954]{hodges1954}Hodges J. L. and Lehmann, E. L. (1954). Testing the approximate validity of statistical hypotheses. \emph{J. R. Stat. Soc. Series. B. Stat. Methodol.}, \textbf{16}, 261--268.

\bibitem[H\"ormann and Kokoszka, 2010]{hormann2010}H\"ormann, S. and Kokoszka, P. (2010). Weakly dependent functional data. {\em 	Ann. Stat.}, \textbf{38}, 1845--1884.

\bibitem[Horv\'ath and Kokoszka, 2012]{horvath2012}Horv\'ath, L. and Kokoszka, P. (2012). {\em Inference for functional data with applications}. Springer Science \& Business Media.

\bibitem[\protect\citeauthoryear{Hsing and Eubank}{Hsing and
  Eubank}{2015}]{hsingeubank2015}
Hsing, T. and Eubank, R. (2015).
\newblock {\em Theoretical Foundations of Functional Data Analysis, with an
  Introduction to linear Operators}.
\newblock New York: Wiley.








\bibitem[Kong et al., 2016]{kong2016}Kong, D., Staicu, A. M. and Maity, A. (2016). Classical testing in functional linear models. {\em J. Nonparametr. Stat.}, \textbf{28}, 813--838.



\bibitem[Kutta~et~al., 2021]{kutta2021}Kutta, T., Dierickx, G. and Dette, H. (2021) Statistical inference for the slope parameter in functional linear regression. \emph{arXiv preprint}. arXiv:2108.07098.




\bibitem[Lei, 2014]{lei2014}Lei, J. (2014). Adaptive global testing for functional linear models. \emph{J. Amer. Statist. Assoc.}, \textbf{109}, 624--634.




\bibitem[Lobato, 2001]{lobato2001}Lobato, I. N. (2001). Testing that a dependent process is uncorrelated. \emph{J. Amer. Statist. Assoc.}, {\bf96}, 1066--1076.





\bibitem[M\"uller and Stadtm\"uller, 2005]{muller2005}M\"uller, H. G. and Stadtm\"uller, U. (2005). Generalized functional linear models. \emph{Ann. Stat.}, \textbf{33}, 774--805.


\bibitem[P\"otscher and Prucha, 1997]{potscher}P\"otscher, B. M. and Prucha, I. (1997). {\em Dynamic nonlinear econometric models: Asymptotic theory}. Springer Science \& Business Media.

\bibitem[Qu and Wang, 2017]{qu2017}Qu, S. and Wang, X. (2017). Optimal global test for functional regression. \emph{arXiv preprint} arXiv:1710.02269.

\bibitem[Ramsay and Silverman, 2005]{ramsay2005}Ramsay, J. O. and Silverman, B. W. (2005). {\em Functional data analysis}. New
York: Springer.


\bibitem[Shang and Cheng, 2015]{shang2015}Shang, Z. and Cheng, G. (2015). Nonparametric inference in generalized functional linear models. \emph{Ann. Stat.}, \textbf{43}, 1742--1773.

\bibitem[Shao, 2010]{shao2010}Shao, X. (2010). A self-normalized approach to confidence interval construction in time series.  \emph{J. R. Stat. Soc. Series. B. Stat. Methodol.}, \textbf{72}, 343--366.


\bibitem[Shao and Zhang, 2010]{shaozhang2010}Shao, X. and Zhang, X. (2010). Testing for change points in time series. \emph{J. Amer. Statist. Assoc.}, \textbf{105}, 1228--1240.

\bibitem[Su et al., 2017]{su2017}Su, Y. R., Di, C. Z. and Hsu, L. (2017). Hypothesis testing in functional linear models. {\em Biometrics}, \textbf{73}, 551--561.

\bibitem[Sun et al., 2018]{sun2018} Sun, X., Du, P., Wang, X. and Ma, P. (2018). Optimal penalized function-on-function regression under a reproducing kernel Hilbert space framework. \emph{J. Amer. Statist. Assoc.}, \textbf{113}, 1601--1611.


\bibitem[Tekbudak et al., 2019]{Tekbudak} Tekbudak, M. Y., Alfaro-C\'ordoba, M., Maity, A. and Staicu, A. M. (2019). A comparison of testing methods in scalar-on-function regression. \emph{AStA Adv. Stat. Anal.}, \textbf{103}, 411--436.




\bibitem[Wahba, 1990]{whaba1990}Wahba, G. (1990). {\em Spline models for observational data}. Society For Industrial and Applied Mathematics.

\bibitem[Wang et al., 2016]{wang2016}Wang, J. L., Chiou, J. M. and Müller, H. G. (2016). Functional data analysis. {\em 	Annu. Rev. Stat. Appl.}, \textbf{3}, 257-295.

\bibitem[Wellek, 2010]{wellek2010}Wellek, S. (2010). {\em Testing Statistical Hypotheses of Equivalence and Noninferiority}. CRC Press, Boca
Raton, second edition.

\bibitem[Wellner, 2003]{wellner2003}Wellner, J. A. (2003). Gaussian white noise models: some results for monotone functions. \emph{Institute of Mathematical Statistics Lecture Notes--Monograph Series}, 87--104.

\bibitem[Wood, 2003]{wood2003}Wood, S. N. (2003). Thin plate regression splines. \emph{J. R. Stat. Soc. Series. B. Stat. Methodol.}, \textbf{65}, 95--114.

\bibitem[Yao et al., 2005]{yao2005}Yao, F., M\"uller, H. G. and Wang, J. L. (2005). Functional linear regression analysis for longitudinal data. \emph{Ann. Stat.}, \textbf{33}, 2873--2903.

\bibitem[Yuan and Cai, 2010]{yuancai}Yuan, M. and Cai, T. T. (2010). A reproducing kernel Hilbert space approach to functional linear regression. \emph{Ann. Stat.}, \textbf{38}, 3412--3444.

\bibitem[Zhang and Lavitas, 2018]{zhang2018}Zhang, T. and Lavitas, L. (2018). Unsupervised self-normalized change-point testing for time series.
\emph{J. Amer. Statist. Assoc.}, \textbf{113}, 637--648.


\bibitem[Zhang and Shao, 2015]{zhang2015}Zhang, X. and Shao, X. (2015). Two sample inference for the second-order property of temporally
dependent functional data. {\em Bernoulli}, \textbf{21}, 909--929.

\bibitem[Zhang et al., 2011]{zhang2011}Zhang, X., Shao, X., Hayhoe, K. and Wuebbles, D. J. (2011). Testing the structural stability of
temporally dependent functional observations and application to climate projections. {\em Electron. J.
Statist.}, \textbf{5}, 1765--1796.

\end{thebibliography}

\begin{thebibliography}{}










\bibitem[Pinelis, 1994]{pinelis1994}Pinelis, I. (1994). Optimum bounds for the distributions of martingales in Banach spaces. \emph{Ann. Probab.}, \textbf{22}, 1679--1706.


\bibitem[Shang and Cheng, 2015b]{shang2015b}Shang, Z. and Cheng, G. (2015b) Supplement to ``Nonparametric inference in generalized functional linear models". DOI:10.1214/15-AOS1322SUPP.

\bibitem[van der Vaart and Wellner, 1996]{vaart1996}Van der Vaart, A. W. and Wellner, J. A. (1996). Weak Convergence and Empirical Processes. \emph{Springer}, New York.

%
%

\end{thebibliography}
\end{document}